\documentclass[11pt]{amsart}

\usepackage{latexsym}
\usepackage{amssymb}
\usepackage{amsmath}
\usepackage{color}

\newtheorem{theorem}{Theorem}[section]
\newtheorem{lemma}[theorem]{Lemma}
\newtheorem{proposition}[theorem]{Proposition}
\newtheorem{corollary}[theorem]{Corollary}
\newtheorem{definition}[theorem]{Definition}

\newtheorem{remark}[theorem]{Remark}

\newcommand\id{\mathop{\rm id}}
\newcommand\tr{\mathop{\rm tr}}
\newcommand\Tr{\mathop{\rm Tr}}
\newcommand\dd{\mathop{\rm d}}
\newcommand\mm{\mathop{\rm m}}

\newcommand\nph{\varphi}
\newcommand\nep{\epsilon}

\newcommand{\cl}[1]{\mathcal{#1}}
\newcommand{\bb}[1]{\mathbb{#1}}

\begin{document}

\title{Quantum no-signalling bicorrelations}

\author[M. Brannan]{Michael Brannan} 
\address{Department of Pure Mathematics and Institute for Quantum Computing, University of Waterloo, 200 University Ave West, Waterloo, Ontario, Canada, N2L 3G1
}
\email{michael.brannan@uwaterloo.ca}

\author[S. J. Harris]{Samuel J. Harris} 
\address{Department of Mathematics and Statistics, Northern Arizona University, 801 S. Osborne Dr., Flagstaff, AZ 86011, USA
}
\email{samuel.harris@nau.edu}

\author[I. G. Todorov]{Ivan G. Todorov}
\address{
School of Mathematical Sciences, University of Delaware, 501 Ewing Hall,
Newark, DE 19716, USA}

\email{todorov@udel.edu}

\author[L. Turowska]{Lyudmila Turowska}
\address{Department of Mathematical Sciences, Chalmers University
of Technology and the University of Gothenburg, Gothenburg SE-412 96, Sweden}
\email{turowska@chalmers.se}

\date{1 February 2023}

\begin{abstract}
We introduce classical and quantum no-signalling bicorrelations
and characterise the different types thereof 
in terms of states on operator system tensor products, exhibiting connections with bistochastic operator matrices and with 
dilations of quantum magic squares.  
We define concurrent bicorrelations as a quantum input-output generalisation of 
bisynchronous correlations.  
We show that concurrent bicorrelations of quantum commuting type 
correspond to tracial states on the universal C*-algebra of the projective free unitary quantum group, 
showing that in  the quantum input-output setup, quantum permutations of finite sets must be replaced by quantum automorphisms of matrix algebras.  We apply our results to study the quantum graph isomorphism game,  describing the game C*-algebra in this case, and make precise 
connections with the algebraic notions of quantum graph isomorphism, existing presently in the literature.
\end{abstract}

\maketitle

\tableofcontents

%%%%%%%%%%%%%%%%%%%%%%%%%%%%%%%%%%%%%%%%%%%%%%%%%
%%%%%%%%%%%%%%%%%%%%%%%%%%%%%%%%%%%%%%%%%%%%%%%%%

\section{Introduction}

In recent years, many fruitful interactions have emerged between {\it entanglement} and \emph{non-locality} in quantum systems, on one hand, and the theory of {\it operator algebras} and {\it operator systems}, on the other. 
At a high level, this connection stems from
the laws of quantum mechanics, 
which dictate that the input-output behaviour of local measurements on (bipartite) quantum systems is encoded by non-commutative operator algebras of observables and their state spaces.  This provides powerful means to translate between questions of a physical nature 
and questions formulated in the language of non-commutative analysis.  
At the base of these developments lie the work of Junge, Navascues, Palazuelos, Perez-Garcia, Scholz and Werner
\cite{jnpp}, where the relation 
between the Tsirelson Problem in quantum physics and the Connes Embedding Problem in 
operator algebra theory was first noticed (see also \cite{ozawa}), and that of 
Paulsen, Severini, Stahlke, Winter and the third author \cite{psstw}, where the notion of 
synchronous no-signalling correlation was first defined and characterised. 
The fruitfulness of these connections has been borne out by many recent works; see  
\cite{ozawa, lmprsstw, mps, lmr, mr, maro, amrssv, mrv, bcehpsw} for an  incomplete list.  
We specifically single out
Sloftsra's ground-breaking work  \cite{slofstra_JAMS, slofstra}, which injected ideas from geometric group theory into the theory of non-local games, showing that the set of bipartite quantum correlations is not closed, 
and the work of Helton, Meyer, Paulsen and Satriano \cite{hmps}, in which
an algebraic approach to non-local games was formulated.
All of these ideas recently culminated in the resolution of the 
weak Tsirelson problem and Connes Embedding problem in  the preprint 
\cite{jnvwy} by Ji, Natarajan, Vidick, Wright and Yuen.

In the present work, we are primarily interested in investigating the structure of {\it quantum input-quantum output} bipartite correlations which generalise the {\it bisynchronous}
correlations introduced by Paulsen and Rahaman in \cite{pr}.  
Recall that 
a no-signalling bipartite correlation over the quadruple $(X, X, A, A)$, 
where $X$ and $A$ are finite sets, 
is a family of conditional probability distributions
\[p = \{p(a,b|x,y) : (x,y) \in X\times X,  (a,b) \in A\times A\}\] 
that has well-defined marginals (see e.g. \cite{lmprsstw}).  
Operationally, in the commuting operator model of quantum mechanics, 
$p$ describes the input-output behaviour of a bipartite quantum system, 
given by a Hilbert space $H$ in state $\xi$, interpreted as a unit vector in $H$, 
on which local measurements are jointly performed: 
for each $x,y \in X$, two non-communicating parties Alice and Bob have access to 
mutually commuting local measurement systems 
$E_x = (E_{x,a})_{a \in A} \subseteq \cl B(H)$ (for Alice) and $F_y = (F_{y,b})_{b \in A} \subseteq \cl B(H)$ (for Bob).
Given input $x$, Alice uses the system $E_x$ to measure $\xi$, and similarly, given $y$, 
Bob uses $F_y$ to measure $\xi$; the resulting outcomes of Alice and Bob's measurements are 
$(a,b) \in A\times A$ with probability  
\[p(a,b|x,y) = \langle E_{x,a}F_{y,b}\xi,\xi\rangle.\]  
We say that a correlation $p$ is {\it synchronous} if $p(a,b|x,x) = 0$ for all $x \in X$ and $a \ne b$.  
Heuristically, 
Alice and Bob's behaviour is synchronised in 
that they appear to invoke the same ``virtual function'' $X \to A$ to 
obtain their outputs, depending on the given inputs.  
A correlation $p$ is called {\it bisynchronous} \cite{pr} 
if it is synchronous and has the additional property that $p(a,a|x,y) = 0$ for all $a\in A$ and
$x \ne y$.  In this case, the ``virtual function'' $X \to A$ behaves as though it were in addition injective. 

Using the language of operator algebras and non-commutative geometry, one can make the
intuition, highlighted in the previous paragraph, precise. 
Let $\cl A_{X,A} = \star_{|X|}\ell^\infty(A)$ be the unital free product of $|X|$ copies of the 
$|A|$-dimensional abelian C*-algebra $\ell^\infty(A)$. 
The C*-algebra $\cl A_{X,A}$ is a C*-cover of the universal operator system $\cl S_{X,A}$ 
with generators
$e_{x,a}$, where $x \in X$ and $a \in A$, subject to the relations $e_{x,a} = e_{x,a}^2 = e_{x,a}^*$ and 
$\sum_{a\in A} e_{x,a} = 1$, $x\in X$.
Within the framework of non-commutative geometry, $\cl A_{X,A}$ can 
be regarded as a quantisation of
the finite-dimensional C$^\ast$-algebra $C(\cl F(X,A))$ of complex-valued functions on the set $\cl F(X,A)$ 
of functions $f : X \to A$. 
It was shown in \cite{psstw} that a no-signalling 
correlation $p$ of quantum commuting type is synchronous if and only if 
there is a tracial state $\tau$ on $\cl A_{X,A}$ such that 
\begin{align} \label{corr-trace}
p(a,b|x,y) = \tau(e_{x,a}e_{y,b}), \ \ \  x,y \in X, \ a,b \in A.
\end{align} 
If the correlation $p$ is bisynchronous (and $|X| = |A|$), 
then \cite{pr} $p$ arises via (\ref{corr-trace}) 
from a tracial state $\tau$ on 
the C*-algebra $C(S_X^+)$ of the quantum permutation group \cite{wang}.  
Similarly to $\cl  A_{X,A}$, the C*-algebra 
$C(S_X^+)$ is the universal unital C*-algebra with generators $e_{x,a}$, $x,a \in X$, 
further satisfying the additional relations 
$\sum_{x \in X} e_{x,a} = 1$, $a \in A$.  
Note that $C(S_X^+)$ is a free analogue of the algebra $C(S_X)$ of complex functions on the permutation group $S_X$ of $X$, and is itself a C$^\ast$-algebraic quantum group \cite{wang}.   

Bisynchronous correlations arise in 
the analysis of certain classes of non-local games, most notably the graph isomorphism game 
\cite{amrssv,lmr,maro, bcehpsw} and 
the related metric isometry game \cite{eifler}.  Here, deep and unexpected connections emerged between 
quantum permutation groups, 
no-signalling correlations and graph theory. At the same time, connections were established 
between graph isomorphism games and 
{\it quantum graphs} \cite{mrv, mrv2, bcehpsw}.  
In particular, in the aforementioned works, a natural (operator) algebraic notion of a quantum isomorphism between quantum graphs was introduced.

One of the main motivations behind 
the present work is the desire 
to provide an operational characterisation of quantum isomorphisms between quantum graphs in terms of bipartite correlations. As the term 
suggests, the description of a quantum graph (in any of its many guises \cite{stahlke, mrv, bcehpsw, bhtt}) 
requires a suitable quantum version of 
the notion of a vertex or edge, using the language of bipartite quantum systems.  
Hence one is naturally led to consider bipartite no-signalling correlations which allow {\it quantum} states as inputs and outputs.  

Quantum input-quantum output no-signalling (QNS) correlations were introduced by Duan and Winter \cite{dw}, and subsequently systematically studied in \cite{tt-QNS,bks, bhtt}.  
Given finite sets $X$ and $A$, and denoting by $M_X$ (resp. $M_A$) the 
full matrix algebra over the $|X|$-dimensional Hilbert space, 
a QNS correlation over the quadruple $(X,X,A,A)$ is a quantum channel 
\[
\Gamma: M_X \otimes M_X \to M_A \otimes M_A
\]
satisfying a pair of additional constraints, equivalent to the existence of marginal channels
(see equations \eqref{eq_qns1} and \eqref{eq_qns2}, and the article \cite{dw} for further details).  
Since any classical no-signalling correlation $p$ over $(X,X,A,A)$ can be regarded as 
a QNS correlation $\Gamma_p$ that preserves the corresponding diagonal subalgebras,
QNS correlations constitute a genuine generalisation of their classical counterparts 
(see also equation \eqref{eq_Gammap}).

The main purpose of the present 
work is to develop a notion, and find (operational and operator algebraic) characterisations, 
of bisynchronicity in the quantum input-output setting.    
In parallel with the classical setting, here 
we focus our attention on the case where the input and output systems are of the same size, 
that is, $|A| = |X|$.  In this case, it is natural to consider \lq\lq bistochastic'' 
correlations $\Gamma:M_X \otimes M_X \to M_A \otimes M_A$, 
that is, unital QNS correlations with 
the additional property that the dual channels $\Gamma^*$ are also QNS correlations;
these channels are referred to as QNS {\it bicorrelations} (see Definition \ref{d_bic}). 
A quantisation of bisynchronicity must involve a suitable
quantum counterpart of the property of sending identical inputs to identical outputs.  
In bipartite quantum systems, this is naturally captured by how $\Gamma$ (and $\Gamma^*$) 
acts on the canonical maximally entangled state.  More precisely, if $(\epsilon_{x,y})_{x,y \in X}$ is the canonical matrix unit system 
of $M_X$, and $J_X = \frac{1}{|X|} \sum_{x,y \in X} \epsilon_{x,y} \otimes \epsilon_{x,y}$ is the maximally entangled state, then it is natural to impose the condition 
\begin{align} \label{eq-concurrent}
\Gamma(J_X) = J_A.
\end{align}
Condition \eqref{eq-concurrent} on a QNS correlation $\Gamma$ was introduced and studied in detail in our previous work \cite{bhtt},
where it was called {\it concurrency}.   For a QNC bicorrelation $\Gamma$, its concurrency 
is equivalent to concurrency for $\Gamma^*$ (see Remark \ref{r_biconc}). From an operational viewpoint,  concurrent bicorrelations $\Gamma$ are characterised by the property that 
$\Gamma$ and $\Gamma^*$ preserve the 
perfect correlation of local measurements in both directions:  the input state $J_X$ is characterised by the property that local measurements performed on $J_X$ in any fixed basis are always perfectly correlated with uniformly random outcomes.  Concurrent bicorrelations thus respect this perfect correlative  structure, and hence rightfully can be interpreted as fully quantum versions of bisynchronous correlations.       

We study the various types of QNS bicorrelations  
(quantum commuting, quantum approximate, quantum and local) in detail, 
providing operator system/algebra characterisations thereof.  
After providing necessary preliminaries in Section \ref{s_dsom}, in Section \ref{s_bom}
we exhibit {\it operator bistochastic matrices}, which 
can be viewed as quantum and operator-valued generalisations of classical bistochastic matrices.
Operator bistochastic matrices turn out to be 
the suitable mathematical objects encoding each of the parties of a 
QNS bicorrelation.    
We characterise concretely  the universal operator system $\cl T_{X}$ 
of an operator bistochastic matrix 
as the subspace spanned by natural order two products associated with the entries of a 
universal block operator {\it bi-isometry} $V: \bb C^{|X|} \otimes H \to \bb C^{|X|} \otimes K$
(that is, an isometry $V$ for which the transpose $V^t$ is also an isometry).  
We further identify the dual operator system of $\cl T_X$ and 
establish several properties of $\cl T_{X}$ and its
universal C*-algebra $\cl C_X$.  At the heart of our arguments is a factorisation result 
for bistochastic operator matrices (Theorem \ref{p_coor}). 
Our results should be compared to those of \cite{tt-QNS}, where a similar development was 
undertaken for the universal operator system $\cl T_{X,A}$ of a block operator 
isometry, and the corresponding C*-algebra $\cl C_{X,A}$.

The diagonal expectations (intuitively, the classical components) 
of bistochastic operator matrices 
coincide with {\it quantum magic squares}, introduced by De Las Cuevas, Drescher and Netzer in \cite{dlcdn}; 
contrapositively, bistochastic operator matrices can be viewed as quantum versions of quantum magic squares. 
In Section \ref{ss_cbom}, we build up on this connection and rephrase some of the results of \cite{dlcdn}
in the language of operator systems. 
Indeed, one of the main results in \cite{dlcdn} is the fact that not every quantum magic square admits 
a dilation to a quantum permutation. 
In Theorem \ref{p_dilate}, we characterise the 
dilatability of a quantum magic square in terms of the complete positivity of 
natural maps, associated with the given quantum magic square, and defined on 
the operator system 
$\cl P_X \subseteq C(S_X^+)$ spanned by the coefficients of a quantum permutation matrix. 
We demonstrate that the non-dilatability of quantum magic squares is due to the distinction between 
different operator system structures. 

In Section \ref{s_rbc}, we introduce the types of quantum no-signalling bicorrelations, 
corresponding to different physical models (local, quantum, approximately quantum, 
quantum commuting and general no-signalling), 
and characterise them in terms of states on the various operator system structures, 
with which the algebraic tensor product $\cl T_X \otimes \cl T_X$ can be endowed.  
Here we rely on the tensor product theory developed in \cite{kptt}. 
We pay a separate attention to \emph{classical} no-signalling bicorrelations, showing that
their corresponding encoding operator system $\cl S_X$ is the universal operator system 
spanned by the entries of an  $X \times X$-quantum magic square studied in Section \ref{ss_cbom}, 
and obtaining similar characterisations in terms of states on operator system tensor products on the 
algebraic tensor product ${\cl S}_X \otimes {\cl S_X}$.

In Section \ref{s_concbic}, we focus our attention on concurrent bicorrelations, 
establishing in Theorem \ref{QNSbicorrelation} a characterisation of concurrent quantum commuting 
bicorrelations in terms of tracial states.  
We show that the C*-algebra, whose tracial states are of interest here, 
is the C*-algebra  $C(\bb P\cl U_X^+)$ 
of functions on {\it projective free unitary quantum group}.  
Recall that the C*-algebra of the free unitary quantum group $C(\cl U_X^+)$ is the universal unital C*-algebra generated by the entries $u_{x,a}$ of an $X \times X$ bi-unitary matrix $U = (u_{x,a})_{x,a}$.
 
The C*-algebra $C(\bb P\cl U_{X^+})$ is the C*-subalgebra of $C(\cl U_X^+)$, 
generated length two words of the form $u_{x,a}^*u_{x',a'}$.  
Note that the C*-algebra $C(\cl U^+_X)$ is the free  analogue of $C(\cl U_X)$, 
the C*-algebra of continuous complex functions  on the unitary group $\cl U_X$.  Similarly, $C(\bb P \cl U_X^+)$ is the free analogue of the algebra of continuous complex functions on the projective unitary group $\bb P \cl U_X = \cl U_{X}/\bb T$.  Recall that the natural action of $\cl U_X$ on $M_X$ by conjugation induces an isomorphism 
of $\bb P \cl U_X$ and the group $\text{Aut}(M_X)$ 
of $\ast$-automorphisms of the matrix algebra $M_X$.  In this way, $C(\bb P\cl U_X)$ can be regarded as the quantum version of the automorphism group of $M_X$.  
In fact, using quantum group theory, this reasoning can be made precise as, by \cite[Corollary 4.1]{banica} and \cite[Theorem 1]{banica0}, 
$C(\mathbb P \cl U_X^+)$ {\it is} the quantum automorphism group of the tracial 
C*-algebra $M_X$ in the sense of Wang \cite{wang}. 

Thus, from an operator algebraic point of view,
Theorem \ref{QNSbicorrelation} provides yet another justification for our definition of concurrent bicorrelations as the appropriate 
quantum versions of bisynchronous correlations; 
indeed, at a correlation level, 
quantisation of bisynchronicity amounts to replacing classical channels on $\cl D_X \otimes \cl D_X$ with quantum channels on $M_X \otimes M_X$.  At the level of tracial states encoding these channels, 
Theorem \ref{QNSbicorrelation} shows that this quantisation amounts to replacing  $C(S_X^+)$ 
(that is, quantum automorphisms of $\cl D_X$) with $C(\bb P \cl U_X^+)$ 
(that is, quantum automorphisms of $M_X$).  
We remark here that the C$^\ast$-algebras $C(S_X^+)$ and $C(\mathbb P \cl U_X^+)$ are indeed distinct C$^\ast$-algebras, as can be seen from the K-theory computations in \cite[Theorem 4.5]{voigt}.
In summary, the operational and the algebraic notions of quantisation are in agreement.      
Our results complement a series of operator characterisations in the literature,
part of which we summarise in the following table:

\vspace{0.1cm} 

\begin{center}
{\footnotesize
\begin{tabular}{ |l|l| } 
 \hline 
 %\vspace{0.1cm}
  {\bf Correlation type:} & {\bf Encoded by states on:}  
%\vspace{0.1cm} 
\\
\hline
 Classical NS correlations $\cl C_{\rm ns}$  & $\cl S_{X,A} \otimes_{\max} \cl S_{X,A}$ \cite[Theorem 3.1]{lmprsstw} \\ 
 \hline 
 Classical qc-correlations $\cl C_{\rm qc}$ & $\cl S_{X,A} \otimes_{\rm c} \cl S_{X,A}$ 
 \cite[Theorem 3.1]{lmprsstw}   \\ 
 \hline
 Classical qa-correlations $\cl C_{\rm qa}$  & $\cl S_{X,A} \otimes_{\min} \cl S_{X,A}$ \cite[Theorem 3.1]{lmprsstw} \\ 
 \hline 
Synchronous qc-correlations $\cl C_{\rm qc}^{\rm s}$  & $\cl A_{X,A}$ (tracial) \cite[Theorem 5.5]{psstw}  \\ \hline 
Bisynchronous qc-correlations $\cl C_{\rm qc}^{\rm bis}$  & $C(S_X^+)$ (tracial) \cite[Theorem 2.2]{pr} \\ 
\hline
QNS correlations $\cl Q_{\rm ns}$ & $\cl T_{X,A} \otimes_{\max} \cl T_{X,A}$ \cite[Theorem 6.2]{tt-QNS} \\ \hline
QNS qc-correlations $\cl Q_{\rm qc}$  & $\cl T_{X,A} \otimes_{c} \cl T_{X,A}$  
\cite[Theorem 6.3]{tt-QNS} \\ \hline
QNS qa-correlations $\cl Q_{\rm qa}$  & $\cl T_{X,A} \otimes_{\min} \cl T_{X,A}$ 
\cite[Theorem 6.5]{tt-QNS} \\
\hline
QNS bicorrelations $\cl Q_{\rm ns}^{\rm bi}$ & $\cl T_{X} \otimes_{\max} \cl T_{X}$ [Theorem \ref{th_bicmax}] \\ \hline
QNS qc-bicorrelations $\cl Q_{\rm qc}^{\rm bi}$   & $\cl T_{X} \otimes_{c} \cl T_{X}$ 
[Theorem \ref{th_bicqc}]\\ \hline
QNS qa-bicorrelations $\cl Q_{\rm qa}^{\rm bi}$  & $\cl T_{X} \otimes_{\min} \cl T_{X}$ 
[Theorem \ref{th_bicqa}]\\ \hline
Classical NS bicorrelations $\cl C_{\rm ns}^{\rm bi}$  & $\cl S_{X} \otimes_{\max} \cl S_{X}$  [Theorem \ref{th_classSX}]\\ \hline
Classical qc-bicorrelations $\cl C_{\rm qc}^{\rm bi}$  & $\cl S_{X} \otimes_{\rm c} \cl S_{X}$  [Theorem \ref{th_classSX}]\\ \hline
Classical qa-bicorrelations $\cl C_{\rm qa}^{\rm bi}$  & $\cl S_{X} \otimes_{\rm min} \cl S_{X}$  [Theorem \ref{th_classSX}]\\ \hline
Concurrent qc-correlations $\cl Q_{\rm qc}^{\rm c}$ & $\cl C_{X,A}$ (tracial) \cite[Theorem 4.1]{bhtt}\\ \hline  
Concurrent qc-bicorrelations $\cl Q_{\rm qc}^{\rm bic}$ & $C(\mathbb P\cl U_{X}^+)$ (tracial) 
[Theorem \ref{QNSbicorrelation}] \\ \hline  
\end{tabular}
}
\end{center}

\vspace{0.1cm}

In Section  \ref{s_qgig}, 
we apply concurrent bicorrelations to study quantum graph isomorphisms.
We consider quantum graphs  
with respect to $M_X$, viewed as symmetric skew subspaces 
$\cl U \subseteq \bb C^X \otimes \bb C^X$ \cite{btw,stahlke, dsw,tt-QNS,bhtt}.  
We define quantum isomorphisms between quantum graphs in terms of perfect 
QNS strategies for a suitable quantum graph isomorphism game, 
building up on the approach to quantum graph homomorphisms followed in \cite{tt-QNS}. 
In Theorem \ref{iso}, we characterise quantum commuting isomorphisms between quantum graphs $\cl U,\cl V \subseteq \bb C^{X} \otimes \bb C^X$ in terms of the existence of a bi-unitary matrix $U = (u_{x,a})_{x,a} \in M_X(\cl B(H))$ such that 
$C(\bb P\cl U_{X^+})$ admits a tracial state $\tau$, and 
\begin{align} \label{eq_iso}
U(\tilde{ \cl S}_{\cl U} \otimes 1)U^* \subseteq \tilde{\cl S}_{\cl V} \otimes \cl B(H) 
\ \mbox{ and } \ U^{\rm t}(\tilde{ \cl S}_{\cl V} \otimes 1) U^{{\rm t} *} \subseteq \tilde{\cl S}_{\cl U} \otimes \cl B(H),
\end{align}
where $\tilde{\cl S}_\cl U$ and $\tilde{\cl S}_\cl V$ are the traceless, symmetric subspaces,
canonically associated to $\cl U$ and $\cl V$, respectively.  
Note that condition \eqref{eq_iso} 
is a quantum counterpart of the 
characterisation \cite{amrssv} 
of quantum isomorphisms of classical graphs in terms of quantum permutations matrices
that intertwine the relevant adjacency matrices,  
through the replacement of 
quantum permutations by bi-unitaries (see  Remark \ref{r_magic_vs_bi}). 
We further formalise the relations
\eqref{eq_iso} in Theorem \ref{th_alg}, 
where we introduce a natural game algebra $\cl A_{P,Q}$ whose tracial states encode the perfect
quantum commuting strategies for the $(\cl U,\cl V)$-isomorphism game.  We note, in particular, that when $\cl U = \cl V$, the algebra 
$\cl A_{P,Q}$ admits the structure of a compact quantum group, which seems to generalise the quantum automorphism group of a classical graph.  We leave the study of these quantum groups for future work.  

Finally, in Section \ref{s_connections}, we compare the 
operational notion of quantum graph isomorphism of Section \ref{s_qgig}
to the operator algebraic notions that have appeared previously in the literature,
and which have been based mainly on adjacency matrices 
\cite{mrv, mrv2, bcehpsw, daws}.  
We show, in Theorem \ref{th_qcalgqc},
that the algebraic quantum isomorphisms considered in the aforementioned works 
fit into our framework as special cases.  
The arguments and ideas for the proof of this theorem rely on the recent work of Daws on quantum graphs \cite{daws}.
In Theorem \ref{converse}, we establish a partial converse, exhibiting the precise 
conditions, under which the algebraic and the operational notions of quantum graph isomorphism 
coincide.

\subsection*{Acknowledgements}
M.B. was partially supported by an NSERC discovery grant.  S.H. was partially supported by an NSERC Postdoctoral Fellowship. I.T. was
partially supported by NSF grant DMS-2154459 and a Simons Foundation grant (award number 708084). \\

%%%%%%%%%%%%%%%%%%%%%%%%%%%%%%%%%%%%%%%%%%%%%%%%%
%%%%%%%%%%%%%%%%%%%%%%%%%%%%%%%%%%%%%%%%%%%%%%%%%

\section{Preliminaries}\label{s_dsom}

In this section, we collect basic preliminaries on quantum no-signalling correlations, 
set notation and introduce terminology.
Let $H$ be a Hilbert space. 
As usual, we denote by $\cl B(H)$ the space of all bounded linear operators on $H$
and sometimes write $\cl L(H)$ if $H$ is finite dimensional.
We denote by $I_H$ the identity operator on $H$ and, 
if $\xi,\eta \in H$, we let $\xi\eta^*$ be the rank one operator given by
$(\xi\eta^*)(\zeta) = \langle\zeta,\eta\rangle \xi$.
In addition to inner products, $\langle\cdot,\cdot\rangle$ will denote 
the duality between a vector space and its dual.
We let $\cl B(H)^+$ be the cone of positive operators in $\cl B(H)$, and
further denote by 
$\cl T(H)$ its ideal of trace class operators and by $\Tr$ -- the trace functional on $\cl T(H)$.

An \emph{operator system} is a self-adjoint subspace $\cl S\subseteq \cl B(H)$, for some Hilbert space $H$, 
containing $I_H$. 
If $\cl S$ is an operator system, the 
\emph{universal C*-cover} of $\cl S$ \cite{kw}
is a pair $(C_u^*(\cl S),\iota)$, where 
$C_u^*(\cl S)$ is a unital C*-algebra and $\iota : \cl S\to C_u^*(\cl S)$ is a 
unital complete order embedding, such that 
$\iota(\cl S)$ generates $C_u^*(\cl S)$ as a C*-algebra 
and, whenever $K$ is a Hilbert space and 
$\phi : \cl S\to \cl B(K)$ is a unital completely positive map, there exists a *-representation 
$\pi_{\phi} : C^*_u(\cl S) \to \cl B(K)$ such that $\pi_{\phi} \circ \iota = \phi$. 
If $\cl S$ is a finite dimensional operator system then its Banach space dual $\cl S^{\rm d}$ can be viewed as an 
operator system \cite[Corollary 4.5]{CE2}. We refer the reader to \cite{Pa} for information and background on 
operator systems and completely positive maps.

We denote by $|X|$ the cardinality of a finite set $X$,
let $H^X = \oplus_{x\in X}H$ and write
$M_X$ for the space of all complex matrices of size $|X|\times |X|$; 
we identify $M_X$ with $\cl L(\bb{C}^X)$ and 
set $I_X = I_{\bb{C}^{X}}$. 
For $n\in \bb{N}$, we let $[n] = \{1,\dots,n\}$ and $M_n = M_{[n]}$. 
We write $(e_x)_{x\in X}$ for the canonical orthonormal basis of $\bb{C}^{X}$,
$(\epsilon_{x,x'})_{x,x'\in X}$ for the canonical matrix unit system in $M_X$, and 
denote by $\cl D_X$ the subalgebra of $M_X$ of all diagonal matrices with respect to the
basis $(e_x)_{x\in X}$.
If $\cl V$ is a vector space, we write $M_X(\cl V)$ for the space of all $X\times X$ matrices with entries in $\cl V$; 
we note that there is a canonical linear 
identification between $M_X(\cl V)$ and $M_X\otimes \cl V$. 
Here, and in the sequel, we use the symbol $\otimes$ to denote the algebraic tensor product of 
vector spaces.

For an element $\omega\in M_X$, we denote by $\omega^{\rm t}$ the
transpose of $\omega$ in the canonical basis, and
write $\overline{\omega}$ for the complex conjugate of $\omega$; thus, $\overline{\omega} = (\omega^{\rm t})^*$. 
The canonical complete order isomorphism
from $M_X$ onto its dual operator system $M_X^{\rm d}$ maps
an element $\omega\in M_X$ to the linear functional $f_{\omega} : M_X\to \bb{C}$ given by 
$f_{\omega}(T) = \Tr(T\omega^{\rm t})$; see e.g. \cite[Theorem 6.2]{ptt}. 
We will thus consider $M_X$ as self-dual with the pairing 
\begin{equation}\label{eq_rdu}
(\rho,\omega) \to \langle \rho,\omega\rangle := \Tr(\rho\omega^{\rm t}).
\end{equation}
On the other hand, note that the Banach space predual $\cl B(H)_*$ can be canonically identified with 
$\cl T(H)$; every normal functional $\phi : \cl B(H) \to \bb{C}$ thus 
corresponds to a (unique) operator $S_{\phi} \in \cl T(H)$ such that 
$\phi(T) = \Tr(TS_{\phi})$, $T\in \cl B(H)$.
In the case where $X$ is a fixed finite set (which will sometimes come in the form of a Cartesian product),
we will use a mixture of the two dualities just discussed: if $\omega,\rho\in M_X$, $S\in \cl T(H)$ and $T\in \cl B(H)$, 
it will be convenient to continue writing
$$\langle \rho\otimes T, \omega\otimes S\rangle = \Tr(\rho\omega^{\rm t}) \Tr(TS).$$

If $X$ and $Y$ are finite sets, we identify $M_X\otimes M_Y$ with $M_{X\times Y}$ and
write $M_{XY}$ in its place. Similarly, we set $\cl D_{XY} = \cl D_X\otimes \cl D_Y$.
For an element $\omega_X\in M_X$ and a Hilbert space $H$, we let 
$L_{\omega_X} : M_X\otimes \cl B(H)\to \cl B(H)$ be the 
linear map given by 
$L_{\omega_X}(S\otimes T) = \langle S,\omega_X\rangle T$. 
If $H = \bb{C}^Y$ and $\omega_Y\in M_Y$, we thus have linear maps
$L_{\omega_X} : M_{XY}\to M_Y$ and $L_{\omega_Y} : M_{XY}\to M_X$; note that  
$$\langle L_{\omega_X}(R),\rho_Y\rangle = \langle R, \omega_X\otimes \rho_Y\rangle, \ \ R\in M_{XY}, \rho_Y\in M_Y,$$
and a similar formula holds for $L_{\omega_Y}$.
We let $\Tr_X : M_{XY}\to M_Y$ (resp. $\Tr_Y : M_{XY}\to M_X$) be the partial trace; thus, 
$\Tr_X = L_{I_X}$ (resp. $\Tr_Y = L_{I_Y}$).

Let $X$, $Y$, $A$ and $B$ be finite sets. 
A \emph{quantum channel} from $M_X$ into $M_A$ is a completely positive trace preserving map
$\Phi : M_X\to M_A$.
A \emph{quantum correlation over $(X,Y,A,B)$} 
(or simply a \emph{quantum correlation} if the sets are understood from the context) 
is a quantum channel $\Gamma : M_{XY}\to M_{AB}$.
Such a $\Gamma$ is called a
\emph{quantum no-signalling (QNS) correlation} \cite{dw} if
\begin{equation}\label{eq_qns1}
\Tr\hspace{-0.07cm}\mbox{}_A\Gamma(\rho_X\otimes \rho_Y) = 0 \ \mbox{ whenever } \Tr(\rho_X) = 0
\end{equation}
and
\begin{equation}\label{eq_qns2}
\Tr\hspace{-0.07cm}\mbox{}_B\Gamma(\rho_X\otimes \rho_Y) = 0 \ \mbox{ whenever } \Tr(\rho_Y) = 0.
\end{equation}
We denote by $\cl Q_{\rm ns}$ the set of all QNS correlations.

A \emph{stochastic operator matrix} over $(X,A)$,
acting on a Hilbert space $H$, is a positive block operator matrix $\tilde{E} = (E_{x,x',a,a'})_{x,x',a,a'}\in M_{XA}(\cl B(H))$ such that $\Tr_A \tilde{E} = I$. 
 
A QNS correlation $\Gamma : M_{XY}\to M_{AB}$ is \emph{quantum commuting} if 
there exist a Hilbert space $H$, a unit vector $\xi\in H$ and
stochastic operator matrices $\tilde{E} = (E_{x,x',a,a'})_{x,x',a,a'}$ and $\tilde{F} = (F_{y,y',b,b'})_{y,y',b,b'}$ 
on $H$ such that 
$$E_{x,x',a,a'} F_{y,y',b,b'} = F_{y,y',b,b'}E_{x,x',a,a'}$$
for all $x,x'\in X$, $y,y'\in Y$, $a,a'\in A$, $b,b'\in B$,
and 
\begin{equation}\label{eq_EFp}
\Gamma(\epsilon_{x,x'} \otimes \epsilon_{y,y'}) = \sum_{a,a'\in A} \sum_{b,b'\in B}
\left\langle E_{x,x',a,a'}F_{y,y',b,b'}\xi,\xi \right\rangle \epsilon_{a,a'} \otimes \epsilon_{b,b'}, 
\end{equation}
for all $x,x' \in X$ and all $y,y' \in Y$.
\emph{Quantum} QNS correlations are defined as in (\ref{eq_EFp}), but 
requiring that $H$ has the form $H_A\otimes H_B$, for some finite dimensional Hilbert spaces $H_A$ and $H_B$, and 
$E_{x,x',a,a'} = \tilde{E}_{x,x',a,a'} \otimes I_B$ and 
$F_{y,y',b,b'} = I_A \otimes \tilde{F}_{y,y',b,b'}$, for some stochastic operator matrices 
$(\tilde{E}_{x,x',a,a'})$ and 
$(\tilde{F}_{y,y',b,b'})$, acting on $H_A$ and $H_B$, respectively. 
\emph{Approximately quantum} QNS correlations are the 
limits of quantum QNS correlations, while \emph{local} QNS correlations are 
the convex combinations of the form 
$\Gamma = \sum_{i=1}^k \lambda_i \Phi_i \otimes \Psi_i$, 
where $\Phi_i : M_X\to M_A$ and $\Psi_i : M_Y\to M_B$ are quantum channels, $i = 1,\dots,k$.

We write $\cl Q_{\rm qc}$ (resp. $\cl Q_{\rm qa}$, $\cl Q_{\rm q}$, $\cl Q_{\rm loc}$) for the (convex) set of all quantum commuting (resp. approximately quantum, quantum, local) QNS correlations, and note the 
inclusions 
$$\cl Q_{\rm loc}\subseteq \cl Q_{\rm q}\subseteq \cl Q_{\rm qa}\subseteq \cl Q_{\rm qc}\subseteq \cl Q_{\rm ns}.$$

Recall that a (classical) no-signalling (NS) correlation is a family $p = \{(p(a,b|x,y))_{a,b} : (x,y)\in X\times Y\}$
of probability distributions over $A\times B$, such that 
$$\sum_{b\in B} p(a,b|x,y) = \sum_{b\in B} p(a,b|x,y'), \ \ x\in X, y,y'\in Y, a\in A,$$
and 
$$\sum_{a\in A} p(a,b|x,y) = \sum_{a\in A} p(a,b|x',y), \ \ x,x'\in X,  y\in Y, b\in B$$
(see e.g. \cite{lmprsstw, psstw}). 
We denote the (convex) set of all NS correlations by $\cl C_{\rm ns}$.
With a correlation $p \in \cl C_{\rm ns}$, we associate the classical information channel 
$\Gamma_p : \cl D_{XY}\to \cl D_{AB}$, given by 
\begin{equation}\label{eq_Gammap}
\Gamma_p(\epsilon_{x,x}\otimes\epsilon_{y,y}) = \sum_{a\in A}\sum_{b\in B} p(a,b|x,y)
\epsilon_{a,a}\otimes\epsilon_{b,b}.
\end{equation}
The subclasses $\cl C_{\rm t}$ of $\cl C_{\rm ns}$, 
for ${\rm t}\in \{{\rm loc}, {\rm q}, {\rm qa},{\rm qc}\}$, 
are defined as in the previous 
paragraph, but using \emph{classical} stochastic operator matrices, that is, stochastic operator matrices 
of the form $E = \sum_{x\in X}\sum_{a\in A} \epsilon_{x,x}\otimes\epsilon_{a,a}\otimes E_{x,a}$. 
Note that the condition for $E$ being stochastic is equivalent to the requirement that 
$(E_{x,a})_{a\in A}$ is a positive operator-valued measure (POVM) for all $x\in X$. 
We note the inclusions
$$\cl C_{\rm loc}\subseteq \cl C_{\rm q}\subseteq \cl C_{\rm qa}\subseteq \cl C_{\rm qc}\subseteq \cl C_{\rm ns},$$
all of which are strict:
$\cl C_{\rm loc} \neq \cl C_{\rm q}$ is the Bell Theorem \cite{bell},  
$\cl C_{\rm q} \neq \cl C_{\rm qa}$ is a negative answer to the weak Tsirelson Problem
\cite{slofstra} (see also \cite{dpp, slofstra_JAMS}), 
and $\cl C_{\rm qa} \neq \cl C_{\rm qc}$ -- in view of \cite{fritz, jnpp, ozawa}, 
a negative answer to the announced solution of the
Connes Embedding Problem \cite{jnvwy}.

%%%%%%%%%%%%%%%%%%%%%%%%%%%%%%%%%%%%%%%%%%%%%%%%%
%%%%%%%%%%%%%%%%%%%%%%%%%%%%%%%%%%%%%%%%%%%%%%%%%

\section{Bistochastic operator matrices}\label{s_bom}

In this section we define and examine bistochastic operator matrices, 
which constitute a specialisation of stochastic operator matrices \cite[Section 3]{tt-QNS}
to the new context to be considered herein. 
Let $X$ be a finite set, and set $A = X$. The distinct symbols $X$ and $A$ will continue to be used 
to indicate the variable with respect to which a partial trace is taken;
the symbol $X$ usually refers to the domain of a quantum channel, while $A$ -- to its codomain.

\begin{definition}\label{d_bom}
Let $H$ be a Hilbert space. A block operator matrix
$E = \left(E_{x,x',a,a'}\right)_{x,x',a,a'} \in (M_{XA}\otimes \cl B(H))^+$ is called 
a \emph{bistochastic operator matrix} if 
$$\Tr\hspace{-0.05cm}\mbox{}_A E = I_X\otimes I_H 
\ \mbox{ and } \ 
\Tr\hspace{-0.05cm}\mbox{}_X E = I_A\otimes I_H.$$
\end{definition}

%%%%%%%%%%%%%%%%%%%%%%%%%%%%%%%%%%%%%%%%%%%%%%%%%
%%%%%%%%%%%%%%%%%%%%%%%%%%%%%%%%%%%%%%%%%%%%%%%%%

\subsection{Factorisation}\label{ss_fact}

A block operator matrix $V = (V_{a,x})_{a,x\in X}$, where $V_{a,x}\in \cl B(H,K)$ for some Hilbert spaces $H$ and $K$, 
will be called a \emph{bi-isometry} if 
$V$ and $V^{\rm t} := (V_{x,a})_{a,x\in X}$ are isometries as operators in $\cl B(H^X,K^X)$.

\begin{theorem}\label{p_coor}
Let $H$ be a Hilbert space and $E \in (M_{XA}\otimes \cl B(H))^+$.
The following are equivalent:
\begin{itemize}
\item[(i)] $E$ is a bistochastic operator matrix;

\item[(ii)] there exist a Hilbert space $K$ and operators $V_{a,x}\in \cl B(H,K)$, $x, a \in X$, 
such that $(V_{a,x})_{a,x\in X}$ is a bi-isometry and
\begin{equation}\label{eq_up}
E_{x,x',a,a'} = V_{a,x}^* V_{a',x'}, \ \ \ x,x',a,a'\in X.
\end{equation}
\end{itemize}
\end{theorem}

\begin{proof}
(ii)$\Rightarrow$(i)
Since $V$ is an isometry, 
$$\sum_{a\in X} E_{x,x',a,a} = \sum_{a\in X} V_{a,x}^* V_{a,x'} = \delta_{x,x'} I_H,$$
and hence $\Tr\mbox{}_A E = I_X\otimes I_H$.
Since $V^{\rm t} = (V_{a,x})_{x,a}$ is an isometry, 
$$\sum_{x\in X} E_{x,x,a,a'} = \sum_{x\in X} V_{a,x}^* V_{a',x} = \delta_{a,a'} I_H,$$
and hence $\Tr\mbox{}_X E = I_A\otimes I_H$.

(i)$\Rightarrow$(ii)
Suppose that $E = (E_{x,x',a,a'})_{x,x',a,a'}$ is a bistochastic operator matrix acting on $H$ and set 
$E_{a,a'} = (E_{x,x',a,a'})_{x,x'}$, $a,a' \in A$; thus, $E_{a,a'}\in M_X\otimes \cl B(H)$. 
Let $\Phi : M_A\to M_X\otimes \cl B(H)$ be the linear map, given by 
$\Phi(\nep_{a,a'}) = E_{a,a'}$, $a,a'\in A$. 
By Choi's Theorem, $\Phi$ is a unital completely positive map and, by Stinespring's Theorem, 
there exist a Hilbert space $\tilde{K}$,
an isometry $V : \bb{C}^X\otimes H\to \tilde{K}$ and
a unital *-homomorphism $\pi : M_A\to \cl B(\tilde{K})$ such that
$\Phi(T) = V^* \pi(T) V$, $T\in M_A$.
Up to unitary equivalence, $\tilde{K} = \bb{C}^A\otimes K$ for some Hilbert space $K$
and $\pi(T) = T\otimes I_K$, $T\in M_A$. 
Write $V_{a,x} : H\to K$, $a\in A$, $x\in X$, for the entries of $V$, 
when $V$ is considered as a block operator matrix.
As in \cite[Theorem 3.1]{tt-QNS}, we conclude that $E_{x,x',a,a'} = V_{a,x}^* V_{a',x'}$, $x,x'\in X$, $a,a'\in A$. 

Note that 
$$\left(\Tr\hspace{-0.07cm}\mbox{}_X\circ\Phi\right)\left(\nep_{a,a'}\right) 
= \Tr\hspace{-0.07cm}\mbox{}_X\left(E_{a,a'}\right) 
= \sum_{x\in X} E_{x,x,a,a'} = \delta_{a,a'} I_H;$$
hence
$$\left(\Tr\hspace{-0.07cm}\mbox{}_X\circ\Phi\right)\left(\rho\right) = \Tr(\rho)I_H, \ \ \ \rho\in M_A.$$
Thus, if $\omega\in \cl T(H)$ and $\rho\in M_A$ then \begin{eqnarray}\label{eq_omerho}
\left\langle\rho,\Tr(\omega) I_A\right\rangle
& = & 
\left\langle \Tr\hspace{-0.07cm}\mbox{}_A (\rho \otimes I_H), \omega\right\rangle
= 
\left\langle \left(\Tr\hspace{-0.07cm}\mbox{}_X \circ\Phi\right)(\rho),\omega\right\rangle\\
& = & 
\left\langle\Tr\hspace{-0.07cm}\mbox{}_X\left(V^*(\rho\otimes I_K)V\right),\omega\right\rangle. \nonumber
\end{eqnarray}
On the other hand, writing $\rho = (\rho_{a,a'})_{a,a'\in X}$, we have 
$$V^*(\rho\otimes I_K)V = 
\sum_{a,a'\in X} \rho_{a,a'} V^*(\epsilon_{a,a'}\otimes I_K)V = 
\left[\sum_{a,a'\in X} \rho_{a,a'} V_{a,x}^*V_{a',x'}\right]_{x,x'},$$
implying 
\begin{eqnarray*}
\left\langle\Tr\hspace{-0.07cm}\mbox{}_X\left(V^*(\rho\otimes I_K)V\right),\omega\right\rangle
& = & 
\sum_{x\in X}\sum_{a,a'\in X} \rho_{a,a'} \Tr(V_{a,x}^*V_{a',x}\omega)\\
& = & 
\sum_{a,a'\in X} \rho_{a,a'} \Tr\left(\sum_{x\in X}V_{a,x}^*V_{a',x}\omega\right).
\end{eqnarray*}
Now (\ref{eq_omerho}) implies that 
$$\Tr\left(\sum_{x\in X}V_{a,x}^*V_{a',x}\omega\right) = \delta_{a,a'} \Tr(\omega), \ \ \ a,a'\in X.$$
The latter equality holds for every $\omega\in \cl T(H)$; thus,
$$\sum_{x\in X}V_{a,x}^*V_{a',x} = \delta_{a,a'} I_K,$$
that is, $V^{\rm t}$ is an isometry.
\end{proof}

%%%%%%%%%%%%%%%%%%%%%%%%%%%%%%%%%%%%%%%%%%%%%%%%%
%%%%%%%%%%%%%%%%%%%%%%%%%%%%%%%%%%%%%%%%%%%%%%%%%

\subsection{The universal operator system}\label{ss_uopsys}

Recall \cite{hestenes, zettl} that a \emph{ternary ring} is a complex vector space $\cl V$, equipped with a 
ternary operation $[\cdot,\cdot,\cdot] : \cl V\times \cl V\times \cl V\to \cl V$, linear on the outer variables and 
conjugate linear in the middle variable, such that 
$$[s,t,[u,v,w]] = [s,[v,u,t],w] = [[s,t,u],v,w], \ \ \ s,t,u,v,w\in \cl V.$$
A \emph{ternary representation} of $\cl V$ is a linear map $\theta: \cl V \to \cl B(H,K)$, for some 
Hilbert spaces $H$ and $K$, such that 
$$\theta\left([u,v,w]\right) = \theta(u)\theta(v)^*\theta(w), \ \ \ u,v,w \in \cl V.$$
We call $\theta$ \emph{non-degenerate} if
$\text{span}\{\theta(u)^*\eta : u\in \cl V, \eta\in K\}$ is dense in $H$. 
A (concrete) \emph{ternary ring of operators (TRO)} \cite{zettl} 
is a subspace $\cl U\subseteq \cl B(H,K)$ for some Hilbert spaces 
$H$ and $K$ such that $S,T,R\in \cl U$ implies $ST^*R\in \cl U$. 
We refer the reader to \cite[Section 4.4]{blm} for details about TRO's and their abstract versions that will be used in the sequel.

Let $\cl V^0_{X}$ be the 
ternary ring, 
generated by elements $v_{a,x}$,  $a,x\in X$, satisfying the relations
\begin{equation}\label{eq_bc}
\sum_{a\in X} [v_{a'',x''},\hspace{-0.05cm}v_{a,x},\hspace{-0.05cm}v_{a,x'}] \hspace{-0.07cm}=\hspace{-0.07cm} \delta_{x,x'}v_{a'',x''} 
\mbox{ and}
\sum_{x\in X} [v_{a'',x''},\hspace{-0.07cm}v_{a,x},\hspace{-0.07cm}v_{a',x}] 
\hspace{-0.05cm}=\hspace{-0.05cm} \delta_{a,a'}v_{a'',x''},
\end{equation}
for all $x, x',x'',a, a', a''\in X$.
Note that relations (\ref{eq_bc}) are equivalent to 
\begin{equation}\label{eq_bcu}
\sum_{a\in X} [u,\hspace{-0.05cm}v_{a,x},\hspace{-0.05cm}v_{a,x'}] \hspace{-0.05cm}=\hspace{-0.05cm} \delta_{x,x'}u
\ \mbox{ and } \ 
\sum_{x\in X} [u,\hspace{-0.05cm}v_{a,x},\hspace{-0.05cm}v_{a',x}] \hspace{-0.05cm}=\hspace{-0.05cm} \delta_{a,a'}u,
\end{equation}
for all $x, x', a, a'\in X$ and all $u\in \cl V^0_{X}$.
Conditions (\ref{eq_bcu}) imply that 
the non-degenerate ternary representations $\theta : \cl V_{X}^0\to \cl B(H,K)$ correspond to bi-isometries $V = (V_{a,x})_{a,x}$ via the assignment $V_{a,x} = \theta(v_{a,x})$; in this case, we write $\theta = \theta_V$. 
Following \cite[Section 5]{tt-QNS}, we let $\hat{\theta} = \oplus_V \theta_V$, where in the direct 
sum we have chosen one representative from each unitary equivalence class of 
bi-isometries and the cardinality of the underlying Hilbert spaces are bounded by that of $\cl V$. 
The assignment $\|u\| := \|\hat{\theta}(u)\|$ defines a semi-norm on $\cl V^0_X$; 
we set $\cl V_X := \cl V^0_X/\ker\hat{\theta}$, observe that $\cl V_X$ is a TRO, 
and continue to write $v_{a,x}$ for the images of the 
canonical generators of $\cl V^0_X$ under the quotient map $q : \cl V^0_X\to \cl V_X$. 
The maps $\hat{\theta}$ and $\theta_V$ (for a bi-isometry $V$) give rise to corresponding 
ternary representations of $\cl V_X$, which we denote in the same way.

Let $\cl C_X$ be the right C*-algebra of the TRO $\cl V_X$ (so that, 
up to a *-isomorphism, $\cl C_X \cong \overline{{\rm span}(\hat{\theta}(\cl V_X)^*\hat{\theta}(\cl V_X))}$), 
write $e_{x,x',a,a'} = v_{a,x}^*v_{a',x'}$, 
and let 
$$\cl T_X = {\rm span}\{e_{x,x',a,a'} : x,x',a,a'\in X\},$$
viewed as an operator subsystem of $\cl C_X$. 
It is immediate that
\begin{equation}\label{eq_posit}
(e_{x,x',a,a'})_{x,x',a,a'}\in M_{XX}(\cl C_X)^+ 
\end{equation}
and that the relations
\begin{equation}\label{eq_birel}
\sum_{b\in A} e_{x,x',b,b} = \delta_{x,x'} 1 \ \mbox{ and } \ \sum_{y\in X} e_{y,y,a,a'} = \delta_{a,a'} 1, 
\ \ \ x,x',a,a'\in X,
\end{equation}
hold true. 

For a bi-isometry $V$, acting on the Hilbert space $H$, 
we write $\pi_V : \cl C_X\to \cl B(H)$ for the *-representation of $\cl C_X$, given by 
\begin{equation}\label{eq_piV}
\pi_V(S^*T) = \theta_V(S)^*\theta_V(T), \ \ \ S,T\in \cl V_{X}.
\end{equation}

\begin{lemma}\label{l_rr}
The following hold true:
\begin{itemize}
\item[(i)] Every non-degenerate ternary representation of $\cl V_{X}$ has the form $\theta_V$, for some 
bi-isometry $V$.
\item[(ii)] The map $\hat{\theta}$ is a faithful ternary representation of $\cl V_{X}$.
\item[(iii)] Every unital *-representation $\pi$ of $\cl C_{X}$ has the form $\pi_V$, for some 
bi-isometry $V$.
\end{itemize}
\end{lemma}

\begin{proof}
The arguments are similar to the ones in \cite[Lemma 5.1]{tt-QNS} where a version of our current setup is considered for isometries (that are not necessarily bi-isometries).  We address (iii) for the 
convenience of the reader. 
Let $\pi : \cl C_{X} \to \cl B(H)$ be a unital *-representation. 
Then there exists a ternary representation $\theta : \cl V_{X}\to \cl B(H,K)$
such that $\pi(S^*T) = \theta(S)^*\theta(T)$, $S,T\in \cl V_{X,A}$
(see e.g. \cite[Theorem 3.4]{blecher} and \cite[p. 1636]{elefkak}). 
Since $\pi$ is unital, $\theta$
is non-degenerate. By the universality of $\cl V_X$ described in (i), 
there exists an operator matrix
$V = (V_{a,x})$, whose entries satisfy the relations (\ref{eq_bc}), such that 
$\theta = \theta_V$, and hence $\pi = \pi_V$.
\end{proof}

Let $\cl V_{X,A}$ be the universal TRO of an isometry 
$(\tilde{v}_{a,x})_{a,x\in X}$, 
defined similarly to the TRO $\cl V_X$ \cite[Section 5]{tt-QNS}. 
Thus, the TRO $\cl V_{X,A}$ arises from a ternary ring, whose canonical generators
$\tilde{v}_{a,x}$, $x,a\in X$, are required to satisfy only the first of the relations (\ref{eq_bc}). 
We let $\cl C_{X,A}$ be the right C*-algebra of $\cl T_{X,A}$. Letting 
$\tilde{e}_{x,x',a,a'} = \tilde{v}_{a,x}^* \tilde{v}_{a',x'}$, $x,x',a,a'\in X$, 
we write 
\begin{equation}\label{eq_TXAd}
\cl T_{X,A} = {\rm span}\{\tilde{e}_{x,x',a,a'} : x,x',a,a'\in X\},
\end{equation}
viewed as an operator subsystem of $\cl C_{X,A}$ \cite{tt-QNS}.
It was shown in \cite[Theorem 5.2]{tt-QNS} that, for a Hilbert space $H$,  
the unital completely positive maps
$\phi : \cl T_{X,A}\to \cl B(H)$ correspond to stochastic operator matrices 
$(E_{x,x',a,a'})_{x,x',a,a'}$ via the assignment $\phi(e_{x,x',a,a'}) = E_{x,x',a,a'}$. 
We next provide a bistochastic version of this fact, to be used subsequently.

\begin{theorem}\label{th_ucp}
Let $H$ be a Hilbert space and $\phi : \cl T_{X}\to \cl B(H)$ be a linear map. 
Consider the conditions
\begin{itemize}
\item[(i)] $\phi$ is a unital completely positive map;
\item[(ii)] $\left(\phi(e_{x,x',a,a'})\right)_{x,x',a,a'} \in M_{XA} \otimes \cl B(H)$ is a bistochastic operator matrix;
\item[(iii)] there exists a unital *-representation $\pi : \cl C_{X}\to \cl B(H)$ such that $\phi = \pi|_{\cl T_{X}}$,
\end{itemize}
and
\begin{itemize}
\item[(i')] $\phi$ is a completely positive map;
\item[(ii')] $\left(\phi(e_{x,x',a,a'})\right)_{x,x',a,a'} \in \left(M_{XA}\otimes \cl B(H)\right)^+$.
\end{itemize}
\noindent 
Then (i)$\Leftrightarrow$(ii)$\Leftrightarrow$(iii) and 
(i')$\Leftrightarrow$(ii').
Thus, the pair $(\cl C_{X}, \iota)$, where $\iota$ is the inclusion map of $\cl T_{X}$ into $\cl C_{X}$, 
is the universal C*-cover of $\cl T_{X}$.

Moreover, if $\left(E_{x,x',a,a'}\right)_{x,x',a,a'}$ is a bistochastic operator matrix acting on a Hilbert space $H$
then there exists a (unique) unital completely positive map 
$\phi : \cl T_{X}\to \cl B(H)$ such that $\phi(e_{x,x',a,a'}) = E_{x,x',a,a'}$ for all $x,x',a,a'$.
\end{theorem}

\begin{proof}
(i)$\Rightarrow$(ii)
By Arveson's Extension Theorem and Stinespring's Theorem, 
there exist a Hilbert space $K$, a *-representation $\pi : \cl C_{X} \to \cl B(K)$ 
and an isometry $W\in \cl B(H,K)$, such that 
$\phi(u) = W^*\pi(u) W$, $u\in \cl T_{X}$. 
By Lemma \ref{l_rr}, $\pi = \pi_V$ for some bi-isometry $V = (V_{a,x})_{a,x}$.
By (\ref{eq_posit}), $E :=\left(\pi(e_{x,x',a,a'})\right)_{x,x',a,a'} \in (M_{XA}\otimes \cl B(K))^+$, and hence 
$$\left(\phi(e_{x,x',a,a'})\right)
= (I_X\otimes I_A\otimes W)^* E (I_X\otimes I_A\otimes W)
\in \left(M_{XA}\otimes B(H)\right)^+.$$
In addition, 
$$\sum_{b\in X} W^*V_{b,x}^* V_{b,x'} W = \delta_{x,x'} W^*W  = \delta_{x,x'} I, \ \ \ x,x'\in X,$$
and 
$$\sum_{y\in X} W^*V_{a,y}^* V_{a',y} W = \delta_{a,a'} W^*W  = \delta_{a,a'} I, \ \ \ a,a'\in X,$$
that is, the operator matrix $\left(\phi(e_{x,x',a,a'})\right)_{x,x',a,a'}$ is bistochastic.

(ii)$\Rightarrow$(iii)
By Theorem \ref{p_coor}, there exist a Hilbert space $K$ and 
a bi-isometry $V = (V_{a,x})_{a,x}\in \cl B(H^X,K^X)$ such that
$$\phi(e_{x,x',a,a'}) = V_{a,x}^*V_{a',x'}, \ \ \ x,x',a,a'\in X.$$
Recalling (\ref{eq_piV}), we have
\[
\pi_V(e_{x,x',a,a'}) = 
\theta_V(v_{a,x})^*\theta_V(v_{a',x'}) 
= V_{a,x}^*V_{a',x'} = \phi(e_{x,x',a,a'}),
\]
and hence 
the *-representation $\pi_V$ of $\cl C_{X}$ is an extension of $\phi$.

(iii)$\Rightarrow$(i) is trivial. 

%\smallskip

(i')$\Rightarrow$(ii') is a direct consequence of (\ref{eq_posit}) and the fact that $\cl T_X$ is an 
operator subsystem of $\cl C_X$. 

(ii')$\Rightarrow$(i') 
Let $T = \phi(1)$ and note that, for any $x,a\in X$, we have 
\begin{equation}\label{eq_bothT}
\sum_{b\in X}  E_{x,x,b,b} = \sum_{b\in X} \phi(e_{x,x,b,b}) = T = 
\sum_{y\in X} \phi(e_{y,y,a,a}) = \sum_{y\in X}  E_{y,y,a,a}.
\end{equation}
Assume first that $T$ is invertible. 
Following the proof of \cite[Proposition 5.4]{tt-QNS}, let
$\psi : \cl T_{X} \to \cl B(H)$ be the map given by 
\begin{equation}\label{eq_Thalf}
\psi(u) = T^{-1/2} \phi(u) T^{-1/2}, \ \ \ u\in \cl T_{X}.
\end{equation}
Setting 
$F = \left(\psi(e_{x,x',a,a'})\right)_{x,x',a,a'}$, we have that 
$$F = \left(I_{XA} \otimes T^{-1/2}\right) E \left(I_{XA} \otimes T^{-1/2}\right) \geq 0,$$
and (\ref{eq_bothT}) shows that 
$F$ is a bistochastic operator matrix. By the implication (ii)$\Rightarrow$(i), 
$\psi$ is completely positive, and hence so is $\phi$, as $\phi(\cdot) = T^{1/2} \psi(\cdot) T^{1/2}$.

Now relax the assumption that $T$ be invertible. 
Using the implication (ii)$\Rightarrow$(i), let $f : \cl T_X\to \bb{C}$ be 
the state given by $f(e_{x,x',a,a'}) = \frac{1}{|X|}\delta_{x,x'}\delta_{a,a'}$ and, 
for $\epsilon > 0$, let $\phi_{\epsilon} : \cl T_X\to \cl B(H)$ be given by 
$\phi_{\epsilon}(u) := \phi(u) + \epsilon f(u) I$.
Then 
$$\left(\phi_{\epsilon}(e_{x,x',a,a'})\right)_{x,x',a,a'} = E + \frac{\epsilon}{|X|} I_{XX}$$ 
and 
$\phi_{\epsilon}(I) = T + \epsilon I$ is invertible. By the previous paragraph, 
$\phi_{\epsilon}$ is completely positive and, since 
$\phi_{\epsilon}\to_{\epsilon \to 0} \phi$ in the point-norm topology, we conclude that $\phi$
is completely positive. 

\smallskip

Finally, suppose that $E = \left(E_{x,x',a,a'}\right)_{x,x',a,a'}$ is a bistochastic operator matrix acting on $H$. 
Letting $V$ be the bi-isometry, associated with $E$ via Theorem \ref{p_coor}, 
we have that the completely positive map 
$\phi := \pi_V|_{\cl T_{X}}$ satisfies the equalities $\phi(e_{x,x',a,a'}) = E_{x,x',a,a'}$ for all $x,x',a,a'$.
\end{proof}

We note that, 
if $\cl S$ is an operator system, its Banach space dual $\cl S^{\rm d}$ can be equipped with a 
natural matricial order structure. To this end, we recall \cite[Section 4]{CE2} that any matrix
$\phi = (\phi_{i,j})_{i,j=1}^n \in M_n(\cl S^{\rm d})$ gives rise to a linear map $F_{\phi} : \cl S\to M_n$,
defined by letting
\begin{equation}\label{eq_Fphi}
F_{\phi}(u) = \sum_{i,j=1}^n \phi_{i,j}(u) \epsilon_{i,j}, 
\end{equation}
and set
$$M_n(\cl S^{\rm d})^+ = \{\phi\in M_n(\cl S^{\rm d}) : F_{\phi} \mbox{ is completely positive}\}.$$
It was shown in \cite[Corollary 4.5]{CE2} that, if 
$\cl S$ is a finite dimensional operator system then the (matrix ordered) dual 
$\cl S^{\rm d}$ is an operator system, when equipped with a suitable 
faithful state as an Archimedean order unit. It is straightforward to verify that, in this case, 
$\cl S^{\rm dd}\cong_{\rm c.o.i.} \cl S$.

We identify an element 
$T\in M_{XA}$ with its matrix
$(\lambda_{x,x',a,a'})_{x,x',a,a'}$, where 
$$\lambda_{x,x',a,a'} = \langle T (e_{x'}\otimes e_{a'}), e_x\otimes e_a\rangle, \ \ \ 
x,x'\in X, a,a'\in A.$$
Let 
$$\cl L_{X,A} = \left\{(\lambda_{x,x',a,a'})\in M_{XA} : \exists \ c\in \bb{C} \mbox{ s.t.} 
\sum_{a\in A} \lambda_{x,x',a,a} = \delta_{x,x'} c, \ x,x'\in X\right\}$$
and consider $\cl L_{X,A}$ as an operator subsystem of $M_{XA}$. 
It was shown in \cite[Proposition 5.5]{tt-QNS} that the linear map 
$\tilde{\Lambda} : \cl T_{X,A}^{\rm d} \to \cl L_{X,A}$, 
given by 
$$\tilde{\Lambda}(\phi) = \left(\phi(\tilde{e}_{x,x',a,a'})\right)_{x,x',a,a'\in X},$$
is a unital complete order isomorphism between $\cl T_{X,A}^{\rm d}$ and $\cl L_{X,A}$. 
Let 
$$\cl L_X = 
\{(\lambda_{x,x',a,a'})_{x,x',a,a'}\in M_{XX} : \mbox{ there exists } c \in \bb{C} \mbox{ s.t. } $$
$$
\sum_{b\in X} \lambda_{x,x',b,b} = \delta_{x,x'} c \mbox{ and } \sum_{y\in X} \lambda_{y,y,a,a'} = \delta_{a,a'}c, 
\mbox{ for all } x,x',a,a'\in X\}.
$$

\begin{remark}\label{r_c1c2}
\rm 
If $C = (\lambda_{x,x',a,a'})_{x,x',a,a'}\in M_{XA}$ is a matrix and $c_1,c_2$ are scalars such that 
$\sum_{b\in X} \lambda_{x,x',b,b} = \delta_{x,x'} c_1$ for all $x,x'\in X$ and 
$\sum_{y\in X} \lambda_{y,y,a,a'} = \delta_{a,a'}c_2$ for all $a,a'\in A$, then
$$c_1 = \frac{1}{|X|} {\rm Tr}(C) = c_2.$$
\end{remark}

\begin{proposition}\label{p_dualTX}
The linear map $\Lambda : \cl T_{X}^{\rm d} \to \cl L_{X}$, given by 
\begin{equation}\label{p_Lam}
\Lambda(\phi) = \left(\phi(e_{x,x',a,a'})\right)_{x,x',a,a'\in X}
\end{equation}
is a well-defined complete order isomorphism. 
\end{proposition}

\begin{proof}
The arguments follow the proof of \cite[Proposition 5.5]{tt-QNS}, 
and we only highlight the required modifications. 
Using Theorem \ref{th_ucp}, we see that the map
$\Lambda_+ : \left(\cl T_{X}^{\rm d}\right)^+ \to \cl L_{X}^+$, given by 
$$\Lambda_+(\phi) = \left(\phi(e_{x,x',a,a'})\right)_{x,x',a,a'}, \ \ \ \phi\in \left(\cl T_{X}^{\rm d}\right)^+,$$
is well-defined; 
by additivity and homogeneity, $\Lambda_+$ 
extends to a ($\bb{C}$-)linear map $\Lambda : \cl T_{X}^{\rm d}\to \cl L_{X}$. 
A further application of Theorem \ref{th_ucp}, combined with Theorem \ref{p_coor}, shows that 
$\Lambda$ is completely positive and bijective. 

Let
$\phi_{i,j}\in \cl T_{X}^{\rm d}$, $i,j=1,\dots,m$, be such that the matrix 
$\left(\Lambda(\phi_{i,j})\right)_{i,j=1}^m$ is a positive element of $M_m\left(\cl L_{X}\right)$.
Let $\Phi : \cl T_{X}\to M_m$ be given by $\Phi(u) = (\phi_{i,j}(u))_{i,j=1}^m$.
Then $\left(\Phi(e_{x,x',a,a'})\right)\in M_m\left(\cl L_{X}\right)^+$. 
By Theorem \ref{th_ucp}, $\Phi$ is completely positive, that is, 
$\left(\phi_{i,j}\right)_{i,j=1}^m\in M_m\left(\cl T_{X}^{\rm d}\right)^+$.
Thus, $\Lambda^{-1}$ is completely positive, and the proof is complete. 
\end{proof}

\begin{corollary}\label{c_flip1}
The linear map $\frak{f} : \cl T_X\to \cl T_X$, given by $\frak{f}(e_{x,x',a,a'}) = e_{x',x,a',a}$, 
is a complete order automorphism. 
\end{corollary}

\begin{proof}
The map $\Phi : M_{XX}\to M_{XX}$, given by 
$\Phi(\epsilon_{x,a}\otimes \epsilon_{x',a'}) = \epsilon_{x',a'}\otimes \epsilon_{x,a}$, is a (unitarily implemented)
complete order automorphism. Further, $\Phi(\cl L_X) = \cl L_X$, and hence $\Phi$ induces a complete order 
automorphism $\Phi_0 : \cl L_X\to \cl L_X$. Using Proposition \ref{p_dualTX}, 
we have that its dual $\Phi_0^*$  a complete order automorphism of $\cl T_X$. 
For $x,x',a,a'\in X$ and $T = (\lambda_{x,x',a,a'})\in \cl L_X$, we have
$$
\left\langle \Phi_0^*(e_{x,x',a,a'}), T\right\rangle
= 
\left\langle e_{x,x',a,a'}, \Phi_0(T)\right\rangle
= \lambda_{x',x,a',a} 
= \left\langle \frak{f}(e_{x,x',a,a'}), T\right\rangle,$$
and the proof is complete.
\end{proof}

Write 
\begin{equation}\label{eq_JXd}
\cl J_{X} = {\rm span} \left\{\sum_{y\in X} \tilde{e}_{y,y,a,a'} - \delta_{a,a'} 1 : a,a'\in X\right\};
\end{equation}
thus, $\cl J_X$ is a linear subspace of the operator system $\cl T_{X,A}$ defined in (\ref{eq_TXAd}). 
Let $\tilde{\cl J}_X$ be the closed ideal of $\cl C_{X,A}$, generated by $\cl J_X$.
Write $q_{X}$ for the quotient map from $\cl T_{X,A}$ onto $\cl T_{X,A}/\cl J_{X}$.

Recall that, if $\cl S$ is an operator system, a subspace $\cl J\subseteq \cl S$ is called a \emph{kernel} 
\cite[Definition 3.2]{kptt_adv} if there exist an operator system $\cl R$ and a unital completely positive map (equivalently, a completely positive map) $\phi: \cl S\to \cl R$ such that $\cl J = \ker(\phi)$.

\begin{proposition}\label{p_quot}
The space $\cl J_{X}$ is a kernel in $\cl T_{X,A}$ and the linear map $\iota$, given by 
\begin{equation}\label{eq_iota}
\iota \left(q_{X}\left(\tilde{e}_{x,x',a,a'}\right)\right) = e_{x,x',a,a'}, \ \ \ x,x',a,a'\in X,
\end{equation}
is a well-defined complete order isomorphism from $\cl T_{X,A}/\cl J_{X}$ onto $\cl T_X$. 
In addition, $\cl C_{X,A}/\tilde{\cl J}_X \cong \cl C_X$, up to a canonical *-isomorphism. 
\end{proposition}

\begin{proof}
Let 
$\alpha : \cl L_X\to \cl L_{X,A}$ be the inclusion map. Since $\cl L_X$ and $\cl L_{X,A}$ are operator 
subsystems of $M_{XX}$, we have that $\alpha$ is a complete order embedding.
By \cite[Proposition 1.15]{fp}, \cite[Proposition 5.5]{tt-QNS} and Proposition \ref{p_dualTX}, 
its dual $\alpha^* : \cl T_{X,A}\to \cl T_{X}$ is a complete quotient map.
Note that, if $T\in \cl L_X$ and $a,a'\in X$ then 

$$\left\langle\alpha^*\left(\sum_{y\in X} \tilde{e}_{y,y,a,a'} - \delta_{a,a'} 1\right),T\right\rangle
= 
\left\langle\sum_{y\in X} \tilde{e}_{y,y,a,a'} - \delta_{a,a'} 1,\alpha(T)\right\rangle = 0,$$

that is, $\cl J_{X}\subseteq \ker(\alpha^*)$. 

Consider the canonical linear mappings
$$\cl T_{X,A}\to \cl T_{X,A}/\cl J_X \to \cl T_{X,A}/\ker(\alpha^*) \to \cl T_X,$$
of which the first two are surjective linear maps whose composition is completely positive,
while the third is a complete order isomorphism (note that the quotient $\cl T_{X,A}/\cl J_X$ is 
linear algebraic). 
Dualising and using Proposition \ref{p_dualTX}, we obtain the chain of maps
\begin{equation}\label{eq_embed}
\cl L_X\cong \left(\cl T_{X,A}/\ker(\alpha^*)\right)^{\rm d} 
\hookrightarrow \left(\cl T_{X,A}/\cl J_X\right)^{\rm d} \to \cl L_{X,A}.
\end{equation}
By the definition of $\cl J_X$ (see (\ref{eq_JXd})),
the elements of 
$\left(\cl T_{X,A}/\cl J_X\right)^{\rm d}$ 
correspond, via the last of the three maps in (\ref{eq_embed}), to elements of the subspace $\cl L_X$ of $\cl L_{X,A}$.
It now follows that the middle map in (\ref{eq_embed}) is a linear isomorphism, and hence 
$\ker(\alpha^*) = \cl J_{X}$. 
In particular, $\cl J_{X}$ is a kernel in $\cl T_{X,A}$ and $\left(\cl T_{X,A}/\cl J_{X}\right)^{\rm d} \cong \cl L_X$
complete order isomorphically.
Dualising, we see that $\cl T_{X,A}/\cl J_{X} \cong \cl T_X$ complete order isomorphically via the 
map $\iota$ defined in (\ref{eq_iota}).

By the universal property of $\cl C_X$, there exists a unital *-epimorphism 
$\pi : \cl C_{X,A}\to \cl C_X$ such that $\pi(\tilde{e}_{x,x',a,a'}) = e_{x,x',a,a'}$, $x,x',a,a'\in X$. 
Let $\cl J = \ker(\pi)$ and $\tilde{\pi} : \cl C_{X,A}/\cl J \to \cl C_X$ be the induced *-isomorphism. 
We have
$$\pi\left(\delta_{a,a'}1 - \sum_{x\in X}\tilde{e}_{x,x,a,a'}\right) = 
\delta_{a,a'}1 - \sum_{x\in X}e_{x,x,a,a'} = 0;$$
thus, $\tilde{\cl J}_X\subseteq \cl J$. 

The block operator matrix $\left(\tilde{e}_{x,x',a,a'} + \tilde{\cl J}_X\right)_{x,x',a,a'}$ is 
bistochastic, and hence it gives rise, via Theorem \ref{th_ucp}, to a canonical unital surjective *-homomorphism 
$\pi' : \cl C_X\to \cl C_{X,A}/\tilde{\cl J}_X$.
We thus have a chain of unital *-homomorphisms
$$\cl C_X \stackrel{\pi'}{\longrightarrow} \cl C_{X,A}/\tilde{\cl J}_X
\longrightarrow \cl C_{X,A}/\cl J 
\stackrel{\tilde{\pi}}{\longrightarrow} \cl C_X,$$
whose composition is the identity. It follows that $\cl J = \tilde{\cl J}_X$, and the proof is complete. 
\end{proof}

In the sequel, write
$\hat{q}_X : \cl C_{X,A}\to \cl C_X$ for the quotient map arising from Proposition \ref{p_quot}, 
and continue to write
$q_X$ for the quotient map from $\cl T_{X,A}$ onto $\cl T_X$.
Before formulating the next corollary, we
recall that an operator system $\cl S$ is said to possess the
\emph{local lifting property} \cite[Section 8]{kptt_adv} if for every finite dimensional operator 
subsystem $\cl S_0\subseteq \cl S$, C*-algebra $\cl A$, and closed ideal $\cl J\subseteq \cl A$, 
every unital completely positive map $\phi_0 : \cl S_0\to \cl A/\cl J$ admits a lifting to a completely positive map 
$\phi : \cl S_0\to \cl A$ (that is, if $q : \cl A\to \cl A/\cl J$ denotes the quotient map, 
the identity $q\circ \phi = \phi_0$ holds).

\begin{corollary}\label{c_OSLLP}
The operator system $\cl T_{X}$ has the local lifting property. 
\end{corollary}

\begin{proof}
By \cite[Corollary 5.6]{tt-QNS}, $\cl T_{X,A}$ is an operator system quotient of 
$M_{XX}$ while, by 
Proposition \ref{p_quot}, $\cl T_X$ is an operator system 
quotient of $\cl T_{X,A}$. It follows that 
$\cl T_X$ is an operator system quotient of $M_{XX}$. 
The statement is now a consequence of \cite[Theorem 6.8]{kavruk_JOT}.
\end{proof}

Realising the commuting tensor product of operator systems as an operator subsystem of 
maximal tensor products has been of
importance from the beginning of the tensor product theory in the  operator system category \cite{kptt}. 
By Theorem \ref{th_ucp} and \cite[Theorem 6.4]{kptt}, for an arbitrary operator system $\cl R$, we have
$\cl T_{X}\otimes_{\rm c} \cl R \subseteq_{\rm c.o.i.} \cl C_{X}\otimes_{\max} C^*_u(\cl R)$;
the next proposition establishes a stronger inclusion.

\begin{proposition}\label{p_inject-comm}
Let $\cl R$ be an operator system. Then 
$\cl T_{X}\otimes_{\rm c} \cl R \subseteq_{\rm c.o.i.} \cl C_{X}\otimes_{\max} \cl R.$
\end{proposition}

\begin{proof}
Let $\iota : \cl T_{X}\to \cl C_{X}$ be the inclusion map.
By the functioriality of the commuting tensor product
and the fact that the commuting and the maximal tensor products coincide provided 
one of the terms is a C*-algebra \cite[Theorem 6.7]{kptt}, 
$\iota\otimes\id : \cl T_{X}\otimes_{\rm c} \cl R\to \cl C_{X}\otimes_{\max} \cl R$
is a (unital) completely positive map. 
Assume that 
$$w\in M_n\left(\cl T_{X}\otimes \cl R\right)\cap M_n\left(\cl C_{X}\otimes_{\max} \cl R\right)^+,$$ 
let $H$ be a Hilbert space, and $\phi : \cl T_{X}\to \cl B(H)$ and 
$\psi : \cl R\to \cl B(H)$ be unital completely positive maps with commuting ranges. 
By Theorem \ref{th_ucp},
$\phi$ extends to a *-homomorphism $\pi : \cl C_{X}\to \cl B(H)$. 
Since $\cl C_{X}$ is generated by $\cl T_{X}$ as a C*-algebra, 
$\pi(u)\in \psi(\cl R)'$ for every $u\in \cl C_{X}$; thus, 
$$(\phi\cdot\psi)^{(n)}(w) = (\pi\cdot\psi)^{(n)}(w)\in M_n\left(\cl B(H)\right)^+,$$
and hence $w\in M_n\left(\cl T_{X}\otimes_{\rm c} \cl R\right)^+$.
It follows that $\iota\otimes\id$ is a complete order embedding.
\end{proof}

%%%%%%%%%%%%%%%%%%%%%%%%%%%%%%%%%%%%%%%%%%%%%%%%%
%%%%%%%%%%%%%%%%%%%%%%%%%%%%%%%%%%%%%%%%%%%%%%%%%

\section{Quantum magic squares}\label{ss_cbom}

In \cite{dlcdn}, the concept of a quantum magic square was defined and studied, exhibiting examples which show that not every 
quantum magic square dilates to a magic unitary. The aim of this section is to present an operator system viewpoint on this result, linking the dilation properties of a quantum magic square to complete positivity of canonical maps, associated with it. The universal operator system of a quantum magic square 
and its properties will further be used in Section \ref{s_rbc}. 

Recall \cite{dlcdn} that a block operator matrix $E = (E_{x,a})_{x,a\in X}$, where $E_{x,a}\in \cl B(H)$, 
$x,a\in X$, is called a \emph{quantum magic square} 
if $E_{x,a}\geq 0$ and 
$$\sum_{b\in X} E_{x,b} = \sum_{y\in X} E_{y,a} = I \ \ \mbox{ for all } x,a\in X.$$
The quantum magic square $E$ is called 
a \emph{magic unitary} 
(or a \emph{quantum permutation})
if $E_{x,a}$ is a projection for all $x,a\in X$ (see e.g. \cite[Definition 2.3]{lmr}). 
Noting that $\cl D_{XX}\otimes\cl B(H) \subseteq M_{XX}\otimes\cl B(H)$, 
we have that $E$ is a quantum magic square precisely when 
$\sum_{x,a\in X} \epsilon_{x,x} \otimes \epsilon_{a,a} \otimes E_{x,a}$ 
is a bistochastic operator matrix in $M_{XX}\otimes\cl B(H)$. 

Two subclasses of quantum magic squares 
were singled out in \cite{dlcdn} (see \cite[Definition 5 and Example 8]{dlcdn}). 
We will call a quantum magic square $(E_{x,a})_{x,a}$, acting on a Hilbert space $H$, \emph{dilatable} if 
there exists a Hilbert space $K$, an isometry $V : H\to K$, and a quantum permutation $(P_{x,a})_{x,a}$ acting on $K$, such that 
\begin{equation}\label{eq_EPxa}
E_{x,a} = V^*P_{x,a}V, \ \ \ x,a\in X.
\end{equation}
The quantum magic square $(E_{x,a})_{x,a}$ will be called
\emph{locally dilatable} if (\ref{eq_EPxa})
holds for a commuting family $\{P_{x,a}\}_{x,a}$ that forms a quantum permutation. 
It is clear that, up to unitary identifications, 
condition (\ref{eq_EPxa}) can be replaced by 
the conditions 
$E_{x,a} = QP_{x,a}Q$, where we have assumed that 
$H\subseteq K$, and $Q : K\to H$ is the orthogonal projection.

For $x,a\in X$, we set $e_{x,a} := e_{x,x,a,a}$ and 
$$\cl S_X := {\rm span}\{e_{x,a} : x,a\in X\},$$
viewed as an operator subsystem of $\cl T_X$.

\begin{theorem}\label{th_ucpclasb}
Let $H$ be a Hilbert space and $\phi : \cl S_{X}\to \cl B(H)$ be a linear map. 
Consider the conditions
\begin{itemize}
\item[(i)] $\phi$ is a unital completely positive map;
\item[(ii)] $\left(\phi(e_{x,a})\right)_{x,a}$ is a quantum magic square,
\end{itemize}
and
\begin{itemize}
\item[(i')] $\phi$ is a completely positive map;
\item[(ii')] $\left(\phi(e_{x,a})\right)_{x,a} \in \left(\cl D_{XX} \otimes \cl B(H)\right)^+$.
\end{itemize}
\noindent 
Then (i)$\Leftrightarrow$(ii) and 
(i')$\Leftrightarrow$(ii').
Moreover, if $\left(E_{x,a}\right)_{x,a}$ is a quantum magic square acting on a Hilbert space $H$
then there exists a (unique) unital completely positive map 
$\phi : \cl S_{X}\to \cl B(H)$ such that $\phi(e_{x,a}) = E_{x,a}$ for all $x,a\in X$.
\end{theorem}

\begin{proof}
(i)$\Rightarrow$(ii) 
Let $\phi : \cl S_X\to \cl B(H)$ be a unital completely positive map, for some Hilbert space $H$. 
By Arveson's Extension Theorem, $\phi$ has a completely positive extension $\tilde{\phi} : \cl T_X\to \cl B(H)$.
Setting $E_{x,x',a,a'} := \tilde{\phi}(e_{x,x',a,a'})$, Theorem \ref{th_ucp} implies that 
$(E_{x,x',a,a'})_{x,x',a,a'}$ is a bistochastic matrix. In particular, 
$(\tilde{\phi}(e_{x,a}))_{x,a}$, that is, $(\phi(e_{x,a}))_{x,a}$, is a quantum magic square.

(ii)$\Rightarrow$(i) 
Set $E_{x,a} := \phi(e_{x,a})$ and $\tilde{E}_{x,x',a,a'} := \delta_{x,x'}\delta_{a,a'} E_{x,a}$, $x,x',a,a'\in X$. 
Then 
$(\tilde{E}_{x,x',a,a'})_{x,x',a,a'}$ is a bistochastic operator matrix and, by Theorem \ref{th_ucp}, 
there exists a (unital) completely positive map $\tilde{\phi} : \cl T_X\to \cl B(H)$ such that 
$\tilde{\phi}(e_{x,x',a,a'}) = \tilde{E}_{x,x',a,a'}$, $x,x',a,a'\in X$. 
As $\phi = \tilde{\phi}|_{\cl S_X}$, the map $\phi$ is completely positive. 

(i')$\Rightarrow$(ii') is a direct consequence of Theorem \ref{th_ucp} and Arveson's Extension Theorem. 

(ii')$\Rightarrow$(i') 
Set $E_{x,a} := \phi(e_{x,a})$, $x,a\in X$. 
For $x\in X$, let $T = \sum_{a\in X} E_{x,a}$; then $T\in \cl B(H)^+$. Assume first that $T$ is invertible. 
Then the matrix $\left(T^{-1/2}E_{x,a}T^{-1/2}\right)_{x,a}$ is a quantum magic square; 
by the implication (ii)$\Rightarrow$(i), the linear map $\psi : \cl S_X\to \cl B(H)$, given by 
$\psi(e_{x,a}) = T^{-1/2}E_{x,a}T^{-1/2}$, is completely positive. 
Since $\phi(u) = T^{1/2}\psi(u)T^{1/2}$, $u\in \cl S_X$, the map $\phi$ is completely positive. 
If $T$ is not invertible, we fix a state $f : \cl S_X\to \bb{C}$ and, for $\epsilon > 0$,  
consider the map $\phi_{\epsilon} : \cl S_X\to \cl B(H)$, given by 
$\phi_{\epsilon}(u) = \phi(u) + \epsilon f(u)I$. 
The proof now proceeds similarly to the proof of the implication (ii')$\Rightarrow$(i') of Theorem \ref{th_ucp}. 

The last statement in the theorem follows from the proof of the implication (ii)$\Rightarrow$(i). 
\end{proof}

Let 
$$\cl M_X = \{
(\mu_{x,a})_{x,a}\hspace{-0.05cm}\in\hspace{-0.05cm} \cl D_{XX} : 
\textstyle\sum_{b\in X} \mu_{x,b} 
\hspace{-0.05cm} = \hspace{-0.05cm} \textstyle\sum_{y\in X} \mu_{y,a}, \ x,a\in X\},$$
considered as an operator subsystem of $\cl D_{XX}$.
Since every operator system is spanned by its positive elements, $\cl M_X$ is
the operator system spanned by the scalar bistochastic matrices in $\cl D_{XX}$.

\begin{corollary}\label{c_SXPXd}
We have that $\cl S_X^{\rm d} \cong \cl M_X$, up to a canonical unital complete order isomorphism.
\end{corollary}

\begin{proof}
Let $\cl M_{X,1}^+$ be the convex set of all scalar bistochastic matrices, that is, matrices
$T = (t_{x,a})_{x,a}\in \cl M_X^+$ with $\sum_{a\in X} t_{x,a} = 1$, $x\in X$. 
By Theorem \ref{th_ucpclasb}, if $T\in \cl M_{X,1}^+$ then 
the map $\gamma(T) : \cl S_X\to \bb{C}$, given by $\gamma(T)(e_{x,a}) = t_{x,a}$, is a (well-defined)
state on $\cl S_X$. Writing an arbitrary element $T\in \cl M_X$ as a linear combination 
$T = \sum_{i=1}^k \lambda_i T_i$, where $T_i\in \cl M_{X,1}^+$, $i = 1,\dots,k$, set 
$\gamma(T) := \sum_{i=1}^k \lambda_i \gamma(T_i)$. The map $\gamma$ is (linear and) well-defined: 
if $T_i = (t_{x,a}^{(i)}) \in \cl M_{X,1}^+$, $i = 1,\dots,k$, and $\sum_{i=1}^k \lambda_i T_i = 0$, then 
$\sum_{i=1}^k \lambda_i t_{x,a}^{(i)} = 0$ for all $x,a\in X$, which implies that $\sum_{i=1}^k \lambda_i \gamma(T_i) = 0$. 

%It is straightforward that the assignment $\gamma$, given by 
%$$\langle\gamma(\epsilon_{x,y}),e_{x',y'}\rangle = \delta_{x,x'} \delta_{y,y'}, \ \ \ x,x',y,y'\in X,$$
%defines a linear map $\gamma : \cl M_X\longrightarrow \cl S_X^{\rm d}$.
%We claim that $\gamma$ is completely positive. 
Let $E = (E^{(i,j)})_{i,j=1}^n \in M_n(\cl M_X)^+$ and, using the canonical shuffle, write 
$E = (E_{x,y})_{x,y}$, where $E_{x,y}\in M_n$, $x,y\in X$, are such that 
\begin{equation}\label{eq_EW}
\sum_{y'\in X} E_{x,y'} = \sum_{x'\in X} E_{x',y}, \ \ \ x,y\in X.
\end{equation}
Using Theorem \ref{th_ucpclasb}, we see that there exists a completely positive map 
$\phi : \cl S_X\to M_n$ such that $\phi(e_{x,y}) = E_{x,y}$, $x,y\in X$. 
On the other hand, the element $\gamma^{(n)}(E)$ of $M_n(\cl S_X^{\rm d})$ gives rise, 
via (\ref{eq_Fphi}), to a linear map $F_{\gamma^{(n)}(E)} : \cl S_X\to M_n$. 
We have that
$$F_{\gamma^{(n)}(E)}(e_{x,a}) = \sum_{i,j=1}^n \gamma\left(E^{(i,j)}\right)(e_{x,a}) \epsilon_{i,j}
= E_{x,y} = \phi(e_{x,a}), \ \ \ x,a\in X,$$
that is, $F_{\gamma^{(n)}(E)} = \phi$. 
In particular, $F_{\gamma^{(n)}(E)}$ is completely positive, and
it follows that the map $\gamma$ is completely positive. 

It follows from Theorem \ref{th_ucpclasb} that the (linear) map $\gamma$ is surjective; 
thus, it is injective. We show that $\gamma^{-1}$ is completely positive. 
Assume that $W\in M_n(\cl S_X^{\rm d})^+$; this means that the linear map $F_W : \cl S_X\to M_n$, 
canonically associated with $W$, is completely positive. 
Set $E_{x,y} := F_W(e_{x,y})$, $x,y\in X$; by Theorem \ref{th_ucpclasb}, 
$E := (E_{x,y})_{x,y}\in (M_n\otimes M_X)^+$.
This, in turns, means that $(\gamma^{-1})^{(n)}(W)\in (M_n\otimes M_X)^+$. 
Since relations (\ref{eq_EW}) are satisfied for the matrices $E_{x,y}$, we have that, in fact, 
$(\gamma^{-1})^{(n)}(W)\in (M_n\otimes \cl M_X)^+$, and the proof is complete. 
\end{proof}

Let 
$$\cl J_{X}^{\neq} = {\rm span} \left\{e_{x,x',a,a'} : x\neq x' \mbox{ or } a\neq a'\right\};$$
note that $\cl J_X^{\neq}$ is a linear subspace of the operator system $\cl T_{X}$.

\begin{proposition}\label{p_SXPXd}
The space $\cl J_{X}^{\neq}$ is a kernel in $\cl T_{X}$ and, up to a unital complete order isomorphism, 
$\cl S_X \cong \cl T_X/\cl J_X^{\neq}$.
\end{proposition}

\begin{proof}
By Theorem \ref{th_ucp}, there exists a unital completely positive map
$\beta : \cl T_X\to \cl S_X$, such that $\beta(e_{x,x',a,a'}) = \delta_{x,x'}\delta_{a,a'} e_{x,a}$, $x,x',a,a'\in X$. 
It is clear that $\cl J_X^{\neq}\subseteq \ker(\beta)$. 
On the other hand, by Proposition \ref{p_dualTX} and Corollary \ref{c_SXPXd}, 
we have a chain of four canonical linear maps
\begin{equation}\label{eq_MXduali}
\cl M_X\cong \cl S_X^{\rm d}
\longrightarrow \left(\cl T_X/\ker(\beta)\right)^{\rm d} 
\longrightarrow \left(\cl T_X/\cl J_X^{\neq}\right)^{\rm d}
\longrightarrow \cl T_X^{\rm d} \cong \cl L_X;
\end{equation}
of which the first, the second and the fourth are completely positive. 
In addition, 
the image of $\cl M_X$ in $\cl L_X$ under the composition of these maps coincides with itself; 
thus, $\ker(\beta) \subseteq \cl J_X^{\neq}$ 
and hence $\cl J_X^{\neq}$ is a kernel in $\cl T_X$. 
Dualising the second map in  (\ref{eq_MXduali}), we further obtain a chain 
$$\cl T_X/\cl J_X^{\neq} \to \cl S_X\to \cl T_X\to \cl T_X/\cl J_X^{\neq}$$
of completely positive maps, whose composition is the identity map on $\cl T_X/\cl J_X^{\neq}$. 

On the other hand, we have a chain of canonical completely positive maps
$$\cl S_X\to \cl T_X\to \cl T_X/\cl J_X^{\neq} \to \cl S_X,$$
whose composition is the identity map on $\cl S_X$. 
It follows that $\cl S_X \cong \cl T_X/\cl J_X^{\neq}$, up to a canonical complete order isomorphism.
\end{proof}

In Theorem \ref{p_dilate} below, we characterise the dilatable and locally dilatable quantum magic squares in 
operator system terms. 
Let $C(S_X^+)$ be the universal C*-algebra 
generated by projections $p_{x,a}$, $x,a\in X$, with the properties 
$$\sum_{b\in X} p_{x,b} = \sum_{y\in X} p_{y,a} = 1, \ \ \ x,a\in X$$
(thus, $C(S_X^+)$ is the universal C$^\ast$-algebra of functions on the quantum permutation group on 
$X$; see e.g. \cite{bcehpsw}). 
Write 
$$\cl P_X = {\rm span}\{p_{x,a} : x,a\in X\},$$
viewed as an operator subsystem of $C(S_X^+)$. 

Recall \cite[Section 3]{ptt} that the 
\emph{minimal operator system} based on $\cl P_X$
has matricial cones $M_n({\rm OMIN}(\cl P_X))^+$, given by 
$$\hspace{-1cm} M_n({\rm OMIN}(\cl P_X))^+
= \{(t_{i,j})_{i,j} \in M_n(\cl P_X) : \sum_{i,j=1}^n \lambda_i\overline{\lambda_j} t_{i,j}\in \cl P_X^+,$$ 
$$\hspace{7cm} \mbox{ for all }
\lambda_i\in \bb{C}, i\in [n]\},$$
and that the corresponding 
\emph{maximal operator system} based on $\cl P_X$ has
matricial cones $M_n({\rm OMAX}(\cl P_X))^+$ generated,
as cones with an Archimedean order unit, 
by the elementary tensors of the form $T\otimes u$, where $T\in M_n^+$ and $u\in \cl P_X^+$.

\begin{proposition}\label{p_PXSXdd}
There exist canonical unital completely positive maps
\begin{equation}\label{eq_OMAXPXSX}
{\rm OMAX}(\cl P_X) \longrightarrow \cl S_X \longrightarrow \cl P_X.
\end{equation}
\end{proposition}

\begin{proof}
By Theorem \ref{th_ucpclasb}, the linear map 
$\frak{q} : \cl S_X\to \cl P_X$, given by $\frak{q}(e_{x,a}) = p_{x,a}$, $x,a\in X$, 
is (unital and) completely positive.
%By the universal property of the maximal operator system \cite[Theorem 3.22]{ptt}, 
%it suffices to show that $\frak{q}^{-1} : \cl P_X\to \cl S_X$ is positive. 

Suppose that $\phi\in (\cl S_X^{\rm d})^+$; by Proposition \ref{p_SXPXd}, 
$\phi$ can be canonically identified with a matrix $(\lambda_{x,a})_{x,a}$ in $\cl M_X^+$. 
By Birkhoff's Theorem and the argument in the proof of Corollary \ref{c_SXPXd}, 
we can further assume that
there exists a permutation $f : X\to X$ such that $\lambda_{x,a} = \delta_{f(x),a}$, $x,a\in X$. 
By the universal property of $C(S_X^+)$, the permutation $f$ gives rise to a canonical 
*-representation $\pi : C(S_X^+)\to \bb{C}$. It follows that $\pi|_{\cl P_X} : \cl P_X\to \bb{C}$ is (completely) positive.
We thus obtain a canonical positive map 
$r : \cl S_X^{\rm d}\to \cl P_X^{\rm d}$ which, by the universal property of the minimal operator system structure, 
gives rise to a canonical completely positive map
$\cl S_X^{\rm d}\to {\rm OMIN}(\cl P_X^{\rm d})$;
dualising, we have a canonical completely positive map 
${\rm OMAX}(\cl P_X) \to \cl S_X$. 

Note that the composition of the maps in (\ref{eq_OMAXPXSX}) is the identity map on $\cl P_X$; 
hence $\frak{q}$ is invertible. Since $\frak{q}^{-1} = r$, we have that $\frak{q}^{-1}$ is positive,
completing the proof. 
%For $x,y\in X$ we have that 
%$$(\frak{q}^{-1})^{\rm d}(\phi)(p_{x,y}) = \phi(\frak{q}^{-1}(p_{x,y})) = \phi(e_{x,y}) 
%= \delta_{f(x),y} = \pi(p_{x,y}).$$
%Therefore $\frak{q}^{\rm d}(\phi) = \pi|_{\cl P_X}$, implying that $\frak{q}^{\rm d}(\phi)\in (\cl P_X^{\rm d})^+$, 
%thus completing the proof. 
\end{proof}

%{p_SXPXd}

\begin{theorem}\label{p_dilate}
Let $H$ be a Hilbert space and 
$E = (E_{x,a})_{x,a}$ be a quantum magic square acting on $H$. Then
\begin{itemize}
\item[(i)] $E$ is dilatable if and only if there exists a completely positive map
$\phi : \cl P_X\to \cl B(H)$, such that 
$\phi(p_{x,a}) = E_{x,a}$, $x,a\in X$;

\item[(ii)] 
%If $H$ is finite dimensional, 
$E$ is locally dilatable if and only if there exists a completely positive map
$\phi : {\rm OMIN}(\cl P_X)\to \cl B(H)$, such that 
$\phi(p_{x,a}) = E_{x,a}$, $x,a\in X$.
\end{itemize}
\end{theorem}

\begin{proof}
(i) 
Let $P = (P_{x,a})_{x,a}$ be a magic unitary on a Hilbert space $K$ containing $H$ such that, if $Q$ is the projection 
from $K$ onto $H$, then  $E_{x,a} = QP_{x,a}Q$, $x,a\in X$. 
By the universal property of $C(S_X^+)$, there exists a unital *-homomorphism $\pi : C(S_X^+)\to \cl B(K)$ 
such that $\pi(p_{x,a}) = P_{x,a}$, $x,a\in X$. 
Let $\phi : \cl P_X\to \cl B(H)$ be the linear map, defined by $\phi(u) = Q\pi(u)Q$, $u\in \cl P_X$. As a compression of a completely positive map, $\phi$ is completely positive; by construction, $\phi(p_{x,a}) = E_{x,a}$, $x,a\in X$. 

For the converse direction, let $\tilde{\phi} : C(S_X^+) \to \cl B(H)$ be a unital completely positive extension of $\phi$, 
whose existence is guaranteed by Arveson's Extension Theorem. 
Using Stinespring's Theorem, let $K$ be a Hilbert space, $\pi :C(S_X^+)\to \cl B(K)$ be a unital *-representation, and $V : H\to K$ be an isometry, such that $\tilde{\phi}(u) = V^*\pi(u)V$, $u\in C(S_X^+)$. 
Letting $P_{x,a} = \pi(p_{x,a})$, we have that $(P_{x,a})_{x,a}$ is a magic unitary that dilates $E$.

(ii) 
We first consider the case where $n := {\rm dim}(H)$ is finite. 
Identifying $\cl B(H)$ with $M_n$, 
suppose that $\phi : {\rm OMIN}(\cl P_X)\to M_n$ is a unital completely positive map. 
Let 
$$f_{\phi} : M_n({\rm OMIN}(\cl P_X))\to \bb{C}$$ 
be the canonical functional, associated with $\phi$ as in \cite[Chapter 6]{Pa}; 
thus, 
$$f_{\phi}((w_{i,j})_{i,j}) = \frac{1}{n} \sum \langle\phi(w_{i,j})e_j,e_i\rangle, \ \ \ (w_{i,j})_{i,j}\in M_n(\cl P_X).$$
By \cite[Theorem 6.1]{Pa}, $f_{\phi}$ is positive. By \cite[Theorem 4.8]{ptt}, 
$f_{\phi}$ can be canonically identified with an element of $M_n({\rm OMAX}(\cl P_X^{\rm d}))^+$
(see \cite[Section 3]{ptt}). 
By Proposition \ref{p_PXSXdd}, Corollary \ref{c_SXPXd} and 
the definition of the maximal operator system structure, 
\begin{equation}\label{eq_sumr}
f_{\phi} \equiv \sum_{l=1}^{r_{\phi}} \alpha_l\otimes\beta_l, 
\end{equation}
where $\alpha_l\in \cl M_X^+$ and $\beta_l\in M_n^+$, $l\in [r_{\phi}]$. 

Assume that the representation (\ref{eq_sumr}) has the form 
$f_{\phi} \equiv \alpha \otimes \beta$, where 
$\alpha \in \cl M_X$ and $\beta \in M_n$.
In this case, $\phi$ is given by 
$$\phi(p_{x,y}) = \alpha_{x,y}\beta, \ \ \ x,y\in X.$$
In particular, if $P_{\theta}$ is the permutation unitary corresponding to the permutation $\theta$ on $X$, 
and $f_{\phi} \equiv P_{\theta} \otimes \beta$, where $\beta \in M_n$, then 
$$\phi(p_{x,y}) = 
\begin{cases}
\beta & \text{if } \theta(x) = y\\
0 & \text{otherwise,}
\end{cases} 
\ \ \ x,y\in X.$$

Returning to the representation (\ref{eq_sumr}), 
use Birkhoff's Theorem to write
$\alpha_l = \sum_{\theta} \lambda_{\theta}^{(l)} P_{\theta}$, 
where the summation is over the permutation group of $X$, 
the coefficients $\lambda_{\theta}^{(l)}$ are non-negative. 
Thus, 
\begin{equation}\label{eq_sumr2}
f_{\phi} \equiv \sum_{\theta} P_{\theta}\otimes\gamma_{\theta}, 
\end{equation}
where $\gamma_{\theta}\in M_n^+$ and the summation is over the permutation group of $X$.
By the previous paragraph, 
$$E_{x,y} = \sum\{\gamma_{\theta} : \theta(x) = y\}, \ \ \ x,a\in X.$$
Now
\cite[Theorem 12 and Remark 7]{dlcdn} implies that $(\phi(p_{x,a}))_{x,a}$ is locally dilatable, 
after noticing that the
matrix convex hull of the set denoted $\mathcal{CP}^{(|X|)}$ therein coincides with the 
locally dilatable magic quantum squares over $M_n$. 
The converse direction follows by reversing the given arguments. 

We now relax the assumption on the finite dimensionality of $H$. For simplicity, we 
consider only the case where $H$ is separable. 
Fix a sequence $(Q_n)_{n\in \bb{N}}$ of projections of finite rank such that $Q_n\to_{n\to \infty} I$ 
in the strong operator topology. 
Assuming that $E$ is locally dilatable, so is $(I_X\otimes Q_n)E(I_X\otimes Q_n)$ for every $n\in \bb{N}$ and hence, by the assumption, 
the map $\phi_n : {\rm OMIN}(\cl P_X)\to \cl B(Q_nH)$,
given by
$\phi_n(p_{x,a}) = Q_nE_{x,a}Q_n$, $x,a\in X$, $i\in \bb{I}$, is completely positive. 
Since $\phi(u) = \lim_{n\to\infty} \phi_n(u)$, in the weak operator topology, $u\in \cl P_X$, we have that 
$\phi$ is completely positive.

Conversely, assuming that $\phi : {\rm OMIN}(\cl P_X)\to \cl B(H)$ is completely positive, let 
$\phi_n : {\rm OMIN}(\cl P_X)\to \cl B(Q_nH)$ be the (completely positive) map, given by 
$\phi_n(u) = Q_n\phi(u)Q_n$, $u\in \cl P_X$. 
%Set $C := \sum_{x,y\in X}\|E_{x,y}\|$. 
%It is clear that $\|\phi_n\|_{\rm cb}\leq \|\phi\|_{\rm cb}$; therefore, $\|f_{\phi_n}\|\leq \|\phi\|_{\rm cb}$, 
%$n\in \bb{N}$. 
Write $f_{\phi_n} = \sum_{\theta} P_{\theta}\otimes\gamma_{\theta}^{(n)}$, 
where $\gamma_{\theta}^{(n)} \in \cl B(Q_nH)^+$ and the summation is over the permutation group of $X$. 
Let $E_{x,y}^{(n)} = \sum\{\gamma_{\theta}^{(n)} : \theta(x) = y\}$; then $E_{x,y}^{(n)} = Q_nE_{x,y}Q_n$. 
Since $\|E_{x,y}\| \leq 1$ for every $x,y\in X$, we therefore have that 
$\|\gamma_{\theta}^{(n)}\|\leq 1$ for every $n\in \bb{N}$. 
We can now choose successively weak* cluster points of the sequences 
$\left(\gamma_{\theta}^{(n)}\right)_{n\in \bb{N}}$,
and assume that
$$f_{\phi_n}\to f := \sum_{\theta} P_{\theta}\otimes \gamma_{\theta},$$ 
where $\gamma_{\theta} \in \cl B(H)^+$ for every permutation $\theta$ of $X$, 
in the weak* topology of $M_X\otimes \cl B(H)$.
We further have that 
$E_{x,y} = \sum\{\gamma_{\theta} : \theta(x) = y\}$. 
The proof of the implication (a)$\Rightarrow$(b) of \cite[Theorem 12]{dlcdn} implies, after replacing the
identity operator denoted $I_s$ therein with $I_H$, 
that $E$ is locally dilatable.
\end{proof}

%\noindent {\bf Remark. }
%Suppose that the Hilbert space $H$ in Theorem \ref{p_dilate} (ii) is arbitrary, 
%fix a net $(Q_i)_{i\in \bb{I}}$ consisting of projections of finite rank such that $Q_i\to_{i\in \bb{I}} I$ in the strong operator topology. 
%Assuming that $E$ is locally dilatable, so is $(I_X\otimes Q_i)E(I_X\otimes Q_i)$ for every $i\in \bb{I}$ and hence, by the assumption, 
%the map $\phi_i : {\rm OMIN}(\cl P_X)\to \cl B(Q_iH)$, given by
%$\phi_i(p_{x,a}) = Q_iE_{x,a}Q_i$, $x,a\in X$, $i\in \bb{I}$, is completely positive. 
%Since $\phi(u) = \lim_{i\in \bb{I}} \phi_i(u)$, $u\in \cl P_X$, in the weak operator topology, we have that 
%$\phi$ is completely positive.
%We do not know if the converse direction of Theorem \ref{p_dilate} (ii) holds true for a Hilbert space of an arbitrary dimension. 

%%%%%%%%%%%%%%%%%%%%%%%%%%%%%%%%%%%%%%%%%%%%%%%%%
%%%%%%%%%%%%%%%%%%%%%%%%%%%%%%%%%%%%%%%%%%%%%%%%%

\section{Representations of bicorrelations}\label{s_rbc}

In this section, we define the notion of a bicorrelation and obtain representations of the 
different bicorrelation types in terms of operator system tensor products. We will use the main 
operator system tensor products, introduced in \cite{kptt}: the minimal ($\min$), 
the commuting (${\rm c}$), and the maximal ($\max$). 
If $\tau \in \{\min, {\rm c}, \max\}$ and $\phi_i : \cl S_i\to \cl T_i$ are completely positive maps between 
operator systems, $i = 1,2$, we write $\phi_1\otimes_{\tau}\phi_2$ for the corresponding tensor product map 
from $\cl S_1\otimes_{\tau}\cl S_2$ into $\cl T_1\otimes_{\tau}\cl T_2$ 
(note that this map is well-defined by \cite[Theorems 4.6, 5.5. and 6.3]{kptt}).

We fix throughout this section 
finite sets $X$ and $Y$, and let $A = X$ and $B = Y$. 
The symbols $A$ and $B$ will continue to be used for clarity, as needed.

%%%%%%%%%%%%%%%%%%%%%%%%%%%%%%%%%%%%%%%%%%%%%%%%%
%%%%%%%%%%%%%%%%%%%%%%%%%%%%%%%%%%%%%%%%%%%%%%%%%

\subsection{Quantum bicorrelations}\label{s_quantumbc}

If $\Gamma : M_{XY}\to M_{XY}$ is a unital quantum channel then, after the canonical identification 
$M_{XY}^{\rm d} \equiv M_{XY}$, its dual $\Gamma^* : M_{XY}\to M_{XY}$, 
defined via the formula
$$\langle \Gamma^*(\omega),\rho\rangle = \langle \omega,\Gamma(\rho)\rangle 
= \Tr\left(\omega\Gamma(\rho)^{\rm t}\right), \ \ \ \omega,\rho\in M_{XY},$$
is also a (unital) quantum channel.

\begin{definition}\label{d_bic}
A QNS correlation $\Gamma : M_{XY}\to M_{XY}$ is called 
a \emph{QNS bicorrelation} if $\Gamma(I_{XY}) = I_{XY}$ and 
$\Gamma^*$ is a QNS correlation. 
\end{definition}

We let $\cl Q^{\rm bi}_{\rm ns}$ be the set of all QNS bicorrelations. 
We next define different types of QNS bicorrelations, motivated by the analogous definitions of QNS correlation types. 
A QNS bicorrelation $\Gamma : M_{XY}\to M_{XY}$ is \emph{quantum commuting} if 
there exist a Hilbert space $H$, a unit vector $\xi\in H$ and
bistochastic operator matrices $\tilde{E} = (E_{x,x',a,a'})_{x,x',a,a'}$ and $\tilde{F} = (F_{y,y',b,b'})_{y,y',b,b'}$ 
on $H$ with mutually commuting entries, 
such that the Choi matrix of $\Gamma$ coincides with 
\begin{equation}\label{eq_Choima}
\left(\left\langle E_{x,x',a,a'}F_{y,y',b,b'}\xi,\xi \right\rangle\right)_{x,x',a,a'}^{y,y',b,b'}
\end{equation}
(equivalently, relation (\ref{eq_EFp}) holds true).
%\begin{equation}\label{eq_EFpbi}
%\Gamma(\epsilon_{x,x'} \otimes \epsilon_{y,y'}) = \sum_{a,a'\in A} \sum_{b,b'\in B}
%\left\langle E_{x,x',a,a'}F_{y,y'b,b'}\xi,\xi \right\rangle \epsilon_{a,a'} \otimes \epsilon_{b,b'}, 
%\end{equation}
%for all $x,x' \in X$ and all $y,y' \in Y$.
\emph{Quantum} QNS bicorrelations are defined 
similarly, but 
requiring that $H$ has the form $H_A\otimes H_B$, for some finite dimensional Hilbert spaces $H_A$ and $H_B$, and 
$E_{x,x',a,a'} = \tilde{E}_{x,x',a,a'} \otimes I_B$, and 
$F_{y,y',b,b'} = I_A \otimes \tilde{F}_{y,y',b,b'}$, for some bistochastic operator matrices 
$(\tilde{E}_{x,x',a,a'})$ and 
$(\tilde{F}_{y,y',b,b'})$, acting on $H_A$ and $H_B$, respectively. 
\emph{Approximately quantum} QNS bicorrelations are the limits of quantum QNS bicorrelations, while \emph{local} QNS bicorrelations are 
the convex combinations of the form 
$\Gamma = \sum_{i=1}^k \lambda_i \Phi_i \otimes \Psi_i$, where
$\Phi_i : M_X\to M_A$ and $\Psi_i : M_Y\to M_B$ are unital quantum channels, $i = 1,\dots,k$.

For  ${\rm t}\in \{{\rm loc}, {\rm q}, {\rm qa}, {\rm qc}\}$, we let  
$\cl Q_{\rm t}^{\rm bi}$ 
%= \cl Q_{\rm t}\cap \cl Q^{\rm bi}_{\rm ns}$. 
be the set of all QNS bicorrelations of type ${\rm t}$. 

\begin{remark}\label{r_dualal}
\rm 
If ${\rm t}\in \{{\rm loc}, {\rm q}, {\rm qa}, {\rm qc}, {\rm ns}\}$ and $\Gamma\in \cl Q_{\rm t}^{\rm bi}$ then $\Gamma^*\in \cl Q_{\rm t}^{\rm bi}$. 
The claim is part of the definition in the case where ${\rm t} = {\rm ns}$ and straightforward 
in the case where ${\rm t} = {\rm loc}$. 
For the case ${\rm t} = {\rm qc}$, suppose that 
$E = (E_{x,x',a,a'})_{x,x',a,a'}$ and $F = (F_{y,y',b,b'})_{y,y',b,b'}$ are bistochastic operator matrices
with mutually commuting entries, such that 
the Choi matrix of $\Gamma$ coincides with (\ref{eq_Choima}).
Let $\tilde{E}_{a,a',x,x'} := E_{x,x',a,a'}$ and 
$\tilde{F}_{b,b',y,y'} := F_{y,y',b,b'}$, and set 
$\tilde{E} = (\tilde{E}_{a,a',x,x'})_{a,a',x,x'}$ and $\tilde{F} = (\tilde{F}_{b,b',y,y'})_{b,b',y,y'}$.
We have that 
$\tilde{E} = \sum_{x,x',a,a'} \epsilon_{a,a'}\otimes \epsilon_{x,x'}\otimes E_{x,x',a,a'}$ and hence 
$\tilde{E}$ is a unitary conjugation of $E$, implying that $\tilde{E}\geq 0$; similarly, $\tilde{F}\geq 0$. 
The claim now follows from the fact that
the Choi matrix of $\Gamma^*$ is equal to 
$\left(\left\langle \tilde{E}_{a,a',x,x'}\tilde{F}_{b,b',y,y'}\xi,\xi \right\rangle\right)_{a,a',x,x'}^{b,b',y,y'}$.
The case ${\rm t} = {\rm q}$ is analogous, while ${\rm t} = {\rm qa}$ is a consequence of the continuity of
taking the dual channel.
\end{remark}

\begin{remark}\label{r_localbi}
\rm 
Suppose that $\Gamma\in \cl Q_{\rm loc}$ is unital. Write 
\begin{equation}\label{eq_Glocb}
\Gamma = \sum_{i=1}^k \lambda_i \Phi_i\otimes \Psi_i
\end{equation}
as a convex combination, 
where $\Phi_i : M_X\to M_X$ and $\Psi_i : M_Y\to M_Y$ are quantum channels, $i = 1,\dots, k$. 
We have that 
$\sum_{i=1}^k \lambda_i \Phi_i(I_X)\otimes \Psi_i(I_Y) = I_{XY}$.
It follows that $0\leq \Phi_i(I_X)\otimes \Psi_i(I_Y) \leq I_{XY}$;  
since $I_{XY}$ is an extreme point in the unit ball of $M_{XY}^+$, we have that 
$\Phi_i(I_X)\otimes \Psi_i(I_Y) = I_{XY}$, for every $i = 1,\dots,k$.
Thus, there exist $c_i > 0$ such that 
$\Phi_i(I_X) = c_i I_X$ and $\Psi_i(I_Y) = \frac{1}{c_i}I_Y$, for all $i = 1,\dots,k$. 
Replacing $\Phi_i$ (resp. $\Psi_i$) with $\frac{1}{c_i}\Phi_i$ 
(resp. $c_i \Psi_i$), we conclude that the representation (\ref{eq_Glocb}) can be chosen with the 
property that $\Phi_i$ and $\Psi_i$ are unital quantum channels, $i = 1,\dots,k$, that is, 
$\Gamma$ is automatically a local bicorrelation.
\end{remark}

We write $f_{y,y',b,b'}$ (resp. $\tilde{f}_{y,y',b,b'}$), $y,y',b,b'\in Y$, 
for the canonical generators of the operator system $\cl T_Y$ (resp. $\cl T_{Y,B}$). 
If $s$ is a linear functional on $\cl T_{X} \otimes \cl T_{Y}$ or on $\cl C_{X} \otimes \cl C_{Y}$, we write 
$\Gamma_s : M_{XY}\to M_{XY}$ for the linear map, given by 
\begin{equation}\label{eq_gammas}
\Gamma_s\left(\epsilon_{x,x'}\otimes \epsilon_{y,y'}\right)
= \sum_{a,a'\in X}\sum_{b,b'\in Y} s\left(e_{x,x',a,a'} \otimes f_{y,y',b,b'}\right) \epsilon_{a,a'}\otimes \epsilon_{b,b'}.
\end{equation}
We note that $\Gamma_s^*$ is given by the identities 
\begin{equation}\label{eq_gammas3}
\Gamma_s^*\left(\epsilon_{a,a'}\otimes \epsilon_{b,b'}\right)
= \sum_{x,x'\in X}\sum_{y,y'\in Y} s\left(e_{x,x',a,a'} \otimes f_{y,y',b,b'}\right) \epsilon_{x,x'}\otimes \epsilon_{y,y'}.
\end{equation}
Clearly, the correspondence $s\to \Gamma_s$ is a linear map 
from the vector space dual $(\cl T_{X,A} \otimes \cl T_{Y,B})^{\rm d}$ 
into the space 
$\cl L(M_{XY})$ of all linear transformations on $M_{XY}$.

\begin{theorem}\label{th_bicmax}
Let $X$ and $Y$ be finite sets and $\Gamma : M_{XY}\to M_{XY}$ be a linear map. The following are equivalent:
\begin{itemize}
\item[(i)] $\Gamma$ is a QNS bicorrelation;

\item[(ii)] there exists a state $s : \cl T_{X} \otimes_{\max} \cl T_{Y} \to \bb{C}$ such that $\Gamma = \Gamma_s$.
\end{itemize}
\end{theorem}

\begin{proof}
(ii)$\Rightarrow$(i)
Suppose that $s$ is a state of $\cl T_X\otimes_{\max} \cl T_Y$ such that $\Gamma = \Gamma_s$, 
and let $\tilde{s} = s\circ (q_X\otimes_{\max} q_Y)$, where $q_X : \cl T_{X,A}\to \cl T_X$
(resp. $q_Y : \cl T_{Y,B}\to \cl T_Y$) is the quotient map (see the paragraph of equation (\ref{eq_JXd}));
we have that $\tilde{s}$ is a state of 
$\cl T_{X,A}\otimes_{\max} \cl T_{Y,B}$. Since $\Gamma = \Gamma_{\tilde{s}}$, by \cite[Theorem 6.2]{tt-QNS},
$\Gamma\in \cl Q_{\rm ns}$. 
In addition, 
\begin{eqnarray*}
\Gamma(I_{XY}) 
& = & 
\sum_{x,a,a'\in X}
\sum_{y,b,b'\in Y} s(e_{x,x,a,a'}\otimes f_{y,y,b,b'}) \epsilon_{a,a'}\otimes 
\epsilon_{b,b'}\\
& = & \sum_{a,a'\in X} \sum_{b,b'\in Y}\delta_{a,a'}\delta_{b,b'} \epsilon_{a,a'}\otimes 
\epsilon_{b,b'} = I_{XY}.
\end{eqnarray*}
We verify that $\Gamma^*$ is no-signalling: for any $\omega_X = (\lambda_{a,a'})_{a,a'} \in M_X$ and any $\omega_Y = (\mu_{b,b'})_{b,b'}\in M_Y$ with ${\rm Tr}(\omega_Y) = 0$, by (\ref{eq_gammas3}) we have

\begin{eqnarray*}
& & {\rm Tr}\mbox{}_{Y} \Gamma^*(\omega_X\otimes\omega_Y)\\
& = & 
{\rm Tr}\mbox{}_{Y} \sum_{x,x',a,a'\in X} \sum_{y,y',b,b'\in Y} 
\lambda_{a,a'} \mu_{b,b'} s(e_{x,x',a,a'}\otimes f_{y,y',b,b'}) \epsilon_{x,x'}\otimes \epsilon_{y,y'}\\
& = & 
\sum_{x,x',a,a'\in X} \sum_{b,b'\in Y} 
\lambda_{a,a'} \mu_{b,b'} \sum_{y\in Y} s(e_{x,x',a,a'}\otimes f_{y,y,b,b'}) \epsilon_{x,x'}\\
& = & 
\sum_{x,x',a,a'\in X} \sum_{b,b'\in Y} 
\lambda_{a,a'} \mu_{b,b'} \delta_{b,b'} s(e_{x,x',a,a'}\otimes 1) \epsilon_{x,x'}\\
& = & 
\left(\sum_{b\in Y} \mu_{b,b}\right) \sum_{x,x',a,a'\in X} \lambda_{a,a'} s(e_{x,x',a,a'}\otimes 1) \epsilon_{x,x'} = 0.
\end{eqnarray*}
Similarly, if $\omega_X\in M_X$ has trace zero and $\omega_Y\in M_Y$ is arbitrary then 
${\rm Tr}\mbox{}_{X} \Gamma^*(\omega_X\otimes\omega_Y) = 0$ and hence $\Gamma^*$ is no-signalling.

(i)$\Rightarrow$(ii)
Let $C = \left(C_{x,x',y,y'}^{a,a',b,b'}\right)_{x,x',y,y'}^{a,a',b,b'}$ be the Choi matrix of $\Gamma$; thus, 
the entries of $C$ are given by
$$C_{x,x',y,y'}^{a,a',b,b'} = \left\langle \Gamma(\epsilon_{x,x'}\otimes \epsilon_{y,y'}),
\epsilon_{a,a'}\otimes \epsilon_{b,b'}\right\rangle.$$
By \cite[Theorem 6.2]{tt-QNS}, $C\in (\cl L_{X,A}\otimes_{\min}\cl L_{Y,B})_+$. 

Let $\tilde C=\left(\tilde C_{a,a',b,b'}^{x,x',y,y'}\right)_{a,a',b,b'}^{x,x',y,y'}$ be the Choi matrix of $\Gamma^*$. As both $\Gamma$ and $\Gamma^*$ are no-signalling,
there exists scalars $\tilde{c}_{y,y'}^{b,b'}$, 
$\tilde{d}_{x,x'}^{a,a'}$, $c_{y,y'}^{b,b'}$ and 
$d_{x,x'}^{a,a'}$, such that 
$$\sum_{x\in X} \tilde C_{a,a',b,b'}^{x,x,y,y'} = \delta_{a,a'}\tilde{c}_{y,y'}^{b,b'}, \ \ \ y,y',b,b'\in Y,$$
$$\sum_{y\in Y} \tilde C_{a,a',b,b'}^{x,x',y,y} = \delta_{b,b'}\tilde{d}_{x,x'}^{a,a'}, \ \ \ x,x',a,a'\in X,$$
$$\sum_{a\in X} C_{a,a,b,b'}^{x,x',y,y'}=\delta_{x,x'} c_{y,y'}^{b,b'}, \ \ \ y,y',b,b'\in Y$$ 
and 
$$\sum_{b\in Y} C_{a,a',b,b}^{x,x',y,y'}=\delta_{y,y'} d_{x,x'}^{a,a'}, \ \ \ x,x',a,a'\in X.$$
Observe that the equalities $\tilde C_{a,a',b,b'}^{x',x,y,y'}=C_{x,x',y,y'}^{a,a',b,b'}$ hold. 
The relations, together with Remark \ref{r_c1c2}, 
now imply that $L_{\omega}(C)\in\cl L_X$ and $L_{\omega'}(C)\in\cl L_Y$ for all $\omega\in M_{YY}$ and all $\omega'\in M_{XX}$ (recall that $L_{\sigma}$ denotes the slice map along a functional $\sigma$). 
Thus, $C\in (\cl L_X\otimes_{\min}\cl L_Y)_+$. 
Statement (ii) now follows from the canonical identification 
$(\cl T_{X} \otimes_{\max} \cl T_{Y})^{\rm d}\cong \cl L_{X} \otimes_{\min} \cl L_{Y}.$
\end{proof}

\begin{theorem}\label{th_bicqc}
Let $X$ and $Y$ be finite sets and $\Gamma : M_{XY}\to M_{XY}$ be a linear map. The following are equivalent:
\begin{itemize}
\item[(i)] $\Gamma\in \cl Q_{\rm qc}^{\rm bi}$;
\item[(ii)] there exists a state $s : \cl T_{X} \otimes_{\rm c} \cl T_{Y} \to \bb{C}$ such that $\Gamma = \Gamma_s$;
\item[(iii)] there exists a state $s : \cl C_X \otimes_{\max} \cl C_{Y} \to \bb{C}$ such that 
$\Gamma = \Gamma_s$.
\end{itemize}
\end{theorem}

\begin{proof}
By Theorem \ref{th_ucp} and \cite[Theorem 6.4]{kptt}, 
$\cl T_{X} \otimes_{\rm c} \cl T_{Y} \subseteq \cl C_X \otimes_{\max} \cl C_{Y}$ completely 
order isomorphically and hence, by Krein's Extension Theorem, (ii) and (iii) are equivalent. 

(i)$\Rightarrow$(iii) 
follows from the universal property of $\cl C_X$ detailed in Theorem \ref{th_ucp} and arguments, similar to the ones in 
\cite[Theorem 6.3]{tt-QNS}. 

(iii)$\Rightarrow$(i)
The GNS representation of $s$ and the universal property of the maximal C*-algebraic tensor product
yield *-representations $\pi_X : \cl C_X\to \cl B(H)$ and $\pi_Y : \cl C_Y\to \cl B(H)$ with commuting ranges, 
and a unit vector $\xi\in H$, such that 
$s(u\otimes v) = \langle \pi_X(u)\pi_Y(v)\xi,\xi\rangle$, $u\in \cl C_X$,  $v\in \cl C_{Y}$. 
The claim follows by setting $E_{x,x',a,a'} = \pi_X(e_{x,x',a,a'})$ and $F_{x,x',a,a'} = \pi_Y(f_{y,y',b,b'})$,
and appealing to Theorem \ref{th_ucp}. 
\end{proof}

\begin{theorem}\label{th_bicqa}
Let $X$ and $Y$ be finite sets and $\Gamma : M_{XY}\to M_{XY}$ be a linear map. The following are equivalent:
\begin{itemize}
\item[(i)] $\Gamma\in \cl Q_{\rm qa}^{\rm bi}$;
%\item[(ii)] $\Gamma\in \overline{\cl Q_{\rm q}^{\rm bi}}$;
\item[(ii)] there exists a state $s : \cl T_{X} \otimes_{\min} \cl T_{Y} \to \bb{C}$ such that $\Gamma = \Gamma_s$;
\item[(iii)] there exists a state $s : \cl C_{X}\otimes_{\min}\cl C_{Y}\to \bb{C}$ such that 
$\Gamma = \Gamma_s$.
\end{itemize}
%Moreover, if $\Gamma\in \cl Q_{\rm qa}^{\rm bi}$ then $\Gamma^*\in \cl Q_{\rm qa}^{\rm bi}$. 
\end{theorem}

\begin{proof}
(ii)$\Leftrightarrow$(iii) follows from the injectivity of the minimal tensor product. 

(i)$\Rightarrow$(iii) 
Given $\varepsilon > 0$, 
let $E$ and $F$ be bistochastic operator matrices acting on finite dimensional Hilbert spaces 
$H_X$ and $H_Y$, respectively, 
and $\xi \in H_X\otimes H_Y$ be a unit vector, such that 
$$\left|\left\langle\Gamma(\epsilon_{x,x'}\otimes \epsilon_{y,y'}),
\epsilon_{a,a'}\otimes \epsilon_{b,b'}\right\rangle
- \left\langle \left(E_{x,x',a,a'}\otimes F_{y,y',b,b'}\right)\xi,\xi\right\rangle\right| < \varepsilon,$$
for all $x,x',a,a'\in X$, $y,y',b,b'\in Y$.
By Lemma \ref{l_rr}, there exists a *-representation $\pi_X$ (resp. $\pi_Y$) 
of $\cl C_{X}$ (resp. $\cl C_{Y}$) on $H_X$ (resp. $H_Y$) such that
$E_{x,x',a,a'} = \pi_X(e_{x,x',a,a'})$ (resp. $F_{y,y',b,b'} = \pi_Y(f_{y,y',b,b'})$), 
$x,x',a,a'\in X$ (resp. $y,y',b,b'\in Y$).
Let $s_\varepsilon$ be the state on $\cl C_{X}\otimes_{\min}\cl C_{Y}$ given by
$$s_\varepsilon\left(u\otimes v\right) = \left\langle \left(\pi_X(u)\otimes \pi_Y(v)\right)\xi,\xi\right\rangle,$$
and $s$ be a cluster point of the sequence $\{s_{1/n}\}_n$ in the weak* topology. Then
\begin{eqnarray*}
s\left(e_{x,x',a,a'}\otimes f_{y,y',b,b'}\right)
& =& 
\lim_{n\to\infty}s_{1/n}\left(e_{x,x',a,a'}\otimes f_{y,y',b,b'}\right)\\
& =& 
\left\langle\Gamma(e_{x} e_{x'}^*\otimes e_{y}e_{y'}^*), e_{a}e_{a'}\otimes e_{b}e_{b'}^* \right\rangle,
\end{eqnarray*}
giving $\Gamma = \Gamma_s$.

(iii)$\Rightarrow$(i) 
Let $s$ be a state satisfying (iv) and $\varepsilon > 0$. 
By \cite[Corollary 4.3.10]{kadison-ringrose}, 
there exist faithful *-representations
$\pi_X : \cl C_{X}\to\cl B(H_X)$ and $\pi_Y : \cl C_{Y}\to\cl B(H_Y)$, unit vectors 
$\xi_1,\ldots,\xi_n\in H_X\otimes H_Y$ and positive scalars $\lambda_1,\ldots,\lambda_n$, with $\sum_{i=1}^n\lambda_i=1$ such that
$$\left|s(e_{x,x',a,a'}\otimes f_{y,y',b,b'})
- \sum_{i=1}^n\lambda_i\left\langle \left(\pi_X(e_{x,x',a,a'})\otimes \pi_Y(f_{y,y',b,b'})\right)\xi_i,\xi_i\right\rangle\right|
< \varepsilon,$$
for all $x,x',a,a'\in X$, $y,y',b,b'\in Y$.
Let $\xi = \oplus_{i=1}^n\sqrt{\lambda_i}\xi_i\in \mathbb C^{n}\otimes(H_X\otimes H_Y)$; then $\|\xi\| = 1$.  
Set $E_{x,x',a,a'} = I_n \otimes \pi_X(e_{x,x',a,a'})$ and $F_{y,y',b,b'} = \pi_Y(f_{y,y',b,b'})$. 
Then $(E_{x,x'a,a'})_{x,x',a,a'}$ (resp. $(F_{y,y',b,b'})_{y,y',b,b'}$)
is a bistochastic operator matrix on $\bb{C}^n\otimes H_X$ (resp. $H_Y$), and 
$$\left|s\left(e_{x,x',a,a'}\otimes f_{y,y',b,b'}\right) 
- \left\langle E_{x,x'a,a'}\otimes F_{y,y',b,b'}\xi,\xi\right\rangle\right| < \varepsilon.$$
Let 
$(P_{\alpha})_{\alpha}$ (resp. $(Q_{\beta})_{\beta}$) be a net of finite rank projections on 
$H_X$ (resp. $H_B$), converging to the identity in the strong operator topology. 
Set $H_{\alpha} = P_{\alpha}H_A$ (resp. $K_{\beta} = Q_{\beta}H_B$), 
$E_{\alpha} = (I\otimes P_{\alpha})E(I\otimes P_{\alpha})$ 
(resp. $F_{\beta} = (I\otimes Q_{\beta})F(I\otimes Q_{\beta})$), 
and $\xi_{\alpha,\beta} = \frac{1}{\|(P_{\alpha}\otimes Q_{\beta})\xi\|}(P_{\alpha}\otimes Q_{\beta})\xi$
(note that $\xi_{\alpha,\beta}$ is eventually well-defined). 
Then $E_{\alpha}$ and $F_{\beta}$ are bistochastic operator matrices acting on $P_{\alpha}H$ and $Q_{\beta}K$, respectively, 
and the QNS correlation associated with the 
the triple $(E_{\alpha},F_{\beta},\xi_{\alpha,\beta})$ is a quantum bicorrelation. 
\end{proof}

\begin{remark}\label{r_qnsti}
\rm 
By Remark \ref{r_dualal}, 
\begin{equation}\label{eq_Qqcnew}
\cl Q_{\rm qc}^{\rm bi}\subseteq \cl Q_{\rm qc} \cap \cl Q_{\rm ns}^{\rm bi}. 
\end{equation}
We do not know if equality holds in (\ref{eq_Qqcnew}). 
The problem reduces to a question about the equality of canonical operator system structures. 
%Let $\cl S$ (resp. $\cl T$) be a finite dimensional operator system,
%$\cl I\subseteq\cl S$ (resp. $\cl J\subseteq\cl T$) be a kernel, and $q_{\cl S} : \cl S\to \cl S/\cl I$
%(resp. $q_{\cl T} : \cl T\to \cl T/\cl J$) be the quotient map. It is easy to verify that, 
%the subspace $\cl N := \cl I\otimes\cl T + \cl S\otimes\cl J$ is the kernel in 
%$\cl S \otimes_{\rm c} \cl T$ of the completely positive map $q_{\cl S}\otimes_{\rm c} q_{\cl T}$.
%It follows that,  
Indeed, it is not difficult to verify that the subspace
$\cl J_{XY} := \cl T_{X,A}\otimes\cl J_Y + \cl J_X \otimes \cl T_{Y,B}$
of the operator system $\cl T_{X,A}\otimes_{\rm c} \cl T_{Y,B}$ is a kernel, 
%$(\cl T_{X,A}\otimes_{\rm c} \cl T_{Y,B})/\cl J_{XY}$ is well-defined (see the definition of 
%$\cl J_X$ and $\cl J_Y$ in (\ref{eq_JXd})). 
and that the states on  
$(\cl T_{X,A}\otimes_{\rm c} \cl T_{Y,B})/\cl J_{XY}$ correspond precisely to the elements of 
$\cl Q_{\rm qc} \cap \cl Q_{\rm ns}^{\rm bi}$. 
However, 
while there is a canonical bijective unital completely positive map $(\cl T_{X,A}\otimes_{\rm c} \cl T_{Y,B})/\cl J_{XY}\to \cl T_X\otimes_{\rm c}\cl T_Y$, it is unclear whether its inverse is completely positive. 
If this is the case then Theorem \ref{th_bicqc} will imply the reverse inclusion in (\ref{eq_Qqcnew}). 
\end{remark}

%%%%%%%%%%%%%%%%%%%%%%%%%%%%%%%%%%%%%%%%%%%%%%%%%
%%%%%%%%%%%%%%%%%%%%%%%%%%%%%%%%%%%%%%%%%%%%%%%%%

\subsection{Classical bicorrelations}\label{s_classical}

In this subsection, we consider a class of correlations that constitute a natural classical counterpart 
of the quantum bicorrelations defined in Subsection \ref{s_quantumbc}. We fix finite sets $X$ and $Y$, and set $A = X$ and $B = Y$.

\begin{definition}\label{d_cbic}
An NS correlation $p = \{(p(a,b|x,y))_{a,b} : (x,y)\in X \times Y\}$ over the quadruple $(X,Y,X,Y)$ is called 
an \emph{NS bicorrelation} if the family 
$$p^* := \left\{(p(a,b|x,y))_{x,y} : (a,b)\in X \times Y\right\}$$ 
is an NS correlation. 
\end{definition}

We let $\Delta : M_{XY}\to \cl D_{XY}$ be the canonical diagonal expectation.
Given an NS correlation $p$ over $(X,Y,X,Y)$, we let 
$\cl E_p : \cl D_{XY}\to \cl D_{XY}$ be the (classical) information channel, given by 
$$\cl E_p(\epsilon_{x,x}\otimes\epsilon_{y,y}) = \sum_{a\in X}\sum_{b\in Y}
p(a,b|x,y) \epsilon_{a,a}\otimes\epsilon_{b,b}.$$
Further, for a given classical information channel $\cl E : \cl D_{XY}\to \cl D_{XY}$, let 
$\Gamma_{\cl E} : M_{XY}\to M_{XY}$
be the quantum channel, given by 
$$\Gamma_{\cl E}(\omega) = (\cl E\circ\Delta)(\omega), \ \ \ \omega\in M_{XY},$$
and set $\Gamma_p = \Gamma_{\cl E_p}$ for brevity. 
In the reverse direction, given a quantum channel $\Gamma : M_{XY}\to M_{XY}$, let $\cl E_{\Gamma} : \cl D_{XY}\to \cl D_{XY}$ be the classical information channel, defined by letting
$\cl E_{\Gamma}(\omega) = (\Delta\circ\Gamma)(\omega)$, $\omega\in \cl D_{XY}$. 
We note the relation $\cl E_{\Gamma_{\cl E}} = \cl E$.

\begin{proposition}\label{l_pstar}
Let $p$ be an NS bicorrelation over $(X,Y,X,Y)$. Then $\Gamma_{p^*} = \Gamma_p^*$.
Thus, if $p\in \cl C_{\rm ns}^{\rm bi}$ then 
$\Gamma_p\in \cl Q_{\rm ns}^{\rm bi}$. 
\end{proposition}

\begin{proof}
For $x,a\in X$ and $y,b\in Y$, we have 
\begin{eqnarray*}
\left\langle \cl E_p^*(\epsilon_{x,x}\otimes\epsilon_{y,y}),\epsilon_{a,a}\otimes\epsilon_{b,b}\right\rangle
& = & 
\left\langle \epsilon_{x,x}\otimes\epsilon_{y,y},\cl E_p(\epsilon_{a,a}\otimes\epsilon_{b,b})\right\rangle\\
& = & 
p(x,y|a,b) 
= 
p^*(a,b|x,y)\\
& = & 
\left\langle \cl E_{p^*}(\epsilon_{x,x}\otimes\epsilon_{y,y}),\epsilon_{a,a}\otimes\epsilon_{b,b}\right\rangle,
\end{eqnarray*}
implying that 
$\cl E_{p^*} = \cl E_p^*$. 
For $\omega_1,\omega_2\in M_{XY}$, we thus have 
\begin{eqnarray*}
\left\langle \Gamma_p^*(\omega_1),\omega_2\right\rangle
& = & 
\left\langle \omega_1,(\cl E_p\circ \Delta)(\omega_2)\right\rangle
= 
\left\langle \omega_1,(\Delta\circ \cl E_p\circ \Delta)(\omega_2)\right\rangle\\
& = &  
\left\langle \Delta(\omega_1),(\cl E_p\circ \Delta)(\omega_2)\right\rangle
=  
\left\langle (\cl E_p^*\circ \Delta)(\omega_1),\Delta(\omega_2)\right\rangle\\
& = & 
\left\langle (\cl E_{p^*}\circ \Delta)(\omega_1),\Delta(\omega_2)\right\rangle
= 
\left\langle \Gamma_{p^*}(\omega_1),\Delta(\omega_2)\right\rangle\\
& = &  
\left\langle \Gamma_{p^*}(\omega_1),\omega_2\right\rangle,
\end{eqnarray*}
completing the proof. 
\end{proof}

For ${\rm t}\in \{{\rm loc}, {\rm q}, {\rm qa}, {\rm qc}\}$, let 
$$\cl C_{\rm t}^{\rm bi} = \left\{p\in \cl C_{\rm ns}^{\rm bi} : \Gamma_p\in \cl Q_{\rm t}^{\rm bi}\right\}.$$
It is straightforward to verify that 
an NS bicorrelation $p$ over $(X,Y,X,Y)$ belongs to $\cl C_{\rm qc}^{\rm bi}$ precisely when 
there exist a Hilbert space $H$, a unit vector $\xi\in H$ and 
quantum magic squares $(E_{x,a})_{x,a\in X}$ and $(F_{y,b})_{y,b\in Y}$ with 
commuting entries, such that 
\begin{equation}\label{eq_bicore}
p(a,b|x,y) = \langle E_{x,a} F_{y,b}\xi,\xi\rangle, \ \ \ x,a\in X, y,b\in Y.
\end{equation}
Similarly, 
$p\in \cl C_{\rm q}^{\rm bi}$ precisely when
the representation (\ref{eq_bicore}) is achieved for $H = H_A\otimes H_B$, where $H_A$ and $H_B$ are 
finite dimensional Hilbert spaces, $E_{x,a} = E_{x,a}' \otimes I_{H_B}$ and 
$F_{y,b} = I_{H_A}\otimes F_{y,b}'$, $x,a\in X$, $y,b\in Y$.  
Finally, $p\in \cl C_{\rm loc}^{\rm bi}$ precisely when $p$ is 
the convex combinations of correlations of the form 
$p^{(1)}(a|x) p^{(2)}(b|y)$, where $(p^{(1)}(a|x))_{x,a}$ and 
$(p^{(2)}(b|y))_{y,b}$ are (scalar) bistochastic matrices.

For a linear functional $s : \cl S_X\otimes \cl S_Y\to \bb{C}$, let 
$p_s : X\times Y\times X\times Y\to \bb{C}$ be the function given by 
$$p_s(a,b|x,y) = s(e_{x,a}\otimes e_{y,b}), \ \ \ x,a\in X, y,b\in Y.$$

\begin{theorem}\label{th_classSX}
Let $X$ and $Y$ be finite sets and 
$p$
be an NS correlation over $(X,Y,X,Y)$. 
Consider the statements
\begin{itemize}
\item[(i)] $p$ is an NS bicorrelation;

\item[(ii)] there exists a state $s : \cl S_{X} \otimes_{\max} \cl S_{Y} \to \bb{C}$ such that $p = p_s$;
\end{itemize}

\begin{itemize}
\item[(i')] $p\in \cl C_{\rm qc}^{\rm bi}$;
\item[(ii')] there exists a state $s : \cl S_{X} \otimes_{\rm c} \cl S_{Y} \to \bb{C}$ such that $p = p_s$;
\end{itemize}

\begin{itemize}
\item[(i'')] $p\in \cl C_{\rm qa}^{\rm bi}$;
\item[(ii'')] there exists a state $s : \cl S_{X} \otimes_{\min} \cl S_{Y} \to \bb{C}$ such that $p = p_s$.
\end{itemize}
Then (i)$\Leftrightarrow$(ii), (i')$\Leftrightarrow$(ii') and (i'')$\Leftrightarrow$(ii'').
\end{theorem}

\begin{proof}
(i)$\Leftrightarrow$(ii)
By Proposition \ref{p_SXPXd} and \cite[Proposition 1.16]{fp}, 
the states of the maximal tensor product $\cl S_{X} \otimes_{\max} \cl S_{Y}$ correspond in a canonical fashion 
to the elements of $\cl M_X\otimes_{\min} \cl M_Y$. 
The proof of the claim can now be completed using a 
straightforward modification of the proof of Theorem \ref{th_bicmax}. 

(i')$\Rightarrow$(ii') 
Write 
$\iota_X : \cl S_X\to \cl T_X$ and $\iota_Y : \cl S_Y\to \cl T_Y$ for the inclusion maps and 
let $p\in \cl C_{\rm qc}^{\rm bi}$. 
By Theorem \ref{th_bicqc}, there exists a state $s : \cl T_X\otimes_{\rm c}\cl T_Y\to \bb{C}$
such that $\Gamma_p = \Gamma_s$. Let 
$\tilde{s} = s\circ(\iota_X\otimes\iota_Y)$; then $\tilde{s}$ is a state on $\cl S_X\otimes_{\rm c}\cl S_Y$
for which $p = p_{\tilde{s}}$.

(ii')$\Rightarrow$(i') 
Let $s : \cl S_{X} \otimes_{\rm c} \cl S_{Y} \to \bb{C}$ be such that $p = p_s$, and  
let $\beta_X : \cl T_X\to \cl S_X$ (resp. $\beta_Y : \cl T_Y\to \cl S_Y$) be the quotient map, as defined in the 
proof of Proposition \ref{p_SXPXd}. We have that 
$$\tilde{s} := s \circ (\beta_X\otimes\beta_Y) : \cl T_X\otimes_{\rm c} \cl T_Y \to \bb{C}$$
is a state. 
By Theorem \ref{th_bicqc}, the map $\Gamma_{\tilde{s}} : M_{XY}\to M_{XY}$, 
corresponding to $\tilde{s}$ via (\ref{eq_gammas}), 
is a quantum commuting QNS bicorrelation. 
Since $\Gamma_{\tilde{s}} = \Gamma_p$, we have that $p\in \cl C_{\rm qc}^{\rm bi}$.

(i'')$\Leftrightarrow$(ii'') follows in a similar way as the equivalence (i')$\Leftrightarrow$(ii'), using 
Theorem \ref{th_bicqa} in the place of Theorem \ref{th_bicqc}. 
\end{proof}

%%%%%%%%%%%%%%%%%%%%%%%%%%%%%%%%%%%%%%%%%%%%%%%%
%%%%%%%%%%%%%%%%%%%%%%%%%%%%%%%%%%%%%%%%%%%%%%%%

\section{Concurrent bicorrelations}\label{s_concbic}

Throughout the section, let $X$ be a finite set and $Y = A = B = X$.
Let $J_X = \frac{1}{|X|} \sum_{x,y\in X} \nep_{x,y}\otimes \nep_{x,y}$ be the 
canonical maximally entangled state in $M_{XX}$. 
We specialise the definition of a concurrent QNS correlation from \cite{bhtt}:

\begin{definition}\label{d_biconc}
A QNS bicorrelation 
$\Gamma : M_{XX}\to M_{AA}$ is called \emph{concurrent} if 
$\Gamma(J_X) = J_A$.
\end{definition}

For ${\rm t}\in \{{\rm loc}, {\rm q}, {\rm qa}, {\rm qc}, {\rm ns}\}$, we let $\cl Q_{\rm t}^{\rm bic}$
be the set of all concurrent bicorrelations that belong to $\cl Q_{\rm t}^{\rm bi}$.

\begin{remark}\label{r_biconc}
\rm 
Note that if $\Gamma\in \cl Q_{\rm ns}^{\rm bic}$, then $\Gamma^* \in  \cl Q_{\rm ns}^{\rm bic}$ as well.  Indeed, since  $\Gamma$ is unital, its dual map $\Gamma^* : M_{AA}\to M_{XX}$
is trace-preserving; thus, $\Tr(\Gamma^*(J_A)) = 1$. 
Therefore
\begin{eqnarray*}
1 
=  
\left\langle \Gamma(J_X),J_A\right\rangle 
\hspace{-0.2cm}& = &\hspace{-0.2cm} \left\langle J_X,\Gamma^*(J_A)\right\rangle
= \left|\left\langle J_X,\Gamma^*(J_A)\right\rangle\right|\\
& \leq & \hspace{-0.2cm} \|J_X\|_2\|\Gamma^*(J_A)\|_2
\leq  
\|\Gamma^*(J_A)\|_1 = 1.
\end{eqnarray*}
The equality clause in the Cauchy-Schwarz inequality
now implies that $\Gamma^*(J_A)$ is a multiple of $J_X$. If $\Gamma^*(J_A)=\alpha J_X$ for some $\alpha \in \bb{C}$, then $\alpha\langle J_X,J_X\rangle=\langle \Gamma^*(J_A),J_X\rangle=\langle J_A,\Gamma(J_X)\rangle=\langle J_A,J_A\rangle$,  giving $\alpha=1$.
\end{remark}

The universal C*-algebra generated by the entries of a unitary matrix $(\tilde{u}_{a,x})_{a,x\in X}$ (known as the 
\emph{Brown algebra}) was 
first studied by  L. G. Brown \cite{b}. 
We will introduce a subquotient of the Brown algebra, whose traces will be shown to 
represent concurrent bicorrelations of different types. 
First, set 
$$\tilde{u}_{x,x',a,a'} = \tilde{u}_{a,x}^*\tilde{u}_{a',x'}, \ \ \ 
x,x',a,a'\in X,$$
and let $\cl U_{X,A}$ be the 
$C^*$-subalgebra of the Brown algebra, generated by the set $\{\tilde{u}_{x,x',a,a'} : x,x',a,a'\in X\}$.

\begin{lemma}\label{l_bhtt}
If $\pi : \cl U_{X,A}\to \cl B(H)$ is a unital *-representation then there exists a block operator unitary $U = (U_{a,x})_{a,x}$ such that 
$\pi(\tilde{u}_{x,x',a,a'}) = U_{a,x}^*U_{a',x'}$, $x,x',a,a'\in X$. 
\end{lemma}

\begin{proof}
Let $\cl V_{X,A}$ be the universal TRO of an isometry $(v_{a,x})_{a,x}$, 
as defined in \cite[Section 5]{tt-QNS}. In the sequel, we will consider products
$v_{a_1,x_1}^{\varepsilon_1}v_{a_2,x_2}^{\varepsilon_2}\cdots v_{a_k,x_k}^{\varepsilon_k}$, where 
$\varepsilon_i$ is either the empty symbol or $\ast$, and $\varepsilon_i\neq \varepsilon_{i+1}$ for all $i$, as elements of either $\cl V_{X,A}$, $\cl V_{X,A}^*$, $\cl C_{X,A}$ or the left C*-algebra corresponding to the TRO $\cl V_{X,A}$. 
Let $\cl J$ be the closed ideal of $\cl C_{X,A}$, generated by the elements 
$$\sum_{x\in X} \tilde{e}_{y,x,b,a}\tilde{e}_{x,y,a,b} - \tilde{e}_{y,y,b,b}, \ \ \ y, a, b\in X.$$
By \cite[Lemma 4.2]{bhtt}, the map  $\rho : \tilde{e}_{x,x',a,a'}\mapsto \tilde{u}_{x,x',a,a'}$, $x,x',a,a'\in X$ extends to  a surjective *-homomorphism $\rho : \cl C_{X,A}\to \cl U_{X,A}$ with $\ker\rho = \cl J$.
Let $\pi : \cl U_{X,A}\to \cl B(H)$ be a *-representation. Then 
$\pi\circ \rho : \cl C_{X,A}\to \cl B(H)$ is a *-representation that annihilates $\cl J$. 
By \cite[Lemma 5.1]{tt-QNS}, there exists a block operator isometry $U = (U_{a,x})_{a,x\in X}$, 
where $U_{a,x}\in \cl B(H,K)$ for some Hilbert space $K$, $x,a\in X$, such that 
$(\pi\circ \rho)(\tilde{e}_{x,x',a,a'}) = U_{a,x}^*U_{a',x'}$, $x,x',a,a'\in X$.

By the definition of $\cl V_{X,A}$, the operator matrix $U$ gives rise to a canonical ternary representation 
$\theta_U : \cl V_{X,A}\to \cl B(H,K)$. Without loss of generality, 
we can assume that $K = \overline{{\rm span} (\theta_U(\cl V_{X,A})H)}$.
The fact that $(\pi\circ \rho)(\cl J) = \{0\}$ now implies that 
\begin{equation}\label{eq_sqrag}
U_{b,y}^*\left(I - \sum_{x\in X} U_{a,x} U_{a,x}^*\right) U_{b,y} = 0, \ \ \ y, a, b\in X.
\end{equation}
Since $UU^*\leq I$, we have that $I - \sum_{x\in X} U_{a,x} U_{a,x}^*\geq 0$, and hence
(\ref{eq_sqrag}) reads
$$\left(I - \sum_{x\in X} U_{a,x} U_{a,x}^*\right)^{1/2} U_{b,y} = 0, \ \ \ y, a, b\in X,$$
showing further that 
$$\left\langle \left(I - \sum_{x\in X} U_{a,x} U_{a,x}^*\right) T\xi,T\xi\right\rangle = 0, 
\ \ a\in X, \xi\in H, T\in \theta_U(\cl V_{X,A}).$$
By polarisation, we have $\sum_{x\in X} U_{a,x} U_{a,x}^* = I$, $a\in X$.
As $I-UU^*$ is a positive block-diagonal operator with the zero diagonal, $I-UU^*=0$; thus, $U$ is unitary. 
Since $U_{a,x}^*U_{b,y} = \pi(\tilde{u}_{x,y,a,b})$, $x,y,a,b\in X$, the proof is complete. 
\end{proof}

Recall that $\tilde{e}_{x,x',a,a'}$ are the canonical generators of the C*-algebra $\cl C_{X,A}$ (so that the matrix $\left(\tilde{e}_{x,x',a,a'}\right)_{x,x',a,a'}$ is a universal stochastic operator matrix).  Let 
$$\tilde{g}_{y,z,b,c}^{x,x'} = 
\delta_{x,x'} \tilde{e}_{y,z,b,c} - \sum_{a\in X} \tilde{e}_{y,x,b,a} \tilde{e}_{x',z,a,c}$$ 
and
$$\tilde{h}_{y,z,b,c}^{a,a'} = 
\delta_{a,a'} \tilde{e}_{y,z,b,c} - \sum_{x\in X} \tilde{e}_{y,x,b,a} \tilde{e}_{x,z,a',c},$$
and $\tilde{\cl J}_1$ (resp. $\tilde{\cl J}_2$) be the closed ideal of $\cl C_{X,A}$, generated by 
$\tilde{g}_{y,z,b,c}^{x,x'}$ (resp. $\tilde{h}_{y,z,b,c}^{a,a'}$), 
$y, z, b, c, x, x'\in X$ (resp. $y, z, b, c, a, a'\in X$).

\begin{lemma}\label{l_UXAquo}
Up to a canonical *-isomorphism, $\cl C_{X,A}/\tilde{\cl J}_2 \simeq \cl U_{X,A}$.
\end{lemma}

\begin{proof}
Denote by $\tilde{\cl J}_2^0$ the closed ideal of $\cl C_{X,A}$, generated by the elements
$\tilde{h}_{y,y,b,b}^{a,a}$, where $a, b, y\in X$.
It was shown in \cite[Lemma 4.2]{bhtt} that
$$\cl C_{X,A}/\tilde{\cl J}_2^0 \simeq \cl U_{X,A}.$$
Let  
$\rho : \cl C_{X,A}\to \cl B(K)$ be a unital *-representation that annihilates $\tilde{\cl J}_2^0$, with the property that the corresponding induced representation of $\cl C_{X,A}/\tilde{\cl J}_2^0$ is faithful.  
By Lemma \ref{l_bhtt},
there exists a unitary $\tilde{U} = (\tilde{U}_{a,x})_{a,x\in X}$ such that, if 
$\tilde{U}_{x,x',a,a'} = \tilde{U}_{a,x}^*\tilde{U}_{a',x'}$, then 
$$\rho(\tilde{e}_{x,x',a,a'}) = \tilde{U}_{x,x',a,a'}, \ \ \ x,x',a,a'\in X.$$
But then, since $\tilde{U}$ is unitary, 
\begin{eqnarray*}
\rho\left(\tilde{h}_{y,z,b,c}^{a,a'}\right) 
& = & 
\delta_{a,a'} \tilde{U}_{y,z,b,c} - \sum_{x\in X} \tilde{U}_{y,x,b,a} \tilde{U}_{x,z,a',c}\\
& = & 
\delta_{a,a'} \tilde{U}_{y,z,b,c} - \sum_{x\in X} \tilde{U}_{b,y}^*\tilde{U}_{a,x} \tilde{U}_{a',x}^*\tilde{U}_{c,z}\\
& = & 
\delta_{a,a'} \tilde{U}_{y,z,b,c} - \delta_{a,a'} \tilde{U}_{b,y}^*\tilde{U}_{c,z} = 0.
\end{eqnarray*}
Thus, $\rho$ automatically annihilates $\tilde{\cl J}_2$. 
The proof is complete. 
\end{proof}

We say that a block operator matrix $U = (u_{a,x})_{a,x}\in M_X(\cl B(H))$ is a \emph{bi-unitary} if both $U$ and 
$U^{\rm t}$ are unitary. 
Let $C(\cl U_X^+)$ be the universal C*-algebra, generated by the entries of a 
bi-unitary $(u_{a,x})_{a,x\in X}$, and  
$C(\mathbb P \cl U_X^+)$ be the subalgebra of $C(\cl U_X^+)$ generated by the length two words of the form 
$$u_{x,x',a,a'} := u_{a,x}^*u_{a',x'}, \ \ \ x,x',a,a'\in X.$$
Further, recall that $e_{x,x',a,a'}$, 
$x,x',a,a'\in X$, denote the canonical generators of the C*-algebra $\cl C_X$
(so that $(e_{x,x',a,a'})_{x,x',a,a'}$ is a universal bistochastic operator matrix), 
set 
\begin{equation}\label{eq_gyzbc}
g_{y,z,b,c}^{x,x'} = 
\delta_{x,x'}e_{y,z,b,c} - \sum_{a\in X} e_{y,x,b,a}e_{x',z,a,c}
\end{equation}
and
\begin{equation}\label{eq_hyzbc}
h_{y,z,b,c}^{a,a'} = 
\delta_{a,a'}e_{y,z,b,c} - \sum_{x\in X} e_{y,x,b,a}e_{x,z,a',c},
\end{equation}
and let $\cl J_1$ (resp. $\cl J_2$) be the closed ideal of $\cl C_X$, generated by 
the elements $g_{y,z,b,c}^{x,x'}$ (resp. $h_{y,z,b,c}^{a,a'}$), where
$y, z, b, c, x, x'\in X$ (resp. $y, z, b, c, a, a'\in X$). 

We note that the universal C$^\ast$-algebra $C(\cl U_X^+)$ and its subalgebra 
$C(\mathbb P\cl U_X^+)$ have been well-studied in the compact quantum group literature.  
The C$^\ast$-algebra $C(\cl U_X^+)$ was introduced by Wang in \cite{wang0}, where it was shown 
to have the structure of a C$^\ast$-algebraic compact quantum group.  In particular, $C(\cl U_X^+)$ comes equipped with a co-associative comultiplication making it into a non-commutative analogue of the C$^\ast$-algebra of continuous functions of the unitary group $\cl U_X$.  The structure of the quantum group 
$C(\cl U_X^+)$ was later studied in detail by Banica in \cite{banica0}.  On the other hand, the subalgebra $C(\mathbb P \cl U_X^+) \subseteq C(\cl U_X^+)$ can be naturally interpreted as a non-commutative version of the space of continuous functions on the projective unitary group $\mathbb P \cl U_X/\mathbb T$.  In the classical setting, the conjugation action of $\cl U_X^+$ on $M_X$ induces a group isomorphism $\mathbb P \cl U_X \cong \text{Aut}(M_X)$, where $\text{Aut}(M_X)$ is the group of $\ast$-automorphisms of $M_X$. 

In the quantum setting, it is natural to expect that a similar identification between $\mathbb P\cl U_X^+$ and quantum automorphisms of $M_X$ should hold, and indeed this is the case:  In \cite{wang}, the quantum automorphism group $\text{Aut}^+(M_X)$ was introduced by Wang (via an abstract universal C$^\ast$-algebra $C(\text{Aut}^+(M_X))$ with generators and relations), and later Banica showed in \cite{banica} that the natural quantum group C$^\ast$-algebra morphism $C(\text{Aut}^+(M_X)) \to  C(\mathbb P \cl U_X^+)$ is actually an isomorphism.  In Lemma \ref{l_J1J2} below, we extend Banica's result by showing that in fact {\it any} ``concrete''  quantum automorphism of $M_X$ 
(that is, a $\ast$-homomorphism $\pi: C(\text{Aut}^+(M_X)) \cong C(\mathbb P \cl U_X^+)\to \cl B(H)$) is implemented by a ``concrete'' conjugation of $M_X$ by a bi-unitary 
(that is, $\pi$ is the restriction of a representation $C( \cl U_X^+)\to \cl B(H)$).

\begin{lemma}\label{l_J1J2}
\begin{itemize}
\item[(i)]
We have $\cl C_X/\overline{\cl J_1 + \cl J_2}\simeq C(\mathbb P \cl U_X^+)$.

\item[(ii)]
If  $\pi : C(\mathbb P \cl U_X^+) \to \cl B(H)$ is a unital *-representation then there exists a bi-unitary 
$(U_{a,x})_{a,x}\in M_X(\cl B(H))$ such that  $\pi(u_{x,x',a,a'}) = U_{a,x}^*U_{a',x'}$.
\end{itemize}
\end{lemma}

\begin{proof}
(i) 
Set $\cl J=\overline{\cl J_1 + \cl J_2}$, 
recall that $\tilde{ \cl J}_X$  is the closed ideal of $\cl C_{X,A}$ generated by the elements
$$\sum_{y\in X} \tilde{e}_{y,y,a,a'}-\delta_{a,a'}1, \ \ \ a,a'\in X$$ 
(see the paragraph containing equation (\ref{eq_JXd})) and, recalling the ideals  
$\tilde{\cl J}_1$ and $\tilde{\cl J}_2$ of $\cl C_{X,A}$ defined before Lemma \ref{l_UXAquo}, 
let 
\begin{equation}\label{eq_JXti}
\tilde{\cl J} = \overline{\tilde{\cl J}_X + \tilde{\cl J}_1 + \tilde{\cl J}_2}.
\end{equation}
According to Proposition \ref{p_quot}, $\cl C_{X,A}/\tilde {\cl J}_X\simeq \cl C_X$; thus, 
$\cl C_{X,A}/\tilde{\cl J}\simeq \cl C_X/\cl J$. 

Recall that 
$\cl U_{X,A}$ is the universal $C^*$-algebra with generators $\tilde{u}_{x,x',a,a'} := \tilde{u}_{a,x}^*\tilde{u}_{a',x'}$, 
$x,x',a,a'\in X$, where the matrix $(\tilde{u}_{a,x})_{a,x}$ is unitary. 
By Lemma \ref{l_UXAquo}, we have the canonical *-isomorphism $\cl C_{X,A}/\tilde{\cl J}_2 \simeq \cl U_{X,A}$.

We have that $\cl C_{X,A}/\tilde{\cl J}\simeq (\cl C_{X,A}/\tilde{\cl J}_2)/(\tilde{\cl J}/\tilde{\cl J}_2)$.  
Using the identification in Lemma \ref{l_UXAquo}, we have that 
$\tilde{\cl J}/\tilde{\cl J}_2$ is generated by the elements 
$$\sum_{y\in X} \tilde{u}_{y,y,a,a'} - \delta_{a,a'}1, \ \ \ a,a'\in X,$$ 
and 
$$\sum_{a\in X} \tilde{u}_{y,x,b,a} \tilde{u}_{x',z,a,c} - \delta_{x,x'} \tilde{u}_{y,z,b,c}, \ \ \ y,z,b,c,x,x'\in X.$$ 

Let $\rho :\cl U_{X,A} \simeq \cl C_{X,A}/\tilde{\cl J}_2 \to \cl B(K)$ be a unital 
*-representation that annihilates $\tilde{\cl J}/\tilde{\cl J}_2$.
By Lemma \ref{l_bhtt}, there exists a unitary 
$\tilde{U} = (\tilde{U}_{a,x})_{a,x}$ such that 
$\rho(\tilde{u}_{x,x',a,a'}) = \tilde{U}_{a,x}^* \tilde{U}_{a',x'}$, $ x,x',a,a'\in X,$
\begin{equation}\label{eq_Ut}
\sum_{y\in X} \tilde{U}_{a,y}^* \tilde{U}_{a',y} = \delta_{a,a'}I, \ \ \ a,a'\in X,
\end{equation}
and 
\begin{equation}\label{eq_Ut2}
\tilde{U}_{b,y}^* \left(\sum_{a\in X} \tilde{U}_{a,x} \tilde{U}_{a,x'}^* - \delta_{x,x'}I\right) \tilde{U}_{c,z} = 0, \ \ \ x,x',y,z,b,c\in X.
\end{equation}
By (\ref{eq_Ut}), $\tilde{U}^{\rm t} = (\tilde{U}_{y,a})_{a,y}$ is an isometry. 
But then $\tilde{U}^{\rm t}(\tilde{U}^{\rm t})^*\leq I$, implying, by comparing the $(x,x)$-entries of the matrices, 
that $\sum_{a\in X} \tilde{U}_{a,x}\tilde{U}_{a,x}^* \leq I$, $x\in X$. 
On the other hand, (\ref{eq_Ut2}) implies 
$\tilde{U}_{b,y}^*\left(\sum_{a\in X} \tilde{U}_{a,x} \tilde{U}_{a,x}^* - I\right) \tilde{U}_{b,y} = 0$. 
Thus, 
$$\left(I - \sum_{a\in X} \tilde{U}_{a,x} \tilde{U}_{a,x}^*\right)^{1/2} \tilde{U}_{b,y} = 0$$ 
and hence
$\left(\sum_{a\in X} \tilde{U}_{a,x} \tilde{U}_{a,x}^* - I\right)\tilde{U}_{b,y} = 0$. 
Since $\tilde{U}$ is unitary, this implies 
$$0 = \left(\sum_{a\in X} \tilde{U}_{a,x}\tilde{U}_{a,x}^* - I\right)
\sum_{y\in X} \tilde{U}_{b,y} \tilde{U}_{b,y}^*
= \sum_{a\in X} \tilde{U}_{a,x} \tilde{U}_{a,x}^*-I.$$ 
Now, 
\begin{equation}\label{eq_conjUby}
\left(\tilde{U}_{b,y}^*\otimes I\right) \left (I-\tilde{U}^{\rm t}\tilde{U}^{{\rm t}*}\right) \left(\tilde{U}_{b,y}\otimes I\right)
\end{equation}
is a positive block matrix in $M_X(\cl B(H))$ and has zeros on its main diagonal.
It follows that the matrix (\ref{eq_conjUby}) is zero and hence 
\begin{equation}\label{eq_I-}
\left (I - \tilde{U}^{\rm t}\tilde{U}^{{\rm t}*}\right)^{1/2} \left(\tilde{U}_{b,y}\otimes I\right) = 0, \ \ \ b,y\in X.
\end{equation}
Multiplying (\ref{eq_I-}) by 
$\tilde{U}_{b,y}^*\otimes I$ on the right and 
adding up along the variable $y$, we obtain
$\tilde{U}^{\rm t}\tilde{U}^{{\rm t}*} = I$; thus, $U^{\rm t}$ is unitary. 
Therefore, $U$ gives rise to a unital *-representation of $C(\cl U_X^+)$ and, after restriction, to a 
unital *-representation of $C(\mathbb P \cl U_X^+)$. 
We have thus shown that 
every unital $*$-representation $\rho:\cl C_{X,A}/\tilde{\cl J_{2}} \to \cl B(K)$ 
that annihilates $\tilde{\cl J}/\tilde{\cl J}_2$
induces a unital $*$-homomorphism from $C(\mathbb P \cl U_X^+)$ to $\cl B(K)$.

By \cite[Theorem 5.2]{tt-QNS}, there exists a $*$-homomorphism
$\varphi: \cl C_{X,A}\to C(\mathbb P \cl U_X^+)$, such that 
$\varphi(\tilde e_{x,x',a,a'}) = u_{x,x',a,a'}$, $x,x',a,a'\in X$. 
A straightforward verification shows that $\varphi$ annihilates $\tilde{\cl J}_2$ and hence gives rise to a $*$-homomorphism $\tilde\varphi: \cl C_{X,A}/\tilde{\cl J}_2\to C(\mathbb P \cl U_X^+)$, $\tilde e_{x,x',a,a'}+\tilde{\cl J}_2\mapsto u_{x,x',a,a'}$. It is easy to see that $\tilde{\cl J}/\tilde{\cl J}_2\subseteq\ker\tilde\varphi$. 
The previous paragraph shows that 
if $T\in \cl C_{X,A}/\tilde{\cl J}_2$ then 
\begin{eqnarray*}
\|T+\tilde{\cl J}/\tilde{\cl J}_2\|
& = &
\sup\{\|\rho(T)\|: \rho \mbox{ a *-rep. of }\cl C_{X,A}/\tilde{\cl J}_2 
\mbox{ with } 
\rho(\tilde{\cl J}/\tilde{\cl J}_2)=0\}\\
& \leq &
\|\tilde\varphi(T)\|,
\end{eqnarray*}
giving the inclusion $\ker(\tilde\varphi)\subseteq \tilde{\cl J}/\tilde{\cl J}_2$ and hence the equality $\ker(\tilde\varphi) = \tilde{\cl J}/\tilde{\cl J}_2$. As $\tilde\varphi$ is surjective we obtain the statement.

(ii)
Let $\pi : C(\mathbb P \cl U_X^+) \to \cl B(H)$ be a unital *-representation. Letting 
$\rho : \cl C_{X,A} \to \cl C_{X,A}/\tilde{\cl J}$ be the quotient map, the proof of (i) allows us to consider $\rho$ as a *-epimorphism from $\cl C_{X,A}$ onto $C(\mathbb P \cl U_X^+)$. It further exhibits a bi-unitary $\tilde{U} = (\tilde{U}_{a,x})_{a,x}$ such that $(\pi\circ\rho)(\tilde{e}_{x,x',a,a'}) = \tilde{U}_{a,x}^*\tilde{U}_{a',x'}$, $x,x',a,a'\in X$. 
We now see that 
$\pi(u_{x,x',a,a'}) = \tilde{U}_{a,x}^*\tilde{U}_{a',x'}$, $x,x',a,a'\in X$, and the proof is complete. 
\end{proof}

We recall that the \emph{opposite} C*-algebra $\cl A^{\rm op}$ of a C*-algebra $\cl A$ has the same set, linear structure and involution as $\cl A$, and multiplication given by $u^{\rm op}v^{\rm op} = (vu)^{\rm op}$, where $u^{\rm op}$ denotes the element $u\in \cl A$ when viewed as an element of $\cl A^{\rm op}$. 
Given a Hilbert space $H$, let $H^{\rm d}$ denote its dual Banach space and, for an operator $T\in \cl B(H)$, let $T^{\rm d} : H^{\rm d}\to H^{\rm d}$ be its dual. 
We note the identity 
\begin{equation}\label{eq_dstar}
(T^{*{\rm d}})^* = T^{\rm d}.
\end{equation}
If $\pi : \cl A\to \cl B(H)$ is a faithful *-representation, then the map $\pi^{\rm op} : \cl A^{\rm op}\to \cl B(H^{\rm d})$, given by 
$\pi^{\rm op}(u^{\rm op}) = \pi(u)^{\rm d}$, is a faithful *-representation. 

The following result can be proved using the existence of the antipode for compact quantum groups together with the fact that $\mathbb P \cl U_X^+$, the antipode is known to be a $\ast$-anti-automorphism of $C(\mathbb P \cl U_X^+)$ (see e.g., \cite[Proposition 1.7.9]{nt}). For the sake of those unacquainted with quantum group technicalities, we supply a self-contained proof.

\begin{lemma}\label{l_opos}
Let $X$ be a finite set. The map 
$$\partial(u_{x,x',a,a'}) = u_{x',x,a',a}^{\rm op}, \ \ \ x,x',a,a'\in X,$$
extends to a *-isomorphism $\partial : C(\mathbb P \cl U_X^+)\to C(\mathbb P \cl U_X^+)^{\rm op}$.
\end{lemma}

\begin{proof}
Let $\pi : C(\mathbb P \cl U_X^+)\to \cl B(H)$ be a faithful *-representation and $U = (U_{a,x})_{a,x} \in M_X(\cl B(H))$ 
be a bi-unitary such that $\pi(u_{x,x',a,a'}) = U_{a,x}^*U_{a',x'}$, $x,x',a,a'\in X$. 

Set $V_{a,x} = U_{a,x}^{*\rm d}$, $x,a\in X$. We observe that $V := (V_{a,x})_{a,x}$ is a bi-unitary. 
Indeed, using (\ref{eq_dstar}), we have 
$$\sum_{a\in X} V_{a,x}^*V_{a,x'} = \sum_{a\in X} U_{a,x}^{\rm d}U_{a,x'}^{*\rm d} = 
\left(\sum_{a\in X} U_{a,x'}^*U_{a,x}\right)^{\rm d} = \delta_{x,x'} I_{H^{\rm d}}$$
and 
$$\sum_{x\in X} V_{a,x}^*V_{a',x} = \sum_{x\in X} U_{a,x}^{\rm d}U_{a',x}^{*\rm d} = 
\left(\sum_{x\in X} U_{a',x}^*U_{a,x}\right)^{\rm d} = \delta_{a,a'} I_{H^{\rm d}},$$
that is, $V^*V = I$ and $V^{\rm t *}V^{\rm t} = I$; the relations 
$VV^* = I$ and $V^{\rm t}V^{\rm t *} = I$ follow analogously.
It follows that there exists a *-representation $\rho : C(\mathbb P \cl U_X^+)\to \cl B(H^{\rm d})$ such that
$\rho(u_{x',x,a',a}) =  \pi^{\rm op}(u_{x,x',a,a'}^{\rm op})$, $x,x',a,a'\in X$;  
note that $\rho$ is a (well-defined) *-homomorphism from $C(\mathbb P \cl U_X^+)$ into $C(\mathbb P \cl U_X^+)^{\rm op}$. 
By symmetry considerations, $\rho$ is a *-isomorphism. 
\end{proof}

Before formulating the next theorem, we introduce some notation and terminology. 
If $\Phi : M_X\to M_X$ is a quantum channel, we write $\Phi^{\sharp} : M_X\to M_X$ for the quantum channel 
given by 
$$\Phi^{\sharp}(\omega) = \Phi(\omega^{\rm t})^{\rm t}, \ \ \ \omega\in M_X.$$
We call a channel $\Phi : M_X\to M_X$ a
\emph{unitary channel} if there exists a 
unitary $U = (\lambda_{a,x})_{a,x\in X}\in M_X$, 
such that 
$\Phi(\omega) = U^*\omega U$, $\omega\in M_X$.
Finally, a trace $\tau : \cl B\to \bb{C}$ of a C*-algebra $\cl B$ is called \emph{abelian} if
there exists an abelian C*-algebra $\cl A$, a *-homomorphism 
$\pi : \cl B\to \cl A$ and a state $\phi : \cl A \to\mathbb C$ such that 
$\tau = \phi\circ\pi$.

\begin{theorem}\label{QNSbicorrelation}
Let $X$ be a finite set and $\Gamma: M_{XX}\to M_{XX}$ be a QNS bicorrelation. Then 

\begin{itemize}
\item[(i)] $\Gamma \in \cl Q_{\rm qc}^{\rm bic}$ if and only if
there exists a trace $\tau: C(\mathbb P \cl U_X^+) \to\mathbb C$ such that 
\begin{equation}\label{eq_UXrep}
\Gamma(\epsilon_{x,x'}\otimes \epsilon_{y,y'}) = (\tau(u_{x,x',a,a'}u_{y',y,b',b}))_{a,a',b,b'}, \ \ \ x,x',y,y'\in X;
\end{equation}

\item[(ii)] 
$\Gamma \in \cl Q_{\rm q}^{\rm bic}$ if and only if 
(\ref{eq_UXrep}) holds for a trace of $C(\mathbb P \cl U_X^+)$ that factors through a finite dimensional C*-algebra;

\item[(iii)] 
$\Gamma \in \cl Q_{\rm loc}^{\rm bic}$ if and only if (\ref{eq_UXrep}) holds for an abelian trace of 
$C(\mathbb P \cl U_X^+)$, 
if and only if there exist unitary channels $\Phi_i$, $i = 1,\dots,k$, such that 
$\Gamma = \sum_{i=1}^k \lambda_i \Phi_i\otimes \Phi_i^{\sharp}$ as a convex combination.
\end{itemize}
\end{theorem}

\begin{proof}
(i) 
Let $U := (u_{a,x})_{a,x}$ be the universal bi-unitary and $\Gamma : M_{XX}\to M_{XX}$ be given via (\ref{eq_UXrep}).
There exists a state
$\nu : C(\mathbb P \cl U_X^+)\otimes_{\max} C(\mathbb P \cl U_X^+)^{\rm op}\to \bb{C}$, given by 
\begin{equation}\label{eq_nu}
\nu(u\otimes v^{\rm op}) = \tau(uv), \ \ \ u,v\in C(\mathbb P \cl U_X^+)
\end{equation}
(see \cite[p. 219]{bo}). 
Let $s = \nu\circ (\id\otimes \partial)$; thus, $s$ is a state on 
$C(\mathbb P \cl U_X^+)\otimes_{\max} C(\mathbb P \cl U_X^+)$. 
We have 
$$
s(u_{x,x',a,a'}\otimes u_{y,y',b,b'}) 
= 
\nu(u_{x,x',a,a'}\otimes u_{y',y,b',b}^{\rm op})
= 
\tau(u_{x,x',a,a'} u_{y',y,b',b}),
$$
implying that 
$\Gamma = \Gamma_s$. 
By Lemma \ref{l_J1J2} (i) and Theorem \ref{th_bicqc},
$\Gamma\in \cl Q_{\rm qc}^{\rm bi}$. 
Since $U$ is unitary, by the proof of \cite[Theorem 4.3]{bhtt}, $\Gamma$ is concurrent.

Conversely, let $\Gamma\in \cl Q_{\rm qc}^{\rm bic}$.
By Theorem \ref{th_bicqc}, there exists a state
$s : \cl C_X\otimes_{\rm max} \cl C_X\to\mathbb C$ such that
$\Gamma = \Gamma_s$.  
Let $V = (v_{a,x})_{a,x}$ be a universal bi-isometry (see Subsection \ref{ss_uopsys}) and 
denote by $f_{y,y',b,b'}$ the canonical generators of the second copy of $\cl C_X$ in the tensor product.  
The concurrency of $\Gamma$ implies the validity of the condition
\begin{equation}\label{eq_onec}
\sum_{x',y'\in X} s\left(e_{x',y',a,b}\otimes f_{x',y',a,b}\right) = 1, \ \ \ a,b\in X.
\end{equation}
Let $\tau : \cl C_X\to \bb{C}$ be the functional, given by $\tau(u) = s(u\otimes 1)$, $u\in \cl C_X$.
By \cite[Lemma 4.2]{bhtt}, there exists a canonical 
*-epimorphism $\pi : \cl C_{X,A}\to \cl C_X$; let $\tilde{\tau} = \tau\circ\pi$. 
Letting $\tilde{s} : \cl C_{X,A}\otimes_{\max} \cl C_{X,A}\to \bb{C}$ be given by 
$\tilde{s}(w) = (s\circ (\pi\otimes \pi))(w)$, we have that 
$$
\sum_{x',y'\in X} \tilde{s}\left(\tilde{e}_{x',y',a,b}\otimes \tilde{f}_{x',y',a,b}\right) = 1,$$
and now, by the proof of \cite[Theorem 4.1]{bhtt}, that 
%\begin{equation}\label{eq_swap3}
$$\tilde{s}\left(\tilde{e}_{x,y,a,b}\otimes \tilde{f}_{x,y,a,b}\right) 
= \tilde{s}\left(\tilde{e}_{x,y,a,b} \tilde{e}_{y,x,b,a}\otimes 1\right), \ \ \ x,y,a,b\in X,$$
%\end{equation}
and that $\tilde{\tau}$ is a tracial state. 
After passing to quotients, we conclude that 
\begin{equation}\label{eq_swap}
s\left(e_{x,y,a,b}\otimes f_{x,y,a,b}\right) 
= s\left(e_{x,y,a,b} e_{y,x,b,a}\otimes 1\right), \ \ \ x,y,a,b\in X,
\end{equation}
and that $\tau$ is a tracial state. 

Recalling notation (\ref{eq_gyzbc}) and (\ref{eq_hyzbc}), set 
$$
G_{y,z,b,c} = \left(g_{y,z,b,c}^{x,x'}\right)_{x,x'} 
\ \mbox{ and } \ 
H_{y,z,b,c} = \left(h_{y,z,b,c}^{a,a'}\right)_{a,a'}.$$
We claim that
\begin{equation}\label{eq_posm}
\tilde{G}_{y,z,b,c} := \left[\begin{matrix}
G_{y,y,b,b} & G_{y,z,b,c}\\
G_{z,y,c,b} & G_{z,z,c,c}
\end{matrix}\right]
\in M_2\left(M_X(\cl C_X) \right)^+.
\end{equation}
Indeed, set $Z_{y,z,b,c} := \left[\begin{matrix}
v_{b,y}\otimes I_{X} & 0\\
0 & v_{c,z} \otimes I_{X}
\end{matrix}\right]$. 
After applying the canonical shuffle $M_2(M_X(\cl C_X)) \simeq M_X(M_2(\cl C_X))$, we obtain 
\begin{eqnarray*}
& & \tilde{G}_{y,z,b,c} 
\hspace{-0.05cm} = \hspace{-0.05cm} 
\left[\begin{matrix}
g_{y,y,b,b}^{x,x'} & g_{y,z,b,c}^{x,x'}\\
g_{z,y,c,b}^{x,x'} & g_{z,z,c,c}^{x,x'}
\end{matrix}\right]_{x,x'}\\
& \hspace{-0.1cm} = \hspace{-0.1cm} & 
\left[\begin{matrix}
\delta_{x,x'}e_{y,y,b,b} - \sum_{a} e_{y,x,b,a}e_{x',y,a,b} & \delta_{x,x'}e_{y,z,b,c} - \sum_{a} e_{y,x,b,a}e_{x',z,a,c}\\
\delta_{x,x'}e_{z,y,c,b} - \sum_{a} e_{z,x,c,a}e_{x',y,a,b} & \delta_{x,x'}e_{z,z,c,c} - \sum_{a} e_{z,x,c,a}e_{x',z,a,c} 
\end{matrix}\right]_{x,x'}\\
& \hspace{-0.1cm} = \hspace{-0.1cm} & 
Z_{y,z,b,c}^*
\left[\begin{matrix}
\delta_{x,x'}1 - \sum_{a} v_{a,x}v_{a,x'}^* & \delta_{x,x'}1 - \sum_{a} v_{a,x}v_{a,x'}^*\\
\delta_{x,x'}1 - \sum_{a} v_{a,x}v_{a,x'}^* & \delta_{x,x'}1 - \sum_{a} v_{a,x}v_{a,x'}^*
\end{matrix}\right]_{x,x'}
Z_{y,z,b,c}.
\end{eqnarray*}
Since $V^{\rm t}$ is an isometry, $V^{\rm t}V^{{\rm t}*}\leq I$ and hence 
$\left(\delta_{x,x'}1 - \sum_{a} v_{a,x}v_{a,x'}^*\right)_{x,x'} \geq 0$, implying (\ref{eq_posm}), 
along with the relations $\sum_{a\in X} v_{a,x} v_{a,x}^* \leq 1$, $x\in X$.
Identity (\ref{eq_posm}) now shows that $G_{y,y,b,b}\in M_X(\cl C_X)^+$, and hence 
\begin{equation}\label{eq_Gyybb}
\tau^{(X)}\left(G_{y,y,b,b}\right)\in M_X^+.
\end{equation}
We have that
\begin{equation}\label{eq_anoo}
\sum_{a\in X} e_{y,x,b,a}e_{x,y,a,b} 
=  
\sum_{a\in X} v_{b,y}^* v_{a,x} v_{a,x}^* v_{b,y}\\
\leq e_{y,y,b,b}
\end{equation}
and hence 
\begin{equation}\label{eq_less1}
\sum_{x,y \in X} \sum_{a,b\in X} e_{y,x,b,a}e_{x,y,a,b} 
\leq 
\sum_{x,y \in X} \sum_{b\in A} e_{y,y,b,b}
= 
|X|^2 1;
\end{equation}
similarly, 
\begin{equation}\label{eq_less2}
\sum_{x,y \in X} \sum_{a,b\in X}  f_{x,y,a,b}f_{y,x,b,a} \leq |X|^21.
\end{equation}
By (\ref{eq_onec}), (\ref{eq_less1}) and (\ref{eq_less2}), 
\begin{eqnarray*} 
0 
& \leq & 
\sum_{x,y,a,b} 
s\left(\left(e_{x,y,a,b}\otimes 1 - 1 \otimes f_{y,x,b,a}\right)^* \left(e_{x,y,a,b}\otimes 1 - 1 \otimes f_{y,x,b,a}\right)\right)\\
& = & 
\sum_{x,y,a,b} 
s\left(\left(e_{y,x,b,a}\otimes 1 - 1 \otimes f_{x,y,a,b}\right) \left(e_{x,y,a,b}\otimes 1 - 1 \otimes f_{y,x,b,a}\right)\right)\\
& = & 
\sum_{x,y,a,b} s\left(e_{y,x,b,a}e_{x,y,a,b}\otimes 1  + 1 \otimes f_{x,y,a,b}f_{y,x,b,a}\right)\\
& - & \sum_{x,y,a,b} s\left(e_{y,x,b,a} \otimes f_{y,x,b,a} + e_{x,y,a,b} \otimes f_{x,y,a,b}\right)
\leq 
2|X|^2  - 2|X|^2 = 0.
\end{eqnarray*}
Applying $\tau$ to (\ref{eq_less1}),
we have
$$
\sum_{x,y \in X} \sum_{a,b\in X} \tau(e_{y,x,b,a}e_{x,y,a,b}) 
\leq 
\sum_{x,y \in X} \sum_{b\in A} 
\tau(e_{y,y,b,b})
= |X|^2.
$$
On the other hand, by (\ref{eq_swap}), 
$$
\sum_{x,y \in X} \sum_{a,b\in X} \tau(e_{y,x,b,a}e_{x,y,a,b}) = 
|X|^2.$$ 
Using (\ref{eq_anoo}), we now have that 
$$\tau\left(e_{y,y,b,b} - \sum_{a\in X} e_{y,x,b,a}e_{x,y,a,b}\right) = 0 \ \ \mbox{ for all } x,y,b\in X.$$
Thus the diagonal entries of $\tau^{(X)}\left(G_{y,y,b,b}\right)$ are zero; the
positivity condition (\ref{eq_Gyybb}) implies that the off-diagonal entries of 
$\tau^{(X)}\left(G_{y,y,b,b}\right)$ are also zero. 
Now the positivity condition (\ref{eq_posm}) implies that 
$$\tau^{(2X)}\left(\tilde{G}_{y,z,b,c}\right) = 0, \ \ \ y,z,b,c\in X.$$
Condition (\ref{eq_posm}) 
and the Cauchy-Schwarz inequality imply $\tau^{(2X)}\left(Q\tilde{G}_{y,z,b,c}^{1/2}\right) = 0$, 
for all $Q\in M_2(M_{X}(\cl C_{X}))$, and hence
$\tau^{(2X)}$ annihilates the closed ideal of $M_2(M_{X}(\cl C_{X}))$ generated by 
$\tilde{G}_{y,z,b,c}^{1/2}$. In particular, 
$\tau^{(2X)}$ annihilates the closed ideal of $M_2(M_{X}(\cl C_{X}))$ generated by 
$\tilde{G}_{y,z,b,c}$; since $\cl C_X$ is unital, this implies that 
$\tau$ annihilates the closed ideal of $\cl C_{X}$ generated by the elements 
$g_{y,z,b,c}^{x,x'}$, $x,x',y,z,b,c\in X$, that is, $\cl J_1$.

Similarly, observe that 
$$\sum_{x\in X} e_{y,x,b,a}e_{x,y,a,b} 
=  
\sum_{x\in X} v_{b,y}^* v_{a,x} v_{a,x}^* v_{b,y}\\
\leq e_{y,y,b,b}, \ \ \ y,a,b\in X.$$
By Remark \ref{r_biconc}, 
$$\sum_{a',b'\in X} s\left(e_{x,y,a',b'}\otimes f_{x,y,a',b'}\right) = 1, \ \ \ x,y\in X.$$
Using (\ref{eq_swap}) yields similarly
$$\tau\left(e_{y,y,b,b} - \sum_{x\in X} e_{y,x,b,a}e_{x,y,a,b}\right) = 0 \ \ \ x,y,b\in X,$$
leading to the relations 
$$\tau^{(2X)}\left(\tilde{H}_{y,z,b,c}\right) = 0, \ \ \ y,z,b,c\in X.$$
It follows that $\tau$ annihilates the ideal $\cl J_2$, generated by $h_{y,z,b,c}^{a,a'}$, where $a,a'$, $y,z,b,c\in X$, and hence 
it annihilates $\overline{\cl J_1 + \cl J_2}$. 
Hence $\tau$
induces a tracial state (denoted in the same fashion) on the quotient $\cl C_X/\cl J$. 
An application of Lemma \ref{l_J1J2} (i) completes the proof.

\smallskip

(ii) 
Suppose that $\Gamma : M_{XX}\to M_{XX}$ is a quantum concurrent QNS bicorrelation.
By \cite[Theorem 4.3]{bhtt}, there exists a finite dimensional C*-algebra $\cl A$, a trace $\frak{t}$ on $\cl A$,
and a *-homomorphism $\alpha : \cl U_{X,A}\to \cl A$, such that
$\Gamma = \Gamma_{\frak{t}\circ\alpha}$.  After taking a quotient, 
we may assume that $\frak{t}$ is faithful.
Let $\rho : \cl C_{X,A}\to \cl U_{X,A}$ be the canonical quotient map, 
whose existence is guaranteed by \cite[Lemma 4.2]{bhtt}.
Let $\tilde{\tau} : \cl C_{X,A}\to \bb{C}$ be the functional, given by
$\tilde{\tau}(u) = (\frak{t}\circ\alpha\circ\rho)(u)$, $u\in \cl C_{X,A}$; clearly, $\tilde{\tau}$ is a trace on $\cl C_{X,A}$.
Note, further, that $\Gamma = \Gamma_{\tilde{\tau}}$
(for brevity here, and in the sequel, $\Gamma_{\tilde{\tau}}$ is 
used to denote $\Gamma_{s_{\tilde{\tau}}}$, where $s_{\tilde{\tau}}$ is the state, canonically associated
with the trace $\tilde{\tau}$).
By the proof of (i), $\tilde{\tau}$ annihilates the ideal $\tilde{\cl J}$ defined in (\ref{eq_JXti}); thus,
as $\frak{t}$ is faithful, $(\alpha\circ\rho)(\tilde{\cl J}) = 0$ and hence we get a $*$-homomorphism $\tilde\rho: C(\mathbb P \cl U_X^+)\to \cl A$ and the trace $\tau=\frak{t}\circ\tilde\rho$ on $C(\mathbb P \cl U_X^+)$ which factors through $\cl A$.

Conversely, suppose that $\cl B$ is a finite dimensional C*-algebra. 
Let $\pi : C(\mathbb P \cl U_X^+)\to \cl B$ be a unital *-homomorphism and $\tilde{\tau} : \cl B\to \bb{C}$ be a trace such that, if 
$\tau = \tilde{\tau}\circ \pi$, then $\Gamma = \Gamma_{\tau}$. 
By Lemma \ref{l_J1J2} (ii), there exists a finite dimensional Hilbert space $K$ and 
a bi-unitary matrix $U = (U_{a,x})_{a,x}\in M_X(\cl B(K))$, such that 
$\pi(u_{x,x',a,a'}) = U_{a,x}^*U_{a',x'}$, $x,x',a,a'\in X$. 
Now a straightforward verification shows that $\Gamma\in \cl Q_{\rm q}^{\rm bic}$.

\smallskip

(iii) 
Suppose that $\Gamma \in \cl Q_{\rm loc}^{\rm bic}$. By \cite[Theorem 4.3 (iii)]{bhtt}, there exists an abelian C*-algebra 
$\cl A$, a *-homomorphism $\tilde{\pi} : \cl U_{X,A}\to \cl A$
and a state $\phi : \cl A\to \bb{C}$ such that, if $\tilde{\tau} = \phi\circ\tilde{\pi}$ 
then ($\tilde{\tau}$ is a trace on $\cl U_{X,A}$ such that) 
$\Gamma = \Gamma_{\tilde{\tau}}$. 
Realise $\cl A = C(\Omega)$ for some compact Hausdorff space $\Omega$ and let $\mu$ be a regular Borel measure
on $\Omega$ such that $\phi(h) = \int_{\Omega} h d\mu$. 
Writing 
$\tilde{U}_{x,x',a,a'} = \tilde{\pi}(\tilde{u}_{x,x'a,a'})$, $x,x'a,a'\in X$, we have 
$$\phi(\tilde{U}_{x,x',a,a'}\tilde{U}_{y',y,b',b}) = \int_{\Omega} \tilde{U}_{x,x',a,a'}(t)\tilde{U}_{y',y,b',b}(t) d\mu(t), \ \ \ x,x',a,a'\in X.$$
As $\mu$ can be approximated by convex combinations of point mass evaluations,
$\Gamma$ can be approximated by convex combinations 
$\sum_{i=1}^k \lambda_i \Gamma_i$, where 
$$\Gamma_i(\epsilon_{x,x'}\otimes \epsilon_{y,y'}) = \left(\mu_{x,x',a,a'}^{(i)}\mu_{y',y,b',b}^{(i)}\right)_{a,a',b,b'}, 
\ \ \ x,x',y,y'\in X,$$
for some scalar matrices $M_i = \left(\mu_{x,x',a,a'}^{(i)}\right)_{x,x',b,b'}$.
Since the matrices $M_i$ give rise to (one-domensional) *-representations of $\cl U_{X,A}$, 
by Lemma \ref{l_bhtt}, they admit factorisations of the form 
$\mu_{x,x',a,a'}^{(i)} = \bar{\lambda}_{a,x}^{(i)}\lambda_{a',x'}^{(i)}$, $x,x',a,a'\in X$, for a unitary matrix 
$U_i = (\lambda_{a,x}^{(i)})_{a,x}$, $i = 1,\dots,k$. 
Note that $\Gamma_i = \Phi_i\otimes \Phi_i^{\sharp}$, where 
$\Phi_i$ is the (unital) quantum channel with Choi matrix $\left(\mu_{x,x',a,a'}^{(i)}\right)_{x,x',a,a'}$. 
By the Carath\'eodory Theorem and compactness, we have that $\Gamma$ is itself a convex combination of this form. 
We further have  that 
$$\Phi_i(\omega) = U_i^{{\rm t} *}\omega U_i^{\rm t}, \ \ \ \omega\in M_X, \ i = 1,\dots,k,$$
and in particular $\Phi_i$ is a unitary channel, $i = 1,\dots,k$. 

Suppose that $\Phi : M_X\to M_X$ is a unitary channel. Let $U = (\lambda_{a,x})_{a,x}\in M_X$ 
be a unitary (and hence a bi-unitary) 
such that $\Phi(\omega) = U^*\omega U$, $\omega\in M_X$. 
We have that 
\begin{eqnarray*}
\left(\Phi\otimes\Phi^{\sharp}\right)(J_X)
& = & 
\frac{1}{|X|} \sum_{x,y\in X} \Phi(\epsilon_{x,y})\otimes \Phi(\epsilon_{y,x})^{\rm t}\\
& = & 
\frac{1}{|X|} \sum_{x,y\in X} (U^*e_x) (U^*e_y)^* \otimes ((U^*e_y) (U^*e_x)^*)^{\rm t}\\
& = & 
\frac{1}{|X|} \sum_{x,y\in X} \sum_{a,b\in X} \sum_{a',b' \in X} 
\lambda_{y,b}\overline{\lambda}_{x,a}\lambda_{x,a'}\overline{\lambda}_{y,b'}
(\epsilon_{a,b}\otimes \epsilon_{a',b'})\\
& = & 
\frac{1}{|X|} \sum_{a,b\in X} \sum_{a',b' \in X} \delta_{a,a'} \delta_{b,b'}
(\epsilon_{a,b}\otimes \epsilon_{a',b'}) 
= J_X.
\end{eqnarray*} 
Thus, $\Phi\otimes\Phi^{\sharp}$ is a concurrent correlation and, since $\Phi$ is unital, 
it is a concurrent bicorrelation. 
Since $\cl Q_{\rm loc}^{\rm bic}$ is convex, we have that all convex combinations of elementary 
tensors of the form $\Phi\otimes\Phi^{\sharp}$ belong to $\cl Q_{\rm loc}^{\rm bic}$. 

Now assume that 
$\Gamma = \sum_{i=1}^k \lambda_i \Phi_i\otimes \Phi_i^{\sharp}$ as a convex combination, 
where $\Phi_i$ is a unitary channel, $i = 1,\dots,k$. 
Assume that $\Phi_i(\omega) = U_i^*\omega U_i$, $\omega\in M_{X}$, 
where $U_i\in M_X$ is a unitary. Since $U_i$ has scalar entries, it is automatically a bi-unitary, 
and hence gives rise to a canonical (one-dimensional) unital *-representation of $C(\mathbb P \cl U_X^+)$.
A standard argument now shows that $\Gamma = \Gamma_{\tau}$ for a trace on the (finite dimensional)
abelian C*-algebra $\cl D_k$.

Finally, if $\Gamma = \Gamma_{\tau}$, where $\tau$ factors through an abelian C*-algebra then 
the argument in the first paragraph of (iii) shows that $\Gamma \in \cl Q_{\rm loc}^{\rm bic}$. 
\end{proof}

\begin{remark}\label{r_amenable}
\rm 
Assume that $\tau$ is an amenable trace of $C(\mathbb P \cl U_X^+)$. By \cite[Theorem 6.2.7]{bo}, 
the functional 
$\mu : C(\mathbb P \cl U_X^+)\otimes_{\min} C(\mathbb P \cl U_X^+)^{\rm op}\to \bb{C}$, given by 
$\mu(u\otimes v^{\rm op}) = \tau(uv)$, is a well-defined state. 
Letting $s = \mu\circ (\id\otimes \partial)$ (a state on 
$C(\mathbb P \cl U_X^+)\otimes_{\min} C(\mathbb P \cl U_X^+)^{\rm op}$), 
one can proceed similarly to the first paragraph of the proof of Theorem \ref{QNSbicorrelation}
to conclude that $\Gamma\in \cl Q_{\rm qa}^{\rm bic}$. 
We do not know if, conversely, every $\Gamma\in \cl Q_{\rm qa}^{\rm bic}$ 
arises from an amenable trace on $C(\mathbb P \cl U_X^+)$. 
\end{remark}

Recall \cite{pr} that an NS correlation $p$ over $(X,X,X,X)$ is called \emph{bisynchronous} if 
$$p(a,b|x,x)\neq 0 \ \Longrightarrow \ a = b
\ \ \mbox{ and } \ \ p(a,a|x,y)\neq 0 \ \Longrightarrow \ x = y.$$
It was shown in \cite[Remark 2.1]{pr} that bisynchronous correlations of type ${\rm t}\neq {\rm ns}$
are (classical) bicorrelations. 
The next statement describes the relation between 
bisynchronicity and concurrency. 

\begin{proposition}\label{p_bysco}
Let ${\rm t}\in \{{\rm loc}, {\rm q}, {\rm qc}\}$. If $p\in \cl C_{\rm t}$ is a bisynchronous NS correlation over the quadruple $(X,X,X,X)$ then there exists $\Gamma \in \cl Q_{\rm t}^{\rm bic}$ such that 
\begin{equation}\label{eq_Ep}
\cl E_p = \Delta\circ \Gamma|_{\cl D_{XX}}.
\end{equation}
\end{proposition}

\begin{proof}
We consider first the case ${\rm t} = {\rm qc}$.
Let $p\in \cl C_{\rm qc}$ be a bisynchronous correlation. 
By \cite[Theorem 2.2]{pr}, there exists a tracial state $\tau : C(S_X^+)\to \bb{C}$ such that 
\begin{equation}\label{eq_pabxybitau}
p(a,b|x,y) = \tau(p_{a,x}p_{b,y}), \ \ \ x,y,a,b\in X.
\end{equation}
Let 
$$p_{x,x'a,a'} := p_{a,x}^*p_{a',x'} = p_{a,x}p_{a',x'},
\ \ \ x,x',a,a'\in X,$$
and let
$C(\mathbb P S_X^+)$ be the subalgebra of 
$C(S_X^+)$, generated by the elements of the form 
$p_{x,x',a,a'}$, $x,x',a,a'\in X$. 
Since every quantum permutation is a bi-unitary, 
there exists a unital *-homomorphism
$\pi : C(\mathbb P \cl U_X^+)\to C(\mathbb P S_X^+)$ with 
$$\pi(e_{x,x',a,a'}) = p_{x,x',a,a'}, \ \ \ 
x,x',a,a'\in X.$$
Let $\tilde{\tau} = \tau\circ\pi$; thus, $\tilde{\tau}$
is a tracial state on $C(\mathbb P \cl U_X^+)$ and hence, 
by Theorem \ref{QNSbicorrelation}, $\Gamma_{\tilde{\tau}}$
is a quantum commuting concurrent QNS bicorrelation. 
Moreover, if $x,y\in X$ then 
\begin{eqnarray*}
(\Delta\circ \Gamma_{\tilde{\tau}}) (\epsilon_{x,x}\otimes \epsilon_{y,y}) 
& = & 
\sum_{a,b\in X}\tilde{\tau}(e_{x,x,a,a}e_{y,y,b,b})
\epsilon_{a,a}\otimes \epsilon_{b,b}\\
& = & 
\sum_{a,b\in X}\tau(p_{x,x,a,a}p_{y,y,b,b})
\epsilon_{a,a}\otimes \epsilon_{b,b}\\
& = & 
\sum_{a,b\in X}\tau(p_{a,x}^*p_{a,x}p_{b,y}^*p_{b,y})
\epsilon_{a,a}\otimes \epsilon_{b,b}\\
& = & 
\sum_{a,b\in X}\tau(p_{a,x}p_{b,y})
\epsilon_{a,a}\otimes \epsilon_{b,b}
= \cl E_p(\epsilon_{x,x}\otimes \epsilon_{y,y}),
\end{eqnarray*}
and (\ref{eq_Ep}) follows. 

The cases ${\rm t} = {\rm q}$ and ${\rm t} = {\rm loc}$ are similar.
\end{proof}

%%%%%%%%%%%%%%%%%%%%%%%%%%%%%%%%%%%%%%%%%%%%%%%%
%%%%%%%%%%%%%%%%%%%%%%%%%%%%%%%%%%%%%%%%%%%%%%%%

\section{The quantum graph isomorphism game}\label{s_qgig}

In this section, we view the concurrent bicorrelations studied in Section \ref{s_concbic} as strategies for 
the non-commutative graph isomorphism game. This allows us to define quantum information versions of 
quantum isomorphisms of non-commutative graphs of different types, which we characterise in terms of 
relations arising from the underlying graphs.

%%%%%%%%%%%%%%%%%%%%%%%%%%%%%%%%%%%%%%%%%%%%%%%%

\subsection{Quantum commuting isomorphisms}\label{ss_qci}

Several related concepts of quantum graphs have been studied in the literature (see 
\cite{bcehpsw, daws, dw}). Here we work with the notion 
that is used in \cite{tt-QNS}, \cite{stahlke} and \cite{bhtt}.
Let $X$ be a finite set, $H = \bb{C}^X$, and recall that 
$H^{\rm d}$ stands for the dual (Banach) space of $H$. 
Note that, as an additive group, $H^{\rm d}$ can be identified with $H$; we 
write $\bar{\zeta}$ for the element of $H^{\rm d}$, corresponding to the vector $\zeta$ in $H$
(so that $\bar{\zeta} : H\to \bb{C}$ is given by $\bar{\zeta}(\xi) = \langle \xi,\zeta\rangle$). 
Let $\theta : H\otimes H\to \cl L(H^{\dd},H)$ be the linear map given by 
$$\theta(\xi\otimes\eta)(\bar{\zeta}) = \langle \xi,\zeta\rangle\eta, \ \ \ \zeta\in H.$$
We have
\begin{equation}\label{eq_MNtheta}
\theta((S\otimes T)\zeta) = T\theta(\zeta)S^{\rm d}, \ \ \ \zeta\in H\otimes H, \ S,T\in \cl L(H).
\end{equation}
For a subspace $\cl U\subseteq \bb{C}^X\otimes \bb{C}^X$, set 
$$\cl S_{\cl U} = \{\theta(\zeta) : \zeta\in \cl U\}.$$
We let $\partial_X : (\mathbb C^X)^{\rm d}\to\mathbb C^X$ be the linear mapping given by 
$\partial_X(\bar{e}_x) = e_x$, $x\in X$, and we set $\tilde{\cl S}_{\cl U} := \cl S_{\cl U}\partial_X^{-1}$; 
thus, $\tilde{\cl S}_{\cl U} \subseteq \cl L(\bb{C}^X)$.

We denote by $\mm : \bb{C}^X\otimes \bb{C}^X\to \bb{C}$ the map, given by 
$$\mm(\zeta) = \left\langle \zeta, \sum_{x\in X} e_x\otimes e_x\right\rangle, \ \ \ \zeta\in \bb{C}^X\otimes \bb{C}^X.$$
Let also $\frak{f} : \bb{C}^X\otimes \bb{C}^X \to \bb{C}^X\otimes \bb{C}^X$ be the flip operator, given by $\frak{f}(\xi\otimes\eta) = \eta\otimes\xi$.

\begin{definition}\label{d_ss}
A \emph{quantum graph} with vertex set $X$ 
is a linear subspace $\cl U\subseteq \bb{C}^X \otimes \bb{C}^X$ that is \emph{skew} in that $\mm(\cl U) = \{0\}$ and 
\emph{symmetric} in that $\frak{f}(\cl U) = \cl U$. 
\end{definition}

In the sequel, for a subspace $\cl U\subseteq \bb{C}^X \otimes \bb{C}^X$, we denote by $P_{\cl U}$ the 
orthogonal projection from $\bb{C}^X \otimes \bb{C}^X$ onto $\cl U$; thus, $P_{\cl U}\in M_{XX}$. 
For a classical (simple, undirected) graph $G$ with vertex set $X$, we use $\sim$ (or $\sim_G$ when a clarification is needed)
to denote the adjacency relation of $G$. The graph $G$ gives rise to the quantum graph
$$\cl U_G = {\rm span}\{e_x\otimes e_y : x\sim y\},$$
and we write $P_G = P_{\cl U_G}$; note that $P_G\in \cl D_{XX}$, and that 
$$ \tilde{\cl S}_{\cl U_G}
= {\rm span}\{\epsilon_{x,y} : x\sim y\}$$
is a traceless self-adjoint subspace of $M_X$.  More generally,  $\tilde{\cl S}_{\cl U} \subseteq M_X$  is always a traceless transpose-invariant subspace for any quantum graph $\cl U$; this is the 
suitable version arising in our setting of Stahlke's quantum graphs \cite{stahlke}, 
where tracelessness and self-adjointness are assumed as part of the definition.

To motivate Definition \ref{d_qgig} below, we first recall the graph isomorphism game \cite{amrssv}
for graphs $G$ and $H$, both with vertex set $X$. 
For elements $x,y\in X$, we denote by ${\rm rel}_G(x,y)$ the element of the set $\{=, \sim, \not\simeq\}$,
which describes the adjacency relation in the pair $(x,y)$, in the graph $G$. 
A correlation $p\in \cl C_{\rm t}$ is said to be a perfect ${\rm t}$-strategy for the $(G,H)$-isomorphism game,
provided $p$ is bisynchronous and
\begin{equation}\label{eq_HxyGab}
p(a,b|x,y) = 0, \mbox{ if } {\rm rel}_G(x,y)\neq {\rm rel}_H(a,b) \mbox{ or } 
{\rm rel}_H(x,y)\neq {\rm rel}_G(a,b).
\end{equation}
We note that, for a given correlation type ${\rm t}$, 
two graphs $G$ and $H$ with vertex set $X$ are ${\rm t}$-isomorphic \cite{amrssv} if and only if 
there exists a bisynchronous bicorrelation $p$ of type $\rm t$ over the quadruple $(X,X,X,X)$, such that 
\begin{equation}\label{eq_gi1dir}
\omega\in \cl D_{XX}^+ \mbox{ and } \omega = P_G\omega P_G 
\ \Longrightarrow \ \Gamma(\omega) = P_H\Gamma(\omega) P_H
\end{equation}
and 
\begin{equation}\label{eq_gi2dir}
\sigma\in \cl D_{XX}^+  \mbox{ and }  \sigma = P_H\sigma P_H 
\ \Longrightarrow \ \Gamma^*(\sigma) = P_G\Gamma^*(\sigma) P_G.
\end{equation}
Indeed, condition (\ref{eq_gi1dir}) is equivalent to requiring that 
$p(a,b|x,y) = 0$ if $x\sim_G y$ but $a\not\sim_H b$, while (\ref{eq_gi2dir}) is equivalent to 
requiring that 
$p(a,b|x,y) = 0$ if $a\sim_H b$ but $x \not\sim_G y$, in conjunction, these two conditions 
are equivalent to (\ref{eq_HxyGab}).

Recall \cite{tt-QNS, bhtt} that, if 
$\cl U\subseteq \bb C^{X}\otimes\bb{C}^X$ and $\cl V\subseteq \bb C^{X}\otimes\bb{C}^X$ are quantum graphs, 
and $P = P_{\cl U}$ and $Q = P_{\cl V}$,  then
the perfect strategies for the 
\emph{quantum homomorphism game} $\cl U\to \cl V$
are the QNS correlations $\Gamma : M_{XX}\to M_{XX}$ such that  
$$\omega\in M_{XX}^+ \mbox{ and } \omega = P\omega P \ \Longrightarrow \ \Gamma(\omega) = Q\Gamma(\omega) Q.$$

\begin{definition}\label{d_qgig}
Let ${\rm t} \in \{{\rm loc}, {\rm q}, {\rm qa}, {\rm qc}, {\rm ns}\}$. 
We say that $\cl U$ and $\cl V$ are \emph{${\rm t}$-isomorphic}, 
and write $\cl U\cong_{\rm t} \cl V$, if there exists 
$\Gamma \in \cl Q_{\rm t}^{\rm bic}$ such that 
\begin{itemize}
\item[(i)] $\Gamma$ is a perfect strategy for $\cl U\to\cl V$, and 
\item[(ii)] $\Gamma^*$ is a perfect strategy for $\cl V\to\cl U$.
\end{itemize}
\end{definition}

\begin{remark}\label{r_moregene}
\rm 
Although our main interest in this section 
lies in quantum graphs, 
it is important to note, for the development in 
Section \ref{s_connections}, that Definition \ref{d_qgig} can be stated in a greater generality, involving subspaces
$\cl U$ and $\cl V$ of $\bb{C}^X\otimes\bb{C}^X$ that are not necessarily quantum graphs. 
\end{remark}

In the next theorem, we give an operator algebraic characterisation of the 
relation $\cl U \cong_{\rm qc} \cl V$. 
We recall the leg numbering notation: 
if $\frak{F} : M_{XX}\otimes \cl B(H)\to M_{XX}\otimes \cl B(H)$ is the 
(unitarily implemented) isomorphism, given by 
$$\frak{F}(S\otimes T\otimes R) = T\otimes S\otimes R, \ \ \ S,T\in M_X, R\in \cl B(H),$$
for $U = (U_{a,x})_{a,x}\in M_X\otimes\cl B(H)$, 
we write $U_{2,3} = I_X\otimes U$, and 
$U_{1,3} = \frak{F}(I_X\otimes U)$.
Note that $U_{2,3}, U_{1,3}\in M_{XX}\otimes \cl B(H)$ and 
\begin{equation}\label{eq_legno2}
U^{\rm t}_{1,3} U^*_{2,3}=\sum_{x,y,a,b\in X}\epsilon_{x,a}\otimes\epsilon_{y,b}\otimes U_{a,x}U_{b,y}^*.
\end{equation}

For the formulation of the next theorem, we set $\bar{A} = A^{{\rm t }*}$, and call 
a von Neumann algebra \emph{tracial} if it admits a tracial state. 
If $H$ is a Hilbert space and $\cl N\subseteq\cl B(H)$ is a von Neumann algebra, 
an operator matrix $U = (U_{a,x})_{a,x\in X}$ will be called \emph{$\cl N$-aligned} if 
$U_{a,x}^*U_{b,y}\in\cl N$ for all $x,y,a,b\in X$.

\begin{theorem}\label{iso}
Let $\cl U$ and $\cl V$ be 
quantum graphs in $\bb{C}^{X}\otimes\bb{C}^X$, and set $P = P_{\cl U}$ and $Q = P_{\cl V}$. 

The following are equivalent: 

\begin{itemize}
\item[(i)]
$\cl U \cong_{\rm qc} \cl V$;

\item[(ii)]
there exists a tracial von Neumann algebra $\cl N\subseteq \cl B(H)$
and an $\cl N$-aligned bi-unitary $U = (U_{a,x})_{a,x}\in M_X(\cl B(H))$ such that  
$$(P\otimes I)U^{\rm t}_{1,3}U^*_{2,3}(Q^\perp\otimes I) = 0 
\ \mbox{ and } \ (\bar{P}^\perp\otimes I)U^{\rm t}_{1,3}  U_{2,3}^*(\bar{Q}\otimes I)=0;$$

\item[(iii)] there exists a tracial von Neumann algebra $\cl N\subseteq \cl B(H)$
and an $\cl N$-aligned bi-unitary $U = (U_{a,x})_{a,x}\in M_X(\cl B(H))$ such that  
$$U(\tilde{\cl S}_{\cl U}\otimes 1)U^*\subseteq \tilde{\cl S}_{\cl V} \otimes \cl B(H) \text{ and }U^{\rm t}
(\tilde{\cl S}_{\cl V} \otimes 1)U^{{\rm t}^*}\subseteq \tilde{\cl S}_{\cl U} \otimes \cl B(H).$$
\end{itemize}
\end{theorem}

\begin{proof}
(i)$\Rightarrow$(ii) 
For a vector 
$\xi = \sum_{x,y\in X} \alpha_{x,y}e_x\otimes e_y \in \bb{C}^{X} \otimes \bb{C}^{X}$, let 
$\overline{\xi} = \sum_{x,y\in X} \overline{\alpha}_{x,y}e_x\otimes e_y$ and set 
$$Y_\xi = \sum_{x,y\in X} \alpha_{x,y} \epsilon_{x,y};$$
note that $Y_\xi\in M_X$ (and that the use of the notation $\overline{\xi}$ agrees, up to 
a canonical identification, 
with the definition in the beginning of Subsection \ref{ss_qci}). 
Let $\Gamma : M_{XX}\to M_{XX}$ be a concurrent quantum commuting bicorrelation 
satisfying conditions (i) and (ii) in Definition \ref{d_qgig}.

By Theorem \ref{QNSbicorrelation}, there exists a tracial state $\tau:C(\mathbb P \cl U_X^+)\to\mathbb C$ such that 
$$\Gamma(e_{x,x'}\otimes e_{y,y'}) = \left(\tau(u_{x,x',a,a'}u_{y',y,b',b})\right)_{a,a',b,b'}, \ \ \ x,x',y,y'\in X.$$
Let $\pi_\tau$ be the *-representation, associated with $\tau$ via the GNS construction, 
and let $\zeta$ be the corresponding cyclic vector. Then 
$\cl N = \pi_\tau(C(\mathbb P \cl U_X^+))''$ is a finite von Neumann algebra, on which the vector state corresponding to $\zeta$ is faithful and tracial. 

Let $E = (\pi_\tau(u_{x,x',a,a'}))_{x,x',a,a'}$. 
As in the proof of \cite[Theorem 5.5]{bhtt}, 
we have that 
$$\langle\Gamma(\xi\xi^*)\eta,\eta\rangle=
\left(\Tr\otimes\tau\right) 
\left(E(Y_{\bar\xi}\otimes Y_{\eta}\otimes 1_{\cl A})E(Y_{\bar\xi}^*\otimes Y_\eta^*\otimes1_{\cl N})\right),$$
implying, by the faithfulness of $\tau$, that 
$$E\left(Y_{\bar\xi}\otimes Y_\eta\otimes I\right)E=0, \ \ \ \xi\in \cl U, \eta\in \cl V^\perp.$$

By Lemma \ref{l_J1J2} (ii), there exists a bi-unitary $U=(U_{a,x})_{a,x}$, such that  
$E = (U_{a,x}^* U_{a',x'})_{x,x',a,a'}$. Writing $\xi=\sum_{x,y\in X}\alpha_{x,y} e_x\otimes e_y$ and $\eta=\sum_{a,b\in X}\beta_{a,b}e_a\otimes e_b$, we calculate
\begin{eqnarray*}
&&E\left(Y_{\bar\xi}\otimes Y_\eta\otimes I\right)E = 
\left(\sum_{x',y',a',b'\in X}\overline{\alpha_{x',y'}}\beta_{a',b'}U_{a,x}^*U_{a',x'}U_{b',y'}^*U_{b,y}\right)_{x,y,a,b}.
\end{eqnarray*}
Hence 
$\sum_{x',y',a',b'}\overline{\alpha_{x',y'}}\beta_{a',b'}U_{a,x}^*U_{a',x'}U_{b',y'}^*U_{b,y}=0$ for any $x$, $y$, $a$, $b$. Letting $R_{\xi,\eta} = \sum_{x',y',a',b'}\overline{\alpha_{x',y'}}\beta_{a',b'}U_{a',x'}U_{b',y'}^*$, 
we have 
$$U_{a,x}^*R_{\xi,\eta} U_{b,y} = 0, \ \ \ x,y,a,b\in X.$$
It follows that 
\begin{equation}\label{eq_Rxieta}
R_{\xi,\eta} = 
\sum_{x,y\in X} U_{a,x} U_{a,x}^*R_{\xi,\eta} U_{b,y}U_{b,y}^* = 0.
\end{equation}

Let 
$F := U^{\rm t}_{1,3} U^*_{2,3}$; thus, 
$F \in M_{XX}\otimes \cl B(H)$.
By (\ref{eq_Rxieta}), the operator $F$ satisfies the conditions
$$ \langle F(\eta\otimes h), \xi\otimes g\rangle = 0, \ \ \ h,g\in H,$$
which imply
$(P\otimes I)F(Q^\perp\otimes I)=0$.

Let 
$\tilde{E} := (U_{a,x}^* U_{a',x'})_{a,a',x,x'}$. 
By symmetry, 
$$\tilde{E} (Y_{\bar\xi'}\otimes Y_{\eta'}\otimes I) \tilde{E} = 0, \ \ \ \xi'\in \cl V, \eta'\in \cl U^\perp.$$
Setting 
$$\tilde{F} := U_{1,3}\bar U_{2,3} = \sum_{x,y,a,b\in X} \epsilon_{a,x} \otimes \epsilon_{b,y} \otimes  U_{a,x} U_{b,y}^*,$$
we similarly obtain that 
$$ \langle \tilde{F}(\eta'\otimes h), \xi'\otimes g\rangle = 0, \ \ \ \xi'\in \cl V, \eta'\in \cl U^\perp, h,g\in H,$$
and hence 
\begin{equation}\label{eq_F'}
(Q\otimes I)\tilde{F} (P^\perp\otimes I) = 0.
\end{equation}

Let $\frak t : M_X\to M_X$ be the map, given by 
$\frak{t}(T) = T^{\rm t}$. 
Since the operators $P^{\perp}$ and $Q$ are self-adjoint, 
$(\frak t\otimes\frak t) (Q) = \bar{Q}$ and
$(\frak t\otimes\frak t) (P^{\perp})=\bar{P}^{\perp}$.
Thus, applying the map 
$\frak t\otimes\frak t\otimes {\rm id}$ to the relation (\ref{eq_F'}), we obtain 
$(\bar{P}^\perp \otimes I)F(\bar{Q}\otimes I)=0$

(ii)$\Rightarrow$(i) 
Assume that 
$(P\otimes I)U^{\rm t}_{1,3}U^*_{2,3}(Q^\perp\otimes I)=0$ and 
$(\bar{P}^\perp \otimes I)U^{\rm t}_{1,3}U^*_{2,3}(\bar{Q}\otimes I) = 0$. 
By Theorem \ref{QNSbicorrelation} (i), the linear map $\Gamma$, given by 
$\Gamma(\epsilon_{x,x'}\otimes \epsilon_{y,y'}) = 
\left(\tau((U_{a,x}^* U_{a',x'} U_{b',y'}^* U_{b,y})\right)_{a,a',b,b'}$, is a concurrent quantum commuting bicorrelation. 
Reversing the arguments from the previous paragraphs and using 
the proof of \cite[Theorem 5.5]{bhtt}, we obtain that, if $E = (U_{a,x}^*U_{a',x'})_{x,x',a,a'}$ then 
$$\langle \Gamma(\xi\xi^*),\eta\eta^*\rangle
= (\Tr\otimes\tau)\left(E(Y_{\bar\xi}\otimes Y_\eta\otimes I)E(Y_{\bar\xi}^*\otimes Y_\eta^*\otimes I)\right) = 0,$$
for all $\xi\in\cl U$ and all $\eta\in \cl V^\perp$.
Similarly, 
$$\langle \Gamma^*(\xi'\xi'^{*}),\eta'\eta'^{*}\rangle =0 \text{ for all }\xi'\in\cl V, \eta'\in\cl U^\perp.$$
It follows that $\cl U \cong_{\rm qc} \cl V$ via $\Gamma$. 

(ii)$\Rightarrow $(iii) For each $\xi\in \cl U$, $\eta\in \cl V^\perp$, $h$, $g\in H$, we have 
\begin{equation}\label{orth}
\langle U_{1,3}^{\rm t} U_{2,3}^*(\eta\otimes h),\xi\otimes g\rangle=\langle \eta\otimes h, U_{2,3}\bar U_{1,3}(\xi\otimes g)\rangle=0.
\end{equation}
Consider $U_{2,3}\bar U_{1,3}$ as a linear operator on $\bb{C}^{XX}\otimes \cl B(H)$ by letting
$$(U_{2,3}\bar U_{1,3})(\xi\otimes T) \hspace{-0.1cm}  := 
\hspace{-0.5cm} \sum_{x,y,a,b\in X}
\hspace{-0.2cm}  (\epsilon_{b,y}\otimes\epsilon_{a,x})\xi\otimes U_{a,x}U_{b,y}^*T,
\ \xi\in \bb C^{XX}, T\in \cl B(H).$$
Fix  $\xi\in \bb{C}^{XX}$. We have 
\begin{eqnarray*}
& & 
(\theta\otimes \text{id})(U_{2,3}\bar U_{1,3}(\xi\otimes I))\\
& = & 
\hspace{-0.6cm} \sum_{x,y,a,b\in X}
\theta((\epsilon_{b,y}\otimes\epsilon_{a,x})\xi)\otimes U_{a,x}U_{b,y}^*
= 
\hspace{-0.4cm} \sum_{x,y,a,b\in X}
\epsilon_{a,x}\theta(\xi)\epsilon_{b,y}^{\rm d}\otimes U_{a,x}U_{b,y}^*\\
& = &
\hspace{-0.3cm} \left(\sum_{a,x\in X}\epsilon_{a,x}\otimes U_{a,x}\right)(\theta(\xi)\otimes I)
\left(\sum_{b,y\in X}\epsilon_{b,y}^{\rm d}\otimes U_{b,y}^*\right).
\end{eqnarray*}
Note that $\partial_X\epsilon_{b,y}^{\rm d}\partial_X^{-1} = \epsilon_{y,b}$. Therefore,
\begin{eqnarray}\label{eq_U23tbu}
& &
(\theta\otimes \text{id})(U_{2,3}\bar U_{1,3}(\xi\otimes I))(\partial_X^{-1}\otimes I) \nonumber\\
& = &
\left(\sum_{a,x\in X}\epsilon_{a,x}\otimes U_{a,x}\right)
\left(\theta(\xi)\partial_X^{-1}\otimes I\right)
\left(\sum_{b,y\in X}\partial_X\epsilon_{b,y}^{\rm d}\partial_X^{-1}\otimes U_{b,y}^*\right)\\
& = &
U(\theta(\xi)\partial_X^{-1}\otimes I)U^*. \nonumber
\end{eqnarray}

To see that 
$U(\tilde{\cl S}_{\cl U} \otimes 1)U^*\subseteq \tilde{\cl S}_{\cl V} \otimes \cl B(H)$, let $\xi\in \cl U$, and 
fix orthonormal bases $(\eta_i)_{i\in \bb{I}}$ and $(\zeta_j)_{j\in \bb{J}}$ of 
$\cl V$ and $\cl V^{\perp}$, respectively.
Then $$U(\theta(\xi)\partial_X^{-1}\otimes I)U^*
= \sum_{i\in \bb{I}} \theta(\eta_i)\partial_X^{-1}\otimes R_i
+ \sum_{j\in \bb{J}} \theta(\zeta_j)\partial_X^{-1}\otimes S_j$$ for 
some $R_i$, $S_j\in\cl B(H)$, $i\in \bb{I}$, $j\in \bb{J}$.
From the previous arguments we obtain
$$(\theta\otimes \text{id})
(U_{2,3}\bar U_{1,3}(\xi\otimes I)) = 
(\theta\otimes\text{id}) \left(\sum_{i\in \bb{I}} \eta_i\otimes R_i 
+ \sum_{j\in \bb{J}} \zeta_j\otimes S_j\right)$$
and 
$$U_{2,3}\bar U_{1,3}(\xi\otimes I) = \sum_{i\in \bb{I}} \eta_i\otimes R_i 
+ \sum_{j\in \bb{J}} \zeta_j\otimes S_j.$$
Let $\omega_{g,h}$ be the vector functional on $\cl B(H)$, given by 
$\omega_{g,h}(T)=\langle Tg,h\rangle$
and, for $\eta\in \bb C^{XX}$, 
let $\ell_\eta$ be the linear functional on $\mathbb C^{XX}$, given by 
$\ell_\eta(\xi)=\langle \xi,\eta\rangle$. 
Then 
\begin{eqnarray*}
(\ell_\eta\otimes\omega_{g,h})(U_{2,3}\bar U_{1,3}(\xi\otimes I)))
& = &
\sum_{x,y,a,b \in X}\langle (\epsilon_{x,a}\otimes\epsilon_{y,b})\xi,\eta\rangle \langle U_{a,x}U_{b,y}^*g,h\rangle\\
& = &
\langle U_{2,3}\bar U_{1,3}(\xi\otimes g),\eta\otimes h\rangle, 
\end{eqnarray*}
while
$$(\ell_\eta\otimes\omega_{g,h})
\hspace{-0.15cm}\left(\hspace{-0.1cm}\sum_{i\in \bb{I}} \eta_i
\hspace{-0.02cm}\otimes \hspace{-0.02cm} R_i 
\hspace{-0.05cm}+\hspace{-0.1cm}\sum_{j\in \bb{J}}\zeta_j\hspace{-0.02cm}\otimes \hspace{-0.02cm}S_j\hspace{-0.1cm}\right)
\hspace{-0.15cm} = \hspace{-0.1cm} 
\sum_{i\in \bb{I}} \langle\eta_i,\hspace{-0.02cm}\eta\rangle\hspace{-0.05cm}\langle R_ig,\hspace{-0.02cm}h\rangle
\hspace{-0.05cm} + \hspace{-0.1cm} \sum_{j\in \bb{J}} \langle\zeta_j,\hspace{-0.02cm}\eta\rangle\hspace{-0.05cm}
\langle S_jg,\hspace{-0.02cm}h\rangle\hspace{-0.01cm}.$$
Taking now $\eta = \zeta_j$ we obtain from (\ref{orth}) that 
$$(\ell_\eta\otimes\omega_{g,h})\left(U_{2,3}\bar U_{1,3}(\xi\otimes I))\right) = 0$$
and that 
$$(\ell_\eta\otimes\omega_{g,h})
\left(\sum_{i\in \bb{I}} \eta_i\otimes R_i +\sum_{j\in \bb{J}} \zeta_j\otimes S_j\right)
= \|\zeta_j\|^2\langle S_jg,h\rangle;$$
thus, $\langle S_jg,h\rangle = 0$. As $g$ and $h$ can be chosen arbitrarily, $S_j=0$ 
for all $j\in \bb{J}$. 
Therefore 
$$U(\theta(\xi)\partial_X^{-1}\otimes I)U^*
= \sum_{i\in \bb{I}} \theta(\eta_i)\partial_X^{-1}\otimes R_i\subseteq \tilde{\cl S}_{\cl V}\otimes\cl B(H).$$
Similar arguments applied to $(Q\otimes I)\tilde{F}(P^\perp\otimes I)=0$, where $\tilde{F} = U_{1,3}\bar U_{2,3}$, give
$$U^{\rm t}(\tilde{\cl S}_{\cl V}\otimes 1)(U^{\rm t})^*\subseteq \tilde{\cl S}_{\cl U} \otimes \cl B(H).$$
(iii)$\Rightarrow$(ii)
follows, using (\ref{eq_U23tbu}), by reversing the arguments in the implication (ii)$\Rightarrow$(iii).
\end{proof}

\smallskip

\noindent {\bf Remarks. (i)}
The arguments in the proof of Theorem \ref{iso} can be used to conclude that $\cl U\to^{\rm qc}\cl V$ if and only if 
there exists a tracial von Neumann algebra $\cl N\subseteq \cl B(H)$
and an $\cl N$-aligned isometry  $V = (V_{a,x})_{a,x}$, $V_{a,x}\in \cl B(H)$, such that  
$$V(\tilde{\cl S}_{\cl U}\otimes 1)V^*\subseteq \tilde{\cl S}_{\cl V}\otimes \cl B(H).$$
This complements the characterisation obtained in \cite[Theorem 5.7]{bhtt}. 

\smallskip

{\bf (ii)} 
Similar results to those of Theorem \ref{iso} hold for $\cl U\simeq_{\rm q}\cl V$, 
in which case the space $H$ is finite-dimensional.
A treatment of the case $\cl U\simeq_{\rm loc}\cl V$ is presented in Subsection \ref{ss_locqiso} below.

\begin{corollary}\label{c_claco}
Let $G$ and $H$ be graphs with vertex set $X$.
The following are equivalent: 

\begin{itemize}
\item[(i)]
$\cl U_G \cong_{\rm qc} \cl U_H$;

\item[(ii)]
there exists a tracial von Neumann algebra $\cl N\subseteq \cl B(H)$
and an $\cl N$-aligned bi-unitary $U = (U_{a,x})_{a,x}\in M_X(\cl B(H))$ such that  
$U_{a,x} U_{b,y}^*=0$ if either $x\sim_G y$ and $a\not\sim_H b$, 
or $x\not\sim_Gy$ and $a\sim_H b$;

\item [(iii)]
there exists a tracial von Neumann algebra $\cl N\subseteq \cl B(H)$
and an $\cl N$-aligned bi-unitary $U = (U_{a,x})_{a,x}\in M_X(\cl B(H))$ such that 
$$(P_G\otimes I) U_{1,3}^{\rm t}U_{2,3}^* = U_{1,3}^{\rm t} U_{2,3}^*(P_H\otimes I);$$

\item [(iv)]
there exists a tracial von Neumann algebra $\cl N\subseteq \cl B(H)$
and a bi-unitary $U = (U_{a,x})_{a,x}\in M_X(\cl B(H))$ such that $U_{a,x}^*U_{b,y}\in\cl N$, $x,y,a,b\in X$, 
and 
$$ U(\cl S_G\otimes 1)U^*\subseteq \cl S_H\otimes \cl B(H)\text{ and }
U^{\rm t}(\cl S_H\otimes 1)U^{{\rm t}*}\subseteq \cl S_G\otimes\cl B(H).$$
\end{itemize}
\end{corollary}

\begin{proof}
We have 
$\cl U_G = \{e_x\otimes e_y: x\sim_G y\}$ and 
$\cl U_H = \{e_a\otimes e_b: a\sim_H b\}$. 
As $\bar P_G = P_G$ and $\bar P_H=P_H$, the conditions
\begin{equation}\label{defrel}
(P_{G}\otimes I)U^{\rm t}_{1,3}U^*_{2,3}(P_{H}^\perp\otimes I) = 0 \text{ and }(\bar{P}_{G}^\perp\otimes I)U^{\rm t}_{1,3}  U_{2,3}^*(\bar{P}_H\otimes I) = 0
\end{equation}
are equivalent to $(P_G\otimes I) U_{1,3}^{\rm t}U_{2,3}^* = U_{1,3}^{\rm t} U_{2,3}^*(P_H\otimes I)$, and also equivalent to $U_{a,x} U_{b,y}^* = 0$ if either $x\sim_G y$ and $a\not\sim_H b$ or $x\not\sim_Gy$ and $a\sim_H b$. The statement now follows from Theorem \ref{iso}.
\end{proof}

\begin{remark} \label{r_magic_vs_bi}
\rm
The conditions on the bi-unitary $U$ contained in Corollary \ref{c_claco} 
are equivalent to the conditions 
$A_{H^c}\ast U(A_G\otimes I)U^*=0$ and $A_{G^c}\ast U^{\rm t}(A_H\otimes I)\bar U = 0$, where  $G^c$ is the complement to $G$ and $\ast$ denotes the Schur product.  
We can formulate a similar characterisation for types $\rm loc$ and ${\rm q}$. 
In the case when the bi-unitary $U$ is actually a quantum permutation (that is, the entries $u_{i,j}$ of $U$ are all orthogonal projections), 
these conditions are equivalent to the condition that $U(A_G \otimes I)U^* = A_H \otimes I$. 
Indeed, if $U$ is a quantum permutation satisfying 
$A_{H^c} \ast U(A_G \otimes I)U^*=0$, then whenever $i \neq j$ and $i \not\sim_H j$, we have
\[0=(U(A_G \otimes I)U^*)_{i,j}=\sum_{k \sim_G \ell} u_{i,k}u_{j,\ell}.\]
Multiplying on the left by $u_{i,k}$ for any fixed $k$ satisfying $k \sim_G \ell$, we obtain $u_{i,k}u_{j,\ell}=0$ 
whenever $i \not\sim_H j$, $i \neq j$ and $k \sim_G \ell$. 
Similarly, if $i=j$ and $k \sim_G \ell$, then $k \neq \ell$, so that $u_{i,k}u_{j,\ell}=0$.

Next, if we interchange the roles of $G$ and $H$ in the above argument and replace $U$ with the magic unitary 
$U^{\rm t}$, the identity $A_{G^c} \ast U^{\rm t} (A_H \otimes I) \bar U=0$ yields  $u_{k,i}u_{\ell, j}=0$ whenever $i \not\sim_G j$, $i \neq j$ and $k \sim_H \ell$ or whenever $i=j$, and $k \sim_H \ell$.

It follows that, if $i \sim_H j$, then (assuming that $n = |X|$) we have 
\begin{align*}
(U (A_G \otimes I)U^*)_{i,j}
& =\sum_{k \sim_G \ell} u_{i,k}u_{j,\ell} \\
&=\sum_{k=1}^n u_{i,k}u_{j,k} + \sum_{k \sim_G \ell} u_{i,k}u_{j,\ell}
+ \sum_{\substack{k \not\sim_G \ell \\ k \ne l}} u_{i,k}u_{j,\ell}\\
&= \sum_{k,\ell=1}^n u_{i,k}u_{j,\ell} 
= \left(\sum_{k=1}^n u_{i,k}\right)\left(\sum_{\ell=1}^n u_{j,\ell}\right) =1=(A_H)_{i,j}.
\end{align*}
Similarly, if $i \not\sim_H j$, then either $i=j$ or $i \sim_{H^c}j$, and we obtain in either case
\[
(U (A_G \otimes I)U^*)_{i,j} = \sum_{k \sim_G \ell} u_{i,k}u_{j,\ell} = 0 = (A_H)_{i,j}.
\]
It follows that $U(A_G \otimes I)U^*=A_H \otimes I$. The converse is immediate.
\end{remark}

%%%%%%%%%%%%%%%%%%%%%%%%%%%%%%%%%%%%%%%%%%%%%%%%%%%%%

\subsection{Local isomorphisms}\label{ss_locqiso}

In this subsection, we restrict our attention to quantum graph isomorphisms of local type. 

\begin{proposition}\label{p_loctype}
Let $X$ be a finite set, and $\cl U$ and $\cl V$ be quantum graphs in $\bb{C}^X\otimes\bb{C}^X$. 
The following are equivalent: 
\begin{itemize}
\item[(i)] $\cl U\cong_{\rm loc} \cl V$;

\item[(ii)] there exists a unitary $U \in M_X$ such that $(U\otimes \bar{U})(\cl U) = \cl V$. 
\end{itemize}
\end{proposition}

\begin{proof}
(i)$\Rightarrow$(ii) 
Let $\Gamma\in \cl Q_{\rm loc}^{\rm bic}$ is a correlation satisfying the conditions of Definition \ref{d_qgig}
for quantum graphs $\cl U$ and $\cl V$. By Theorem \ref{QNSbicorrelation} (iv), 
$\Gamma = \sum_{i=1}^k \lambda_i \Phi_i\otimes \Phi_i^{\sharp}$ as a convex combination, 
where $\Phi_i : M_X\to M_X$ is a unitary quantum channel, $i = 1,\dots,k$. 
Conditions (i) and (ii) in Definition \ref{d_qgig} are equivalent to
\begin{equation}\label{eq_PUPVperp}
\left\langle \Gamma(P_{\cl U}),P_{\cl V}^{\perp}\right\rangle = 0 
\ \mbox{ and } \ \left\langle \Gamma^*(P_{\cl V}),P_{\cl U}^{\perp}\right\rangle = 0.
\end{equation}
The monotonicity of the trace functional now implies that 
$\Phi_i\otimes \Phi_i^{\sharp}$ satisfies the conditions in Definition \ref{d_qgig} for every $i = 1,\dots,k$. 
We may thus assume that 
$\Gamma = \Phi\otimes \Phi^{\sharp}$, where $\Phi : M_X\to M_A$ is a
unitary quantum channel. Let $U\in M_X$ be a unitary such that 
$\Phi(\omega) = U^*\omega U$, $\omega\in M_X$. 
A direct verification shows that 
$$\Phi^{\sharp}(\omega) = \bar{U}^*\omega\bar{U}, \ \ \ \omega\in M_X.$$
Thus, 
$$\Gamma(\omega) = (U\otimes \bar{U})^*\omega(U\otimes \bar{U}), \ \ \ \omega\in M_{XX}.$$
The first condition in (\ref{eq_PUPVperp}) now implies that, for every $\xi\in \cl U$, we have 
$$\left((U\otimes \bar{U})^*\xi\right)\left((U\otimes \bar{U})^*\xi\right)^* 
= (U\otimes \bar{U})^*(\xi\xi^*) (U\otimes \bar{U}) \leq P_{\cl V},$$
that is, $(U\otimes \bar{U})^*(\cl U)\subseteq \cl V$. 
On the other hand, 
$$\Gamma^*(\omega) = (U\otimes \bar{U})\omega(U\otimes \bar{U})^*, \ \ \ \omega\in M_{XX},$$
and arguing by symmetry implies that $(U\otimes \bar{U})(\cl V)\subseteq \cl U$; thus, 
(ii) follows.

(ii)$\Rightarrow$(i) 
Given a unitary $U\in M_X$, let $\Phi(\omega) = U^*\omega U$, $\omega\in M_X$, 
and $\Gamma = \Phi\otimes\Phi^{\sharp}$. 
Then the arguments in the first part of the proof imply that $\cl U\cong_{\rm loc}\cl V$ via $\Gamma$. 
\end{proof}

\noindent {\bf Remark. }
Proposition \ref{p_loctype} 
can equivalently be seen as a consequence 
of Theorem \ref{iso}. 
Indeed, note that, by Theorem \ref{QNSbicorrelation} (iv) and its proof, $\Gamma\in \cl Q_{\rm loc}^{\rm bic}$ if and only if $\Gamma=\sum_{i=1}^k\lambda_i\Gamma_i$ as a convex combination,  where 
$\Gamma_i(e_{x,x'}\otimes e_{y,y'})
= (\pi_i(u_{x,x',a,a'}u_{y',y,b',b}))_{a,a',b,b'}$ for some $*$-representation
$\pi_i : C(\mathbb P \cl U_X^+)\to\mathbb C$. 
Using the fact that all $\Gamma_i$ are positive,  it can be easily seen that one can assume that $k = 1$.
Let $U = (u_{a,x})_{a,x}\in M_X$ be the unitary that corresponds to $\pi_1$ as in the proof of the implication 
(i)$\Rightarrow$(ii); we have that 
$U$ satisfies the corresponding conditions (ii) and (iii). 
In particular, 
$U\tilde{\cl S}_{\cl U}U^*\subseteq \tilde{\cl S}_{\cl V}$ and 
$U^{\rm t}\tilde{\cl S}_{\cl V}(U^{\rm t})^*\subseteq \tilde{\cl S}_{\cl U}$. As $\tilde{\cl S}_{\cl U}^{\rm t} = \tilde{\cl S}_{\cl U}$ and $\tilde{\cl S}_{\cl V}^{\rm t} = \tilde{\cl S}_{\cl V}$, we obtain that $U^* \tilde{\cl S}_{\cl V} U\subseteq \tilde{\cl S}_{\cl U}$, which implies $U^*\tilde{\cl S}_{\cl V} U = \tilde{\cl S}_{\cl U}$. This gives in particular that $(U\otimes\bar U)(\cl U)=\cl V$.

\begin{proposition}\label{p_localch}
Let $G$ and $H$ be graphs with vertex set $X$. Then
$\cl U_G\cong_{\rm loc}\cl U_H$ if and only if $G\cong H$. 
\end{proposition}

\begin{proof}
A graph isomorphism $\nph : X\to X$ between $G$ and $H$ gives rise to a permutation unitary 
operator $U_{\nph} : \bb{C}^X\to \bb{C}^X$; letting $\Phi : M_X\to M_A$ be the conjugation by $U_{\nph}$, 
we have that the correlation $\Phi\otimes\Phi^{\sharp}$ implements an isomorphism $\cl U_G\cong_{\rm loc}\cl U_H$. 

Conversely, suppose that $\cl U_G\cong_{\rm loc}\cl U_H$. 
By Proposition \ref{p_loctype}, there exists a unitary $U\in M_X$ such that 
$(U\otimes \bar{U})(\cl U_G) = \cl U_H$. 
Letting 
$$\cl S_G = {\rm span}\{\epsilon_{x,y} : x\sim y \mbox{ or } x = y\},$$
we now have that $U\cl S_G U^* = \cl S_H$. By \cite[Proposition 3.1]{op}, $G\cong H$. 
\end{proof}

\begin{corollary}\label{c_qnotloc}
There exist quantum graphs $\cl U$ and $\cl V$ such that $\cl U\cong_{\rm q}\cl V$ but $\cl U\not\cong_{\rm loc}\cl V$.
\end{corollary}

\begin{proof}
By \cite[Theorem 6.4]{amrssv}, there exists graphs $G$ and $H$ such that $G\cong_{\rm q} H$ but $G\not\cong_{\rm loc} H$. 
By Proposition \ref{p_localch}, 
$\cl U_G\not\cong_{\rm loc}\cl U_H$; to complete the proof, we show that $\cl U_G\cong_{\rm q}\cl U_H$.
By \cite[Theorem 2.1]{lmr}, there exists a quantum permutation matrix $(P_{x,a})_{x,a}$, acting on a finite dimensional Hilbert space $H$, such that 
$$P_{x,a}P_{y,b} = 0 \ \mbox{ if }
x\sim_G y \ \& \ a\not\sim_H b, \ \mbox{ or } \ 
x\not\sim_G y \ \& \ a\sim_H b.$$
By Remark \ref{r_magic_vs_bi}, $\cl U_G\cong_{\rm q}\cl U_H$ 
\end{proof}

%%%%%%%%%%%%%%%%%%%%%%%%%%%%%%%%%%%%%%%%%%%%%%%%%%%%%
%%%%%%%%%%%%%%%    Quantum groups of fuzzy symmetries     %%%%%%%%%%%%%%%%%
%%%%%%%%%%%%%%%%%%%%%%%%%%%%%%%%%%%%%%%%%%%%%%%%%%%%%

\subsection{The quantum isomorphism algebra}\label{qgfs}

Let $X$ be a finite set, and $\cl U\subseteq \bb{C}^{XX}$ and 
$\cl V\subseteq \bb{C}^{XX}$ be quantum graphs. 
We will introduce a C*-algebra whose tracial 
properties reflect the properties of the isomorphism game $\cl U\cong \cl V$. 
Let $P$ (resp. $Q$) be the projection from $\bb{C}^{XX}$ onto $\cl U$
(resp. from $\bb{C}^{XX}$ onto $\cl V$). 
For matrices $S,T\in M_{XX}$, define a 
linear map
$$\gamma_{S,T} : M_{XX}\otimes C(\mathbb P \cl U_X^+) \otimes M_{XX}
\otimes C(\mathbb P \cl U_X^+)^{\rm op} \to C(\mathbb P \cl U_X^+)$$
by letting 
$$\gamma_{S,T}(\omega
\otimes u \otimes v^{\rm op}) 
= \Tr(\omega(S\otimes T)) uv, \ \ \ 
\omega\in M_{XX}\otimes M_{XX}, 
\ u,v\in C(\mathbb P \cl U_X^+).$$
Set $W = (u_{x,x',a,a'})_{x,x',a,a'}\in M_{XX}\otimes C(\mathbb P \cl U_X^+)$, and let 
$$\cl I_{P,Q} = \left\langle 
\gamma_{P,Q^{\perp}}\left(W\otimes W^{\rm op}\right), 
\gamma_{P^{\perp},Q}\left(W\otimes W^{\rm op}\right)\right\rangle$$ 
be the closed ideal in $C(\mathbb P \cl U_X^+)$, 
generated by the elements 
$\gamma_{P,Q^{\perp}}(W\otimes W^{\rm op})$ and 
$\gamma_{P^{\perp},Q}(W\otimes W^{\rm op})$.
Set $\cl A_{P,Q} = C(\mathbb P \cl U_X^+)/\cl I_{P,Q}$. 
We write $\dot{u}$ for the image of an element $u\in C(\mathbb P \cl U_X^+)$ in $\cl A_{P,Q}$ 
under the quotient map.

\begin{theorem}\label{th_alg}
Let $X$ be a finite set, $\cl U\subseteq \bb{C}^{XX}$ (resp. $\cl V\subseteq \bb{C}^{XX}$)
be a quantum graph and $P\in M_{XX}$ (resp. $Q\in M_{XX}$) be the projection onto $\cl U$
(resp. $\cl V$). 
The following are equivalent for a QNS bicorrelation $\Gamma : M_{XX}\to M_{XX}$:
\begin{itemize}
\item[(i)] $\Gamma$ is a perfect quantum commuting (resp. quantum/local) strategy for the isomorphism game 
$\cl U\cong \cl V$; 
\item[(ii)] there exists a trace $\tau$ 
(resp. a trace $\tau$ that factors through a finite dimensional/abelian *-representation)
of $\cl A_{P,Q}$ such that  
\begin{equation}\label{eq_agains}
\Gamma(\epsilon_{x,x'} \otimes \epsilon_{y,y'}) 
= \left(\tau(\dot{u}_{x,x',a,a'}\dot{u}_{y',y,b',b})\right)_{a,a',b,b'}, \ \ \ x,x',y,y'\in X.
\end{equation}
\end{itemize}
\end{theorem}

\begin{proof}
(i)$\Rightarrow$(ii)
We consider first the quantum commuting case.
By Theorem \ref{QNSbicorrelation}, there exists 
a tracial state $\tau : C(\mathbb P \cl U_X^+)\to\bb{C}$ such that 
$\Gamma = \Gamma_{\tau}$. 
Writing 
$$\tilde{W} = (u_{x,x',a,a'}u_{y',y,b',b})\in M_{XXXX}\otimes C(\mathbb P \cl U_X^+),$$
we thus have
\begin{eqnarray*}
& & 
\langle\Gamma(\epsilon_{x,x'}\otimes \epsilon_{y,y'}),
\epsilon_{a,a'}\otimes \epsilon_{b,b'}\rangle
= 
\tau(u_{x,x',a,a'}u_{y',y,b',b})\\
&  =  & 
{\rm Tr}(((\epsilon_{x,x'}\otimes \epsilon_{y,y'})\otimes
(\epsilon_{a,a'}\otimes \epsilon_{b,b'}))
\tau^{(XXXX)}(\tilde{W}))\\
& = & 
\tau(\gamma_{\epsilon_{x,x'}\otimes \epsilon_{y,y'},\epsilon_{a,a'}\otimes \epsilon_{b,b'}}
(W\otimes W^{\rm op})).
\end{eqnarray*}
By linearity, 
\begin{equation}\label{eq_ST}
\left\langle\Gamma(S),T\right\rangle = 
\tau\left(\gamma_{S,T}(W\otimes W^{\rm op})\right), 
\ \ \ S,T\in M_{XX}.
\end{equation}
Since $\Gamma$ is a perfect strategy for the game $\cl U\cong \cl V$, equation (\ref{eq_ST}) implies that 
$$\tau\left(\gamma_{P,Q^{\perp}}(W\otimes W^{\rm op})\right) 
= 
\tau\left(\gamma_{P^{\perp},Q}(W\otimes W^{\rm op})\right)
= 0.$$

Set $g = \gamma_{P,Q^{\perp}}(W\otimes W^{\rm op})$; 
we claim that $g\in C(\mathbb P \cl U_X^+)^+$. 
To see this, let 
$$\frak{m} : M_{XX}(C(\mathbb P \cl U_X^+))\otimes_{\max} M_{XX}(C(\mathbb P \cl U_X^+))^{\rm op}\to M_{XX}(C(\mathbb P \cl U_X^+))$$
be the multiplication map, and note that, if 
$u\in M_{XX}(C(\mathbb P \cl U_X^+))^+$ and 
$v^{\rm op}\in M_{XX}(C(\mathbb P \cl U_X^+))^{{\rm op} +}$ then 
$$\frak{m}(u \otimes v^{\rm op})\in M_{XX}(C(\mathbb P \cl U_X^+))^+$$
(this can be seen by realising $M_{XX}(C(\mathbb P \cl U_X^+))$ and 
$M_{XX}(C(\mathbb P \cl U_X^+))^{\rm op}$ as mutually commuting 
C*-algebras acting on the same Hilbert space).
We have that $W\in M_{XX}(C(\mathbb P \cl U_X^+))^+$ and, by 
Lemma \ref{l_opos}, that $W^{\rm op}\in M_{XX}(C(\mathbb P \cl U_X^+))^{{\rm op}+}$.
It follows that 
$$\tilde{W}\in M_{XXXX}(C(\mathbb P \cl U_X^+))^+.$$
Taking partial trace against the positive matrix $P\otimes Q^{\perp}$ yields a positive operator; the claim is now proved after noticing that the latter operator coincides with $g$. 

Similarly, 
$$h := \gamma_{P^{\perp},Q}(W\otimes W^{\rm op})\in C(\mathbb P \cl U_X^+)^+.$$
We have that 
$$\tau(g) = \tau(h) = 0;$$
by a straightforward application of the Cauchy-Schwartz inequality, $\tau$ annihilated $\cl I_{P,Q}$ and hence induces a trace (denoted in the same way) $\tau : \cl A_{P,Q}\to \bb{C}$. The validity of equation (\ref{eq_agains}) persists on $\cl A_{P,Q}$. 

Now consider the case where $\Gamma$ is a quantum correlation. 
By Theorem \ref{QNSbicorrelation}, there exists 
a trace $\tau : C(\mathbb P \cl U_X^+)\to\bb{C}$ that factors through a finite dimensional 
C*-algebra, such that $\Gamma = \Gamma_{\tau}$. 
By the previous paragraphs, $\tau$ annihilates $\cl J_{P,Q}$. 
Thus $\tau$ induces a trace (denoted in the same way) $\tau : \cl A_{P,Q}\to \bb{C}$ 
that factors through a finite dimensional 
C*-algebra and, as before, $\Gamma = \Gamma_{\tau}$.
The case where $\Gamma$ is of local type are similar. 

(ii)$\Rightarrow$(i) follows in a straightforward way from relation (\ref{eq_ST}). 
\end{proof}

\begin{remark}\label{r_eqST}
\rm
It follows from identity (\ref{eq_ST}) and the proof of Theorem \ref{iso} that
$\cl A_{P,Q}$ is be the universal C$^\ast$-algebra
generated by elements $u_{a,x}^*u_{a',x'}$, where $U = (u_{x,a})_{a,x}$ is a bi-unitary matrix, subject to the relations  \begin{align}
\label{mb1}(P\otimes I)U^{\rm t}_{1,3}U^*_{2,3}(Q^\perp\otimes I) = 0 \quad \& \quad (\bar{P}^\perp\otimes I)U^{\rm t}_{1,3}  U_{2,3}^*(\bar{Q}\otimes I) = 0.
\end{align}
\end{remark}

\begin{remark}\label{r_P=Q}
{\rm
Let us consider the special  case $P = Q$; 
this is the case of {\it quantum automorphisms $\cl U \to \cl U$}.  
We would like to interpret $\cl A_{P,P}$ as a quantum group of automorphisms of the quantum graph $\cl U \subseteq \mathbb C^X \otimes \mathbb C^X$.  This intuition can be made precise by equipping $\cl A_{P,P}$ with a natural co-associative comultiplication $\Delta_P:\cl A_{P,P} \to \cl A_{P,P} \otimes \cl A_{P,P}$, which turns it into a C$^\ast$-algebraic compact quantum group.  

To construct such a comultiplication $\Delta_P$ on $\cl A_{P,P}$, we first consider $C(\cl U_X^+)$, the universal C$^\ast $-algebra generated by the entries of 
a bi-unitary $U = (u_{x,a}) \in M_X(C(\cl U_X^+))$.  
The C*-algebra 
$C(\cl U_X^+)$ is well-known to be a compact matrix quantum group when equipped with the comultiplication 
$\Delta:C(\cl U_X^+) \to C(\cl U_X^+) \otimes C(\cl U_X^+)$, given by  
$\Delta(u_{x,a}) = \sum_{c\in X} u_{x,c} \otimes u_{c,a}$ on $C(\cl U_X^+)$ \cite{wang0}.  Define a new C$^\ast$-algebra $\cl B$ obtained from $C(\cl U_X^+)$  by quotienting by the relations given in \eqref{mb1}.  Denote the canonical matrix of generators of $\cl B$ by $V = (v_{x,a}) \in M_X(\cl B)$.  
(Note that, by definition, $V$ is the universal $X \times X$ bi-unitary satisfying the relations \eqref{mb1}.) 
We claim that the assignment $\Delta_\cl B(v_{x,a}) := \sum_c v_{x,c} \otimes v_{c,a}$, ($x,a \in X$),   determines  a co-associative co-multiplication $\Delta_\cl B: \cl B \to \cl B \otimes \cl B$, turning 
$(\cl B,\Delta_B)$ into a compact matrix quantum group.  To see this, it suffices to check that matrix $\tilde V \in M_X \otimes \cl B \otimes \cl B$, given by 
\[\tilde V = \left(\sum_{c\in X} v_{x,c} \otimes v_{c,a}\right)_{x,a \in X} =  V_{1,2}V_{1,3},\] 
satisfies the defining relations for $V$ (that is, 
$\tilde V$ is bi-unitary and satisfies the equations \eqref{mb1} 
in $M_X \otimes M_X \otimes \cl B \otimes \cl B$).  
Indeed, if the above is verified, then the co-multiplication $\Delta$ on $C(\cl U_X^+)$ will have been shown to factor the quotient $C(\cl U_X^+) \to \cl B$, 
proving that $\Delta_\cl B$ is well defined and induces a quantum group structure on $\cl B$.

First note that fact that $\tilde V$ is bi-unitary follows immediately from the formula for $\tilde V$ and the bi-unitarity of $V$.  To check \eqref{mb1}, we first note that in $M_{X} \otimes M_X \otimes (\cl B \otimes \cl B)$ we have
\[\tilde V_{1,3}^{\rm t}\tilde V_{2,3}^* = (V_{1,3}V_{1,4})^{\rm t}(V_{2,3}V_{2,4})^* = V_{1,4}^{\rm t}V_{1,3}^{\rm t}V_{2,4}^*V_{2,3}^* = V^{\rm t}_{1,4}V^*_{2,4}V^{\rm t}_{1,3}V^*_{2,3},
\]  
and hence
\begin{eqnarray*}
&&
(P \otimes I) \tilde V_{1,3}^{\rm t}\tilde V_{2,3}^* (P^\perp \otimes I)
=
(P\otimes I \otimes I )V^{\rm t}_{1,4}V^*_{2,4}V^{\rm t}_{1,3}V^*_{2,3}(P^\perp\otimes I \otimes I) \\
& = & 
(P\otimes I \otimes I )V^{\rm t}_{1,4}V^*_{2,4}(P \otimes I \otimes I) V^{\rm t}_{1,3}V^*_{2,3}(P^\perp\otimes I \otimes I)
= 0,
\end{eqnarray*}
where in the last line we have used relation \eqref{mb1} for $V$ to insert the extra copy of $(P \otimes I \otimes I)$ in the middle.  This shows that the first relation in \eqref{mb1} holds for $\tilde V$.  The second relation in \eqref{mb1}  is verified similarly.

Finally, we note that $\cl A_{P,P}$ is, by construction, the C$^\ast$-subalgebra of $\cl B$ 
generated by order two elements of $\cl B$ of the form $v_{x,a}^*v_{x',a'}$, $x,x',a,a' \in X$.  
The natural co-multiplication $\Delta_P$ on $\cl A_{P,P}$ is then the restriction of 
$\Delta_\cl B$ to $\cl A_{P,P}$ (note that  $\Delta_\cl B(\cl A_{P,P}) \subseteq \cl A_{P,P} \otimes \cl A_{P,P})$.  
}
\end{remark}

\begin{remark}
{\rm 
Note that, by Proposition \ref{p_loctype}, any character on $\cl A_{P,P}$ corresponds to a unitary $U \in \cl U_X$ such that $(U \otimes \bar U)\cl U = \cl U$.  In other words, the abelianisation of $\cl A_{P,P} $ corresponds via Gelfand duality 
to the classical compact group of unitary matrices \[G = \{U \otimes \bar U: U \in \cl U_X \ \mbox{and} \ (U \otimes \bar U)\cl U = \cl U\} \subseteq M_{X} \otimes M_X.\]  
The pair $(\cl A_{P,P}, \Delta_P)$ is therefore the 
quantisation of this very natural matrix group of automorphisms of $\cl U$.}
\end{remark}

%%%%%%%%%%%%%%%%%%%%%%%%%%%%%%%%%%%%%%%%%%%%%%%%%%%%
%%%%%%%%%%%%%%%%%%%%%%%%%%%%%%%%%%%%%%%%%%%%%%%%%%%

\section{Connection with algebraic quantum isomorphisms}
\label{s_connections}

The purpose of this section is to clarify the connection between the notion of a quantum graph isomorphism defined and characterised in Section \ref{s_qgig} 
and the notion, defined and studied in \cite{bcehpsw}. Our main reference for the latter concept will be \cite{daws}, and we follow its notation as closely as possible.

%%%%%%%%%%%%%%%%%%%%%%%%%%%%%%%%%%%%%%%%%%%%%%%%%%%%%%

\subsection{Algebraic isomorphism as a tighter equivalence}\label{ss_dir1}

We fix throughout the section a finite set $X$ and let $n = |X|$. 
We denote by ${\rm tr}$ the normalised trace on $M_X$; thus, ${\rm tr} = \frac{1}{|X|} {\rm Tr}$. 
In order to simplify the notation, we will write $1$ in the place of $I_X$. 

Denote by $L^2(M_X)$ the Hilbert space with underlying linear space $M_X$ and inner product arising from the GNS construction applied to the pair $(M_X,{\rm tr})$. 
More specifically, if
$\Lambda:M_X\to L^2(M_X)$ is the GNS map, we set 
$\langle\Lambda(a),\Lambda(b)\rangle = \tr(a^*b)$
(note that the inner product is linear in the second variable). 
In what follows, we view $M_X$ as a subalgebra of $\cl B(L^2(M_X))$, where an element $a\in M_X$ gives rise to 
the operator (denoted in the same way and given by)
$$a\Lambda(b)=\Lambda(ab), \ \ \ a,b\in M_X.$$
Note that $\Lambda(a)=a\Lambda(1)$, $a\in M_X$. 

Let $m : L^2(M_X)\otimes L^2(M_X)\to L^2(M_X)$ be the multiplication map, that is, the map, defined by letting 
$m(\Lambda(a)\otimes\Lambda(b))=\Lambda(ab)$, and
$m^*: L^2(M_X)\to L^2(M_X)\otimes L^2(M_X)$ be its Hilbert space adjoint.
For notational simplicity, we will often suppress the use of $\Lambda$, and consider $m$ (resp. $m^*$) as a map from $M_X\otimes M_X$ to $M_X$
(resp. from $M_X$ to $M_X\otimes M_X$).
We note that 
\begin{equation}\label{eq_ijsp}
m^*(\epsilon_{i,j}) = 
n\sum_{k=1}^n \epsilon_{i,k}\otimes \epsilon_{k,j}. 
\end{equation}
Indeed, for $p,q,s,t = 1,\dots,n$, we have 
\begin{equation}\label{eq_mstar1}
\langle m^*(\epsilon_{i,j}), \epsilon_{p,q}\otimes \epsilon_{s,t}\rangle 
= \langle \epsilon_{i,j},\epsilon_{p,q}\epsilon_{s,t}\rangle
= 
\tr(\epsilon_{j,i}\epsilon_{p,q}\epsilon_{s,t}),
\end{equation}
while 
\begin{eqnarray}\label{eq_mstar2}
\left\langle n\sum_{k=1}^n \epsilon_{i,k}\otimes \epsilon_{k,j}, \epsilon_{p,q}\otimes \epsilon_{s,t}\right\rangle
& = &  
n \sum_{k=1}^n 
\tr(\epsilon_{k,i}\epsilon_{p,q})
\tr(\epsilon_{j,k}\epsilon_{s,t})\\
& = & 
n\tr(\epsilon_{s,i}\epsilon_{p,q})
\tr(\epsilon_{j,t}).\nonumber
\end{eqnarray}
The right hand sides of (\ref{eq_mstar1}) and (\ref{eq_mstar2}) are thus equal, establishing 
(\ref{eq_ijsp}) which, further, implies that 
\begin{equation}\label{eq_mstar3}
m^*(1) = n\sum_{i,j=1}^n \epsilon_{i,j}\otimes \epsilon_{j,i}.
\end{equation}

Let $\eta : \mathbb C\to L^2(M_X)$ be the map, given by $\eta(\lambda) = \lambda\Lambda(1)$.
Recall \cite[Definition 2.4]{daws} that a self-adjoint 
linear map $A : L^2(M_X)\to L^2(M_X)$ is called a \emph{quantum adjacency matrix} if it has the following properties:
\begin{enumerate}
    \item $m(A\otimes A)m^* =  A$;
    \smallskip
    \item $({\rm id}\otimes\eta^*m)(1\otimes A\otimes 1)(m^*\eta\otimes {\rm id}) = A$;
    \smallskip
    \item $m(A\otimes 1)m^*= 0$. 
\end{enumerate}
We stress that condition (3) reflects the fact that we work with a quantum version of graphs 
without loops (graphs with loops are quantised in this context by requiring the condition 
$m(A\otimes 1)m^*= 1$ instead of (3) \cite[p. 6]{daws}).
A triple $G = (M_X,\tr,A)$, where $A$ is a quantum adjacency matrix, is called in \cite{bcehpsw, daws}
a quantum graph. 
In order to distinguish this notion from the one used in the present paper, we will hereafter refer to it as an \emph{algebraic quantum graph}.

We fix an algebraic quantum graph $G = (M_X, {\rm tr}, A)$.
We associate with $G$ the $M_X$-bimodule $S'$ 
in $\cl B(L^2(M_X))$ generated by $A$
(its dependence on $G$ is suppressed for notational simplicity); thus, recalling that the elements of $M_X$ are viewed as operators on $L^2(M_X)$, we have that 
\begin{equation}\label{eq_S'}
S' = {\rm span}\left\{aAb : a,b\in M_X\right\}.
\end{equation}
If $x,y\in M_X$, we write 
$\Theta_{\Lambda(x),\Lambda(y)}$ for the rank one operator, given by $$\Theta_{\Lambda(x),\Lambda(y)}(\xi) = \left\langle\Lambda(x),\xi\right\rangle \Lambda(y), \ \ \ 
\xi\in L^2(M_X).$$
Let $\Psi : \cl B(L^2(M_X))\to M_X\otimes M_X$ 
be the linear map, given by 
$$\Psi\left(\Theta_{\Lambda(x),\Lambda(y)}\right) = x^*\otimes y, \ \ \ x,y\in M_X;$$
by finite dimensionality, $\Psi$ is bijective. 
Set $e = (1\otimes A)(m^*(1))$; 
recalling (\ref{eq_mstar3}), we have that 
\begin{equation}\label{eq_idfore}
e = n\sum_{i,j=1}^n
\epsilon_{i,j}\otimes A(\epsilon_{j,i}).
\end{equation}

\begin{lemma}\label{l_PsiS'}
Let $G = (M_X,\tr,A)$ be an algebraic quantum graph. 
Then 
\begin{itemize}
\item[(i)] $\Psi(A)= e$, 
\item[(ii)] $e=e^*$, and 
\item[(iii)]
$\Psi(S') = {\rm span}\{(1\otimes a) e (b\otimes 1): a,b\in M_X\}\subseteq M_X\otimes M_X$.
\end{itemize}
\end{lemma}

\begin{proof}
(i)
Note that $\{\sqrt{n}\Lambda(\epsilon_{i,j})\}_{1 \le i,j \le n}$ is an orthonormal basis for $L^2(M_X)$; thus,
\[
A = \sum_{i,j=1}^n \Theta_{\sqrt{n} \Lambda(\epsilon_{i,j}), \sqrt{n}\Lambda( A (\epsilon_{i,j}))}, 
\]
and the claim now follows from (\ref{eq_idfore}). 

(ii) 
Let $R = \Theta_{\Lambda(a),\Lambda(b)}$, $a$, $b\in M_X$, and 
$T := ({\rm id}\otimes\eta^*m)(1\otimes R\otimes 1)(m^*\eta\otimes {\rm id})$. For notational simplicity write 
$m^*\eta(1) = m^*(1)=\sum_{i=1}^m\xi_i\otimes\eta_i$. If $x$, $y\in M_X$, then 
\begin{eqnarray*}
& & 
\langle\Lambda(x),T\Lambda(y)\rangle\\
&=&
\sum_{i=1}^m\langle\Lambda(x)\otimes\Lambda(1), (1\otimes m)(1\otimes\theta_{\Lambda(a),\Lambda(b)}\otimes 1)(\xi_i\otimes\eta_i\otimes\Lambda(y)\rangle\\
&=&
\sum_{i=1}^m\langle\Lambda(x)\otimes\Lambda(1),(1\otimes m)(\xi_i\otimes\langle\Lambda(a),\eta_i\rangle\Lambda(b)\otimes\Lambda(y)\rangle\\
&=&
\sum_{i=1}^m\langle\Lambda(x)\otimes\Lambda(1),\xi_i\otimes\Lambda(by)\rangle\langle\Lambda(a),\eta_i\rangle\\
&=&
\sum_{i=1}^m \langle\Lambda(x),\xi_i\rangle\langle\Lambda(a),\eta_i \rangle\langle\Lambda(1), \Lambda(by)\rangle\\
&=& 
\langle\Lambda(x)\otimes\Lambda(a),m^*(1)\rangle\langle\Lambda(b^*),\Lambda(y)\rangle
=
\langle\Lambda(xa),\Lambda(1)\rangle\langle\Lambda(b^*),\Lambda(y)\rangle\\
&=&
\langle\Lambda(x),\Lambda(a^*)\rangle\langle\Lambda(b^*),\Lambda(y)\rangle,
\end{eqnarray*}
showing that $T = \Theta_{\Lambda(b^*),\Lambda(a^*)}$, and 
hence that $\Psi(T) = b\otimes a^*=\frak{f}(\Psi(R))$, where $\frak{f}$ is the flip map.
By linearity, we obtain
$$\Psi({\rm id}\otimes\eta^*m)(1\otimes A\otimes 1)(m^*\eta\otimes {\rm id}))=\frak{f}(\Psi(A))$$
and therefore by (i) and condition (2), $e=\frak{f}(e)$.

Furthermore, 
$$A^* = \left(\sum_{i,j=1}^n \Theta_{\sqrt{n}\Lambda(\epsilon_{i,j}),\sqrt{n}\Lambda(A(e_{i,j}))}\right)^*
= \sum_{i,j=1}^n \Theta_{\sqrt{n}\Lambda(A(e_{i,j})), \sqrt{n}\Lambda(\epsilon_{i,j})},$$ 
and therefore
$$e=\Psi(A)=\Psi(A^*)=n\sum_{i,j=1}^n A(\epsilon_{i,j})^*\otimes \epsilon_{i,j}=\frak{f}(e^*),$$
giving $e=e^*=\frak{f}(e)$.

(iii)
The claim follows from the fact that 
\begin{equation}\label{eq_LL}
\Psi\left(a \Theta_{\Lambda(x), \Lambda (y)}b\right) = \Psi\left(\Theta_{\Lambda(b^*x), \Lambda(ay)}\right) = x^*b \otimes ay, \ \ \ a,b,x,y \in M_X. 
\end{equation}
\end{proof}

We set $\Lambda^{\otimes 2} = \Lambda\otimes\Lambda$ and 
write $\mathcal U_G = \Lambda^{\otimes 2}(\Psi(S'))$; thus,  $\mathcal U_G\subseteq L^2(M_X)\otimes L^2(M_X)$
(we note that, in the case $G$ is classical, the space $\mathcal U_G$ is closely related to, 
although not identical, to the space 
denoted in the same way in Section \ref{s_qgig}). 
Throughout this section, we fix an orthonormal basis $\{\Lambda(f_j)\}_{j=1}^{n^2}$ of $L^2(M_X)$; 
we note that
$\{\Lambda(f_i^*)\}_{i=1}^{n^2}$ is also an orthonormal basis. 
Let $\partial : L^2(M_X)\to L^2(M_X)$ be the linear operator with  
\begin{equation}\label{eq_fjst}
\partial\left(\Lambda(f_j^*)\right) = \Lambda(f_j), \ \ \ j = 1,\dots,n^2, 
\end{equation}
and set 
$\tilde{\mathcal U}_G = (\partial \otimes 1)(\mathcal U_G)$.
We next record the properties of the spaces of the form $\tilde{\cl U}_G$, akin to 
the properties of quantum graphs in the sense of 
Definition \ref{d_ss}. 

We write $\frak{d}$ for the conjugate-linear map on $L^2(M_X)\otimes L^2(M_X)$, given by 
$$
\frak{d}\left(\sum_{i,j=1}^{n^2}\alpha_{i,j}\Lambda(f_i)\otimes\Lambda(f_j)\right)  = \sum_{i,j=1}^{n^2} 
\bar\alpha_{i,j}\Lambda(f_i)\otimes\Lambda(f_j)$$ 
and recall that 
$\frak{f}$ is the flip map on $L^2(M_X)\otimes L^2(M_X)$.
We note that the definitions of the maps $\partial$ and $\frak{d}$ 
depend on the basis, but the 
concrete basis we are working with will be fixed or clear from the context.
The same comment applies for the notion we define next.

\begin{definition}\label{d_qpseufogr}
A subspace $\cl W\subseteq L^2(M_X)\otimes L^2(M_X)$ is called a \emph{quantum pseudo-graph} if
$\cl W$ is skew and $(\frak{d}\circ \frak{f})(\cl W) = \cl W$.
\end{definition}

Let $J_0: L^2(M_X)\otimes L^2(M_X)\to L^2(M_X)\otimes L^2(M_X)$ 
be the anti-linear map given by 
$J_0(\Lambda(x)\otimes\Lambda(y))=\Lambda(y^*)\otimes\Lambda(x^*)$.

\begin{lemma}\label{pseudograph}
Let $\cl W\subseteq L^2(M_X)\otimes L^2(M_X)$ and $\cl U=(\partial^{-1}\otimes 1)(\cl W)$. The following are equivalent: 
\begin{enumerate}
\item [(i)]$(\frak{d}\circ \frak{f})(\cl W) = \cl W$; 
\item [(ii)]$J_0(\cl U)=\cl U$; 
\item [(iii)]$\Psi^{-1}(\cl U)$ is self-adjoint.
\end{enumerate}
\end{lemma}
\begin{proof}
Let $x$, $y\in M_X$. Then
\begin{eqnarray*}
&&
(J_0\circ(\partial^{-1}\otimes 1))(\Lambda(x)\otimes\Lambda(y))\\
&&=
(J_0\circ (\partial^{-1}\otimes 1))\left(\sum_{i=1}^{n^2} \langle\Lambda(f_i),\Lambda(x)\rangle\Lambda(f_i)\otimes\Lambda(y)\right)\\
&&=
J_0\hspace{-0.1cm} 
\left(\hspace{-0.08cm}\sum_{i=1}^{n^2} \langle\Lambda(f_i),\Lambda(x)\rangle\Lambda(f_i^*)
\hspace{-0.05cm} \otimes\hspace{-0.05cm} \Lambda(y)
\hspace{-0.15cm}\right)
\hspace{-0.1cm} = \hspace{-0.05cm} \sum_{i=1}^{n^2}\overline{\langle\Lambda(f_i),\Lambda(x)\rangle}\Lambda(y^*)\hspace{-0.05cm} \otimes\hspace{-0.05cm} \Lambda(f_i)\\
&&=
\sum_{i,j=1}^{n^2} \overline{\langle\Lambda(f_i),\Lambda(x)\rangle}\langle\Lambda(f_j^*),\Lambda(y^*)\rangle\Lambda(f_j^*)\otimes\Lambda(f_i).
\end{eqnarray*}
Therefore,
\begin{eqnarray*}
&&
((\partial\otimes 1)\circ J_0\circ (\partial^{-1}\otimes 1))(\Lambda(x)\otimes\Lambda(y))\\
&&=
\sum_{i,j=1}^{n^2} 
\overline{\langle\Lambda(f_i),\Lambda(x)\rangle}\overline{\langle\Lambda(f_j),\Lambda(y)\rangle}\Lambda(f_j)\otimes\Lambda(f_i)\\
&&=
(\frak{d}\circ\frak{f})\left(\sum_{i,j=1}^{n^2} \langle\Lambda(f_i),\Lambda(x)\rangle\langle\Lambda(f_j),\Lambda(y)\rangle\Lambda(f_i)\otimes\Lambda(f_j)\right)\\
&&=
(\frak{d}\circ\frak{f})(\Lambda(x)\otimes\Lambda(y)),
\end{eqnarray*}
giving the equivalence (i)$\Leftrightarrow$(ii). 
As 
\begin{eqnarray*}
\Psi(\Theta_{\Lambda(x),\Lambda(y)}^*)
& = &
\Psi(\Theta_{\Lambda(y),\Lambda(x)}) = \Lambda(y^*)\otimes\Lambda(x)\\
&&=
J_0(\Lambda(x^*)\otimes\Lambda(y)) = J_0(\Psi(\Theta_{\Lambda(x),\Lambda(y)}),
\end{eqnarray*}
we obtain the equivalence (ii)$\Leftrightarrow$(iii). 
\end{proof}
\begin{proposition}\label{p_flipnew}
Let $G = (M_X,\tr,A)$ be an algebraic quantum graph. 
Then $\tilde{\mathcal U}_G$ is a quantum pseudo-graph. 
\end{proposition}

\begin{proof}
As $A$ is self-adjoint, there exist $x_i\in M_X$ and $\lambda_i\in\mathbb R$ such that $A=\sum_{i=1}^{n^2} \lambda_i\Theta_{\Lambda(x_i),\Lambda(x_i)}$. 
Using (\ref{eq_LL}), we have 
\begin{eqnarray*}
&&(\partial\otimes 1)((\Lambda^{\otimes 2}\circ\Psi)(aAb))
= \sum_{i=1}^{n^2}\lambda_i(\partial\otimes 1)(\Lambda(x_i^*b)\otimes\Lambda(ax_i))\\
&&
=\sum_{i=1}^{n^2}\lambda_i(\partial\otimes 1)
\left(\sum_{j,k=1}^{n^2} \langle\Lambda(f_k^*),\Lambda(x_i^*b)\rangle\langle \Lambda(f_j),\Lambda(ax_i)\rangle \Lambda(f_k^*)\otimes\Lambda(f_j)\right)\\
&&
=\sum_{i=1}^{n^2}\lambda_i\sum_{j,k=1}^{n^2} \overline{\langle\Lambda(f_k),\Lambda(b^*x_i)\rangle}\langle \Lambda(f_j),\Lambda(ax_i)\rangle \Lambda(f_k)\otimes\Lambda(f_j).
\end{eqnarray*}
Hence 
$$(\frak{d}\circ\frak{f}) ((\partial\otimes 1)((\Lambda^{\otimes 2}\circ\Psi)(aAb))) = (\partial\otimes 1)
(((\Lambda^{\otimes 2}\circ\Psi)(b^*Aa^*)),$$
implying the condition $(\frak{d}\circ \frak{f})(\tilde{\mathcal U}_G) = \tilde{\mathcal U}_G$.

Using (\ref{eq_mstar3}), we have 
$$0 = 
m(A\otimes 1)m^*(\epsilon_{i,i})
= n\sum_{k=1}^{n} A((\epsilon_{i,k}))\epsilon_{k,i}
= n\sum_{j,k=1}^{n} \lambda_j \tr(x_j^* \epsilon_{i,k})x_j \epsilon_{k,i},$$
and hence, for all $y\in M_X$, we have 
\begin{eqnarray*}
&&
0 = n\sum_{i,k,j=1}^{n^2}\lambda_j
\tr(x_j^*\epsilon_{i,k})\tr(x_j\epsilon_{k,i}y)\\
&&
= n \sum_{i,k,j=1}^{n^2}\lambda_j\tr(x_j^*\epsilon_{i,k})\tr(\epsilon_{i,k}^* y x_j)
= n \sum_{i,k,j=1}^{n^2} \lambda_j 
\langle x_j,\epsilon_{i,k}\rangle \langle \epsilon_{i,k}, y x_j\rangle\\
&&
= \sum_{j=1}^{n^2} \lambda_j 
\langle x_j, y x_j\rangle
= \sum_{j=1}^{n^2}\lambda_j\tr(yx_jx_j^*)
= \tr\left(y\sum_{j=1}^{n^2} \lambda_jx_jx_j^*\right).
\end{eqnarray*}
Therefore, $\sum_{j=1}^{n^2} \lambda_jx_jx_j^*=0$.
By the previous paragraph, we have 
\begin{eqnarray*}
&&
\hspace{-0.7cm} 
\left\langle (\partial\otimes 1)(((\Lambda^{\otimes 2}\circ\Psi)(aAb)), \sum_{k=1}^{n^2} \Lambda(f_k)\otimes\Lambda(f_k)\right\rangle\\
&&
\hspace{0.7cm} = \hspace{0.05cm}\sum_{i=1}^{n^2}\lambda_i\sum_{k=1}^{n^2} \overline{\langle\Lambda(f_k),\Lambda(b^*x_i)\rangle}\langle \Lambda(f_k),\Lambda(ax_i)\rangle \\
&&
\hspace{0.7cm} =\hspace{0.05cm}
\sum_{i=1}^{n^2}\lambda_i\langle\Lambda(b^*x_i),\Lambda(ax_i)\rangle
=\tr\left(\sum_{i=1}^{n^2}\lambda_iax_ix_i^*b\right)
= 0,
\end{eqnarray*}
showing that $\tilde{\mathcal U}_G$ is skew. 
\end{proof}

\begin{remark}\label{r_diffsetup}
\rm 
Proposition \ref{p_flipnew} shows that 
an algebraic quantum graph $G = (M_X,\tr, A)$ gives rise to a 
canonical quantum pseudo-graph $\tilde{\cl U}_G\subseteq L^2(M_X)\otimes L^2(M_X)$. 
The reason we are led to work with quantum pseudo-graphs instead of quantum graphs
in the sense of our Definition \ref{d_ss} lies in the
setup of QNS correlations, which is borrowed from \cite{dw}.
In defining QNS correlations, 
instead of no-signalling quantum channels $\Gamma : M_X\otimes M_Y\to M_A\otimes M_B$, 
one could start with no-signalling quantum channels 
$\Gamma' : M_X\otimes M_Y^{\rm op}\to M_A\otimes M_B^{\rm op}$. For the class of 
quantum commuting no-signalling correlations, this would lead to Choi matrices of the form 
$(\tau(e_{x,x',a,a'}e_{y,y',b,b'}))$, as opposed to 
the matrices $(\tau(e_{x,x',a',a}e_{y',y,b',b}))$ that arise through the current setup. 
As we will shortly see, 
in order to obtain a neat connection between the 
two types of quantum isomorphisms, one also needs to 
work with a slightly different concept of quantum isomorphism than the one 
employed in Section \ref{s_qgig}. 
We make this discussion rigorous in Theorem \ref{th_qcalgqc}. 
\end{remark}

Let $G_r = (M_X,\tr, A_r)$ be an algebraic quantum graph, $r = 1,2$.
Let $\cl O(G_1,G_2)$ be the universal (unital) C$^*$-algebra 
with generators $p_{i,j}$, $i,j = 1,\dots,n^2$, and relations that turn the map 
$\rho : M_X\to M_X\otimes \cl O(G_1,G_2)$,
given by 
\begin{equation}\label{eq_defrgo2}
\rho(f_i) = \sum_{j=1}^{n^2} f_j\otimes p_{j,i}, \ \ \ i = 1,\dots,n^2,  
\end{equation}
into a unital $*$-homomorphism such that 
\begin{equation}\label{eq_rhoint}
    (A_2\otimes {\rm id})\circ \rho = \rho\circ A_1,
\end{equation} and 
\begin{equation}\label{eq_rhotp}
    (\tr\otimes {\rm id})\circ \rho = \tr(\cdot) 1
\end{equation}

\begin{remark}\label{l_intertwi}
\rm
It follows from the proof of \cite[Theorem 4.7]{dr-vv} that the matrix
$P = (p_{i,j})_{i,j=1}^{n^2}\in M_{n^2}(\cl O(G_1,G_2))$
is automatically unitary.
Identifying $A_i$ with its corresponding matrix in $M_{n^2}$ with respect to the basis $\{f_j\}_{j=1}^{n^2}$, 
one can further check that equation \eqref{eq_rhoint} is equivalent to
\begin{equation}\label{eq_rhoint2}
(A_2\otimes 1_{\cl O(G_1,G_2)})P = P(A_1\otimes 1_{\cl O(G_1,G_2)}).
\end{equation}
Indeed, we have that 
\[
(\rho\circ A_1)(f_i) = \sum_{k,j=1}^{n^2} (A_1)_{k,i}f_j \otimes p_{j,k} = \sum_{j=1}^{n^2} f_j \otimes (P(A_1 \otimes I_H))_{j,i}
\]
and 
\begin{eqnarray*}
(A_2 \otimes 1_{\cl O(G_1,G_2)})\rho(f_i) 
\hspace{-0.3cm} & = & 
\hspace{-0.3cm} \sum_{k=1}^{n^2} A_2(f_k) \otimes p_{k,i}\\
\hspace{-0.3cm} & = & 
\hspace{-0.3cm} \sum_{k,j=1}^{n^2} (A_2)_{j,k}f_j \otimes p_{k,i} 
\hspace{-0.05cm} = \hspace{-0.05cm} \sum_{j=1}^{n^2} f_j \otimes ((A_2 \otimes I_H)P)_{j,i}.
\end{eqnarray*}
Identity (\ref{eq_rhoint2}) now follows by comparing the corresponding coefficients.
We note that reversing these arguments shows that relations (\ref{eq_rhoint2}) and (\ref{eq_rhoint}) are equivalent.  

Note that if $\{\Lambda(g_j)\}_{j=1}^{n^2}\subset L^2(M_X)$ is another orthonormal basis and $U\in M_{n^2}$ is unitary such that $\Lambda(g_j)=\Lambda(f_j)$, $j=1,\ldots, n^2$, then 
$$\rho(g_i)=\sum_{j=1}^{n^2} g_j\otimes ((U^*\otimes 1)P(U\otimes 1))_{j,i}.$$
\end{remark}

For the remainder of this section, we make the 
underlying assumption that the C$^*$-algebra $\cl O(G_1,G_2)$ is non-trivial.

\begin{proposition}
Let $G_r = (M_X,\tr, A_r)$ be an 
algebraic quantum graph, $r= 1,2$. 
Then the matrix $P\in M_{n^2}\otimes \cl O(G_1,G_2)$ is bi-unitary.
\end{proposition}

\begin{proof}
We verify that $P^{\rm t} = (p_{j,i})_{i,j=1}^{n^2}$ is unitary. By the previous remark we may assume that $\{\Lambda(f_i)\}_{i=1}^{n^2}$ is $\{\sqrt{n}\Lambda(\epsilon_{i,j})\}_{i,j=1}^{n}$.
Following the proof of \cite[Lemma 9.4]{daws}, 
let $W \in M_{n^2}$ be the matrix with entries
$$W_{(i,j), (k,l)} := 
n\left\langle \Lambda(\epsilon_{i,j}),\Lambda(\epsilon_{k,l}^*)\right\rangle
= \Tr(\epsilon_{j,i}\epsilon_{l,k})=\delta_{i,l}\delta_{j,k}.$$ 
Then 
$$(W^*W)_{(i,j),(k,l)}=\sum_{p,q=1}^{n} W_{(p,q),(i,j)}W_{(p,q),(k,l)} = 1$$ 
if $(i,j)=(k,l)$ and zero otherwise; thus, $W^*W = I_{n^2}$. 
As $\rho$ is $*$-preserving, 
we obtain
\begin{eqnarray*}
& & 
\sum_{j,l=1}^{n^2} f_l\otimes \left\langle\Lambda(f_j),\Lambda(f_i^*)\right\rangle p_{l,j}
= 
\sum_{j=1}^{n^2}
\left\langle\Lambda(f_j),\Lambda(f_i^*)\right\rangle\rho(f_j)\\
& = &
\rho(f_i^*) = \rho(f_i)^*
= \sum_{j=1}^{n^2} f_j^*\otimes p_{j,i}^*
= 
\sum_{j,l=1}^{n^2} f_l\otimes
\left\langle \Lambda(f_l),\Lambda(f_j^*)\right\rangle p_{j,i}^*.
\end{eqnarray*}
Thus
\begin{eqnarray*}
\sum_{j=1}^{n^2} p_{l,j}W_{j,i} 
& = & 
n \sum_{j=1}^{n^2}
\left\langle\Lambda(f_j),\Lambda(f_i^*)\right\rangle p_{l,j}\\
& = & 
n \sum_{j=1}^{n^2}\left\langle\Lambda(f_l),\Lambda(f_j^*)\right\rangle p_{j,i}^*
= \sum_{j=1}^{n^2}
W_{l,j}p_{j,i}^*
\end{eqnarray*}
for all $i,l = 1,\dots,n^2$; 
equivalently, $P(W\otimes 1) = (W\otimes 1)P^{{\rm t}*}$. 
It follows that 
$P^{{\rm t}*} = (W^{-1}\otimes 1)P(W\otimes 1)$, and hence 
\begin{eqnarray*}
P^{{\rm t}*}P^{\rm t} 
& = & 
(W^{-1}\hspace{-0.05cm} \otimes \hspace{-0.05cm} 1)P(W\hspace{-0.05cm} \otimes \hspace{-0.05cm} 1)(W^*\hspace{-0.05cm} \otimes \hspace{-0.05cm} 1)P^*((W^*)^{-1}\hspace{-0.05cm} \otimes \hspace{-0.05cm} 1)\\
& = & 
(W^{-1}\hspace{-0.05cm} \otimes \hspace{-0.05cm} 1)PP^*((W^*)^{-1}\hspace{-0.05cm} \otimes \hspace{-0.05cm} 1) = 1
\end{eqnarray*}
and 
\begin{eqnarray*}
P^{\rm t} P^{{\rm t}*} 
& = & 
(W^*\hspace{-0.05cm} \otimes \hspace{-0.05cm} 1)P^*((W^*)^{-1}\hspace{-0.05cm} \otimes\hspace{-0.05cm} 1)(W^{-1} \hspace{-0.05cm} \otimes \hspace{-0.05cm} 1)P(W \hspace{-0.05cm} \otimes \hspace{-0.05cm} 1)\\
& = & 
(W^* \hspace{-0.05cm} \otimes \hspace{-0.05cm} 1)P^*P(W \hspace{-0.05cm} \otimes \hspace{-0.05cm} 1) = 1.
\end{eqnarray*}

\end{proof}

Let $G_r = (M_X,\tr, A_r)$ be an algebraic quantum graph, $r = 1,2$. 
We will write $S_r'$ for the space corresponding to $G_r$ via (\ref{eq_S'}), $r = 1,2$. 
We say 
\cite[Definition 4.4]{bcehpsw} 
that 
$G_1$ and $G_2$ are \emph{quantum commuting isomorphic},
denoted $G_1\simeq_{\rm qc} G_2$, 
if the C$^*$-algebra
$\cl O(G_1,G_2)$ admits a tracial state, say $\tau$.
We assume, unless specified otherwise, that 
$G_1\simeq_{\rm qc} G_2$. 
Let $H$ be the Hilbert space, arising from the GNS construction applied to $\tau$ 
and, by abuse of notation, continue to write 
$p_{i,j}$ for the image of the corresponding canonical generator of $\cl O(G_1,G_2)$ under the 
$\ast$-representation arising from $\tau$. 
By (\ref{eq_rhoint2}), we have 
\begin{equation}\label{eq_A1intoA2}
A_2\otimes I_H = P(A_1\otimes I_H)P^*. 
\end{equation}
We view $P = (p_{i,j})_{i,j=1}^{n^2}$ as an operator on 
$L^2(M_X)\otimes H$ and note that, by (\ref{eq_defrgo2}), we have
\begin{equation}\label{eq_PL2}
P(\Lambda(b)\otimes\xi)=\rho(b)(\Lambda(1)\otimes\xi), \ \ \ b\in M_X, \xi\in H.
\end{equation}
Moreover, for $a,d\in M_X$ and $\xi\in H$ we have
 \begin{eqnarray*}
&&P(a\otimes 1)P^*P(\Lambda(d)\otimes\xi)=P(a\otimes 1)(\Lambda(d)\otimes\xi)=P(\Lambda(ad)\otimes\xi)\\&&=\rho(ad)(\Lambda(1)\otimes\xi)=\rho(a)(\rho(d)(\Lambda(1)\otimes\xi)=\rho(a)P(\Lambda(d)\otimes\xi), 
\end{eqnarray*}
and hence 
\begin{equation}\label{p_rho} 
P(a\otimes 1)P=\rho(a), a\in M_X,
\end{equation}
as maps on $L^2(M_X)\otimes\cl B(H)$.

We define $\tilde P\in \cl B(L^2(M_X)\otimes H)$ by letting 
$$\tilde P(\Lambda(f_j^*)\otimes\eta)= \sum_{k=1}^{n^2} 
\Lambda(f_k^*)\otimes p_{k,j}^*\eta, 
\ \ \ \eta \in H.$$
Using leg-notation, we write $P_{2,3}$ and $P_{1,3}$ for the 
 corresponding operators on $L^2(M_X)\otimes L^2(M_X)\otimes H$, arising from $P$.

\begin{lemma}\label{l_UiSi}
We have
$((\Lambda^{\otimes 2}\circ\Psi)\otimes {\rm id})(P(S_1'\otimes 1)P^*) = P_{2,3}\tilde P_{1,3}(\mathcal U_{G_1}\otimes 1).$
\end{lemma}

\begin{proof}
Let $x,y\in M_X$. We have 
\begin{eqnarray*}
& & 
\hspace{-0.4cm} 
P\left(\Theta_{\Lambda(x),\Lambda(y)}\otimes 1\right)P^*(\Lambda(f_k)\otimes\eta)\\
& &  
\hspace{-0.9cm} = P\left(\Theta_{\Lambda(x),\Lambda(y)}\otimes 1\right)
\left(\sum_{j=1}^{n^2} \Lambda(f_j)\otimes p_{k,j}^*\eta\right)\\
& & 
\hspace{-0.9cm} =
P\left(\sum_{j=1}^{n^2}
\langle\Lambda(x),\Lambda(f_j)\rangle\Lambda(y)\otimes p_{k,j}^*\eta\right)\\
& & 
\hspace{-0.9cm} = 
P\left(\sum_{i,j=1}^{n^2}
\langle\Lambda(x),\Lambda(f_j)\rangle\langle\Lambda(f_i),\Lambda(y)\rangle\Lambda(f_i)\otimes p_{k,j}^*\eta\right)\\
& &
\hspace{-0.9cm} = 
\hspace{-0.1cm} \sum_{m=1}^{n^2}\Lambda(f_m)\otimes \left(\sum_{i=1}^{n^2}
\langle\Lambda(f_i),\Lambda(y)\rangle p_{m,i} \sum_{j=1}^{n^2}
\langle\Lambda(x),\Lambda(f_j)\rangle p_{k,j}^*\right)(\eta)\\
& &
\hspace{-0.9cm} = 
\hspace{-0.25cm} \sum_{l,m=1}^{n^2}
\hspace{-0.1cm} \Theta_{\Lambda(f_l), \Lambda(f_m)}
\hspace{-0.1cm}\otimes\hspace{-0.15cm}
\left(\hspace{-0.05cm} \sum_{i=1}^{n^2}
\langle\Lambda(f_i),\hspace{-0.05cm} \Lambda(y)\rangle p_{m,i} \hspace{-0.1cm} \sum_{j=1}^{n^2}
\langle\Lambda(x),\hspace{-0.05cm} \Lambda(f_j)\rangle p_{l,j}^*\hspace{-0.1cm} \right)
\hspace{-0.15cm}(\hspace{-0.04cm}\Lambda(f_k)\hspace{-0.05cm}\otimes\hspace{-0.05cm}\eta)
\end{eqnarray*}
and hence
\begin{eqnarray*}
& & 
\hspace{-0.9cm} P(\Theta_{\Lambda(x),\Lambda(y)}\otimes 1)P^*\\
& & 
\hspace{-0.5cm} 
= \hspace{-0.2cm}\sum_{l,m=1}^{n^2}\Theta_{\Lambda(f_l),\Lambda(f_m)}\otimes \left(\sum_{i=1}^{n^2}
\langle\Lambda(f_i),\Lambda(y)\rangle p_{m,i}\right)
\hspace{-0.2cm}\left(\sum_{j=1}^{n^2}\langle\Lambda(x),\Lambda(f_j)\rangle p_{l,j}^*\right)\hspace{-0.15cm}.
\end{eqnarray*}
It follows that 
\begin{eqnarray*}
&&
\hspace{-1cm}
((\Lambda^{\otimes 2}\circ\Psi)\otimes \text{id})
\left(P(\Theta_{\Lambda(x),\Lambda(y)}\otimes 1)P^*\right)\\
&& 
\hspace{-1cm} = \hspace{-0.2cm} \sum_{l,m=1}^{n^2}\Lambda(f_l^*)
\hspace{-0.05cm}\otimes\hspace{-0.05cm}\Lambda(f_m)
\hspace{-0.05cm}\otimes\hspace{-0.1cm} \left(\sum_{i=1}^{n^2}\langle\Lambda(f_i),\Lambda(y)\rangle p_{m,i}\right)
\hspace{-0.2cm}\left(\sum_{j=1}^{n^2}\langle\Lambda(x),\Lambda(f_j)\rangle p_{l,j}^*\right)\hspace{-0.15cm}.
\end{eqnarray*}
As $P(\Lambda(y)\otimes\xi)=\sum_{i,m=1}^{n^2}
\langle \Lambda(f_i),\Lambda(y)\rangle\Lambda(f_m)\otimes p_{m,i}\xi$ and
\begin{eqnarray*}
\tilde P(\Lambda(x^*)\otimes\eta)
& = & \sum_{j=1}^{n^2}\tilde P(\langle\Lambda(f_j^*),\Lambda(x^*)\rangle(\Lambda(f_j^*)\otimes\eta))\\
& = & 
\sum_{j,l=1}^{n^2}\langle\Lambda(x),\Lambda(f_j)\rangle\Lambda(f_l^*)\otimes p_{l,j}^*\eta,
\end{eqnarray*}
we obtain
$$((\Lambda^{\otimes 2}\circ\Psi)\otimes \text{id})(P(\Theta_{\Lambda(x),\Lambda(y)}\otimes 1)P^*)=P_{2,3}\tilde P_{1,3}(\Lambda(x^*)\otimes\Lambda(y)\otimes 1).$$
The statements now follow by linearity 
from the definition of $\cl U_{G_1}$.
\end{proof}

Let $\cl N\subseteq \cl B(H)$ be a von Neumann algebra, equipped with a faithful trace $\tilde{\tau}$, 
and let $U = (u_{i,j})_{i,j}\in M_{n^2}(\cl N)$ be a bi-unitary block operator matrix
(with entries in $\cl N$). 
Suppose that $\Gamma : M_{n^2}\otimes M_{n^2}\to M_{n^2}\otimes M_{n^2}$ is a QNS correlation, given by 
\begin{equation}\label{eq_twistGa}
\Gamma(\epsilon_{i,i'}\otimes \epsilon_{j,j'}) = (\tilde{\tau}(u_{k,i}^*u_{k',i'}u_{l',j'}^*u_{l,j}))_{k,k',l,l'}.
\end{equation}
We let $\tilde\Gamma : M_{n^2}\otimes M_{n^2}\to M_{n^2}\otimes M_{n^2}$ 
be the unital completely positive map, given by 
$$\tilde\Gamma(\epsilon_{k,k'}\otimes \epsilon_{l,l'}) = (\tilde{\tau}(u_{l,j}^*u_{k,i}u_{k',i'}^*u_{l',j'}))_{i,i',j,j'}.$$
If $f_{k,k',i,i'} = u_{k,i}u_{k',i'}^*$ then $\tilde\Gamma$ has Choi matrix
$(\tilde{\tau}(f_{k,k',i,i'}f_{l',l,j',j}))$ and is hence a quantum commuting QNS correlation. 
We remark that, as can be verified in a straightforward way, 
if $\sigma: M_{n^2}\otimes M_{n^2}\to M_{n^2}\otimes M_{n^2}$ is the map, given by 
$\sigma(\epsilon_{k,k'}\otimes\epsilon_{l,l'}) = \epsilon_{k',l}\otimes\epsilon_{l',k}$, 
then $\tilde\Gamma=\sigma\circ\Gamma^*\circ\sigma$.

We call two quantum pseudo-graphs $\cl W_1$ and $\cl W_2$ \emph{qc-pseudo-isomorphic} if
there exists $\Gamma\in \cl Q_{\rm qc}^{\rm bic}$ of the form described in the previous paragraph, 
such that 
\begin{itemize}
\item[(i)] 
$\Gamma$ is a perfect strategy for $\cl W_1\to \cl W_2$, and 
\item[(ii)] 
$\tilde{\Gamma}$ is a perfect strategy for $\cl W_2\to \cl W_1$.
\end{itemize}

\begin{theorem}\label{th_qcalgqc}
Let  $G_r = (M_X,\tr, A_r)$, $r = 1,2$, 
be algebraic quantum graphs with $G_1\simeq_{\rm qc} G_2$. Then the quantum pseudo-graphs 
$\tilde{\cl U}_{G_1}$ and $\tilde{\cl U}_{G_2}$ are qc-pseudo-isomorphic. 
\end{theorem}

\begin{proof}
Set $\tilde{\cl U}_r = \tilde{\cl U}_{G_r}$ for brevity, $r = 1,2$. 
By assumption, the C*-algebra $\cl O(G_1,G_2)$ has a tracial state, say $\tau$. 
Let $\cl N\subseteq\cl B(H)$ be the von Neumann algebra associated with $\tau$ via the GNS construction, and 
$\tilde{\tau}$ be the (faithful) trace on $\cl N$, corresponding to $\tau$.
Write $u_{i,j}$ for the images of the canonical generators $p_{i,j}$ under the Gelfand map, and let
$U$ and $\tilde{U}$ be the matrices, corresponding to $P$ and $\tilde{P}$, respectively. 
Let $\Gamma$ be the QNS correlation given by (\ref{eq_twistGa}). 
Note that, by (\ref{p_rho}) $P(a\otimes 1)P^* = \rho(a)$, $a\in M_{X}$, 
and, by 
(\ref{eq_A1intoA2}), $P(A_1\otimes 1)P^*=A_2\otimes 1$; thus, 
\begin{eqnarray*}
P(aA_1b\otimes 1)P^*
& = & 
P(a\otimes 1)(A_1\otimes 1)(b\otimes 1)P^*\\
& = & P(a\otimes 1)P^*P(A_1\otimes 1)P^*P(b\otimes 1)P^*\\
& = & 
\rho(a)(A_2\otimes 1)\rho(b).
\end{eqnarray*}
It now follows from (\ref{eq_defrgo2}) that
$U(S_1'\otimes 1)U^*\subseteq S_2'\otimes \cl B(H)$.
By Lemma \ref{l_UiSi}, 
\begin{equation}\label{eq_2313}
U_{2,3}\tilde U_{1,3}(\mathcal U_1\otimes 1)\subseteq \mathcal U_2\otimes \cl B(H).
\end{equation}
Recalling the map $\partial$ defined in (\ref{eq_fjst}), 
note that 
$$(\partial\hspace{-0.05cm} \otimes\hspace{-0.05cm} 1)
\tilde U
(\partial\hspace{-0.05cm}\otimes\hspace{-0.05cm} 1)(\Lambda(f_i)\hspace{-0.05cm}\otimes\hspace{-0.05cm} \eta)
\hspace{-0.05cm} = \hspace{-0.05cm} 
(\partial \hspace{-0.05cm}\otimes\hspace{-0.05cm} 1)
\hspace{-0.1cm}\left(\hspace{-0.02cm}\sum_{k=1}^{n^2}\hspace{-0.05cm}\Lambda(f_k^*)\hspace{-0.05cm}\otimes\hspace{-0.05cm} u_{k,i}^*\eta\hspace{-0.05cm}\right)
\hspace{-0.1cm} = \hspace{-0.05cm}
\sum_{k=1}^{n^2}\Lambda(f_k)
\hspace{-0.05cm}\otimes \hspace{-0.05cm} u_{k,i}^*\eta,$$ that is, 
$$(\partial \otimes 1)\tilde U(\partial \otimes 1) = U^{{\rm t}*} =: \bar{U}.$$
Let 
$P_{\tilde{\mathcal U}_r}$ be the projection onto $\tilde{\mathcal U}_r$, $r = 1,2$.
Then condition (\ref{eq_2313}) implies that 
\begin{equation}\label{eq_Utilde}
(P_{\tilde{\mathcal U}_2}^{\perp}\otimes 1)U_{2,3}\bar U_{1,3}(P_{\tilde{\mathcal  U}_1}\otimes 1) = 0.
\end{equation}
The arguments in the proof of Theorem \ref{iso} 
(see also the subsequent Remark)
now imply that 
$\Gamma$ is a perfect strategy for the quantum graph homomorphism game 
$\tilde{\cl U}_1\to \tilde{\cl U}_2$. 

Relation (\ref{eq_A1intoA2}) implies
$U^*(A_2\otimes 1)U = A_1\otimes 1$.
By symmetry, we obtain the relation
\begin{equation}\label{eq_P23bars}
(P_{\tilde{\mathcal U}_1}^{\perp}\otimes 1) U_{2,3}^*\bar U_{1,3}^*(P_{\tilde{\mathcal U}_2}\otimes 1) = 0.
\end{equation}
We now show that 
$\tilde\Gamma$ is a perfect strategy for $\tilde{\cl U}_2\to \tilde{\cl U}_1$. 
Let $\xi = \sum_{k,l=1}^{n^2}\alpha_{k,l} f_k\otimes f_l \in \tilde {\cl U}_2$ and 
$\eta = \sum_{i,j=1}^{n^2} \beta_{i,j} f_i\otimes f_j\in \tilde{\cl U}_1^\perp$; then 
\begin{eqnarray*}
\langle\tilde\Gamma(\xi\xi^*)\eta,\eta\rangle
= \tilde{\tau}\left(\sum_{k,l,k',l'=1}^{n^2} \hspace{0.1cm} \sum_{i,i',j,j'=1}^{n^2} 
\alpha_{k,l}\overline{\alpha_{k',l'}}
u_{j,l}^*u_{k,i}u_{k',i'}^*u_{l',j'}\beta_{i',j'}\overline{\beta_{i,j}}\right).
\end{eqnarray*}
Writing 
$R_{\xi,\eta} = \sum_{k,l,i,j=1}^{n^2} \overline{\alpha_{k,l}}\beta_{i,j}u_{k,i}^*u_{l,j}$, 
we have
$\langle\tilde\Gamma(\xi\xi^*)\eta,\eta\rangle = \tilde{\tau}(R_{\xi,\eta}^*R_{\xi,\eta}),$
and using the fact that $\tilde{\tau}$ is faithful, this implies that 
$$\langle\tilde\Gamma(\xi\xi^*)\eta,\eta\rangle
= 0 \ \Longleftrightarrow \ R_{\xi,\eta}=0.$$
Taking into account (\ref{eq_legno2}), we now have
$$\langle R_{\xi,\eta}f,h\rangle = \langle \bar U_{1,3} U_{2,3}(\eta\otimes f),\xi\otimes h\rangle, \ \ \ f,h\in H,$$ 
and therefore 
$$(P_{\tilde{\cl U}_2}\otimes 1)\bar U_{1,3}U_{2,3}(P_{\tilde{\cl U}_1}^\perp\otimes 1) = 0 \ \Longleftrightarrow \ \tilde{\Gamma} \mbox{ is a  perfect strategy for } \tilde{\cl U}_2\to \tilde{\cl U}_1.$$
The proof is complete in view of (\ref{eq_P23bars}). 
\end{proof}

\begin{remark}
\rm
For a classical graph $G$ with vertex set $X$, let $A_G : M_X\to M_X$ be Schur multiplication map 
against the adjacency matrix of $G$. Then $(M_X,{\rm tr},A_G)$ is an algebraic quantum graph. 
Let
$$\cl W_G 
= \text{span}\{\epsilon_{x,x}\otimes \epsilon_{y,y}: x\sim y \mbox{ in } G\} \subseteq M_X\otimes M_X.$$  
Then $\cl W_G$ is a quantum pseudo-graph in $L^2(M_X)\otimes L^2(M_X)$. 

Let $G_1$, $G_2$ be classical graphs with vertex set $X$.
We have the following three types of quantum commuting isomorphism for the graphs $G_1$ and $G_2$:
\begin{itemize}
\item[(a)] quantum commuting isomorphism in the sense of classical non-local games \cite{amrssv};
\item[(b)] quantum commuting isomorphism of the algebraic quantum graphs $(M_X,{\rm tr},A_{G_1})$ and 
$(M_X,{\rm tr},A_{G_2})$; 
\item[(c)] quantum commuting isomorphism in the sense of quantum non-local games (Section \ref{s_qgig}), 
employing the quantum pseudo-graphs $\cl W_1$ and $\cl W_2$.
\end{itemize}
We have that (a) implies (b), and that (b) implies (c). We do not know if these implications are reversible. 
\end{remark}

%%%%%%%%%%%%%%%%%%%%%%%%%%%%%%%%%%%%%%%%%%%%%%%%%%%%

\subsection{A partial converse}\label{ss_dir2}

In the remainder of this section, we discuss to what extent the 
implication established in Theorem \ref{th_qcalgqc} can be reversed.
We first note that the quantum pseudo-graphs of the form $\tilde{\cl U} = \tilde{\cl U}_G$, 
for an algebraic quantum graph $G = (M_X,\tr,A)$, automatically have some extra structure, and hence 
a full reversal of 
Theorem \ref{th_qcalgqc} cannot be expected.  Indeed, let 
$\mathcal U = (\partial^{-1}\otimes 1)(\tilde{\mathcal U})$, and recall that
$S' = \Psi^{-1}(\mathcal U)\subseteq \cl B(L^2(M_X))$ is an $M_X$-bimodule. 
We first show that any quantum pseudo-graph 
$\tilde{\mathcal U}$, for which $\Psi^{-1}(\mathcal U)$ is a $M_X$-bimodule, arises in this way.
In what follows we fix a basis $\{\Lambda(f_i)\}_i$ in $L^2(M_X)$ when define pseudo-graphs $\tilde {\cl U}$.

Let $M_X^{\rm op}$ be the opposite algebra to $M_X$. 
For notational simplicity, we will 
consider $M_X^{\rm op}$ as having the same underlying 
vector space as $M_X$, and will denote its product by $\cdot_{\rm op}$; thus, 
$a\cdot_{\rm op}b = ba$, $a,b\in M_X$. 
Let $(L^2(M_X^{\rm op}), \Lambda^{\rm op})$ be the GNS construction applied to $(M_X^{\rm op},\tr)$. As
$$\langle\Lambda^{\rm op}(a),\Lambda^{\rm op}(b)\rangle=\tr(a^*\cdot_{\rm op}b)=\tr(ba^*)=\langle\Lambda(a),\Lambda(b)\rangle,$$
we have that $L^2(M_X^{\rm op})\otimes L^2(M_X)$ and $L^2(M_X)\otimes L^2(M_X)$ can be identified also as Hilbert spaces.  
Recall that $L^2(M_X)^{\rm d}$ is the Banach space dual of 
$L^2(M_X)$ (equivalently, the conjugate Hilbert space to $L^2(M_X)$).  
If $\cl A\subseteq \cl B(L^2(M_X))$ is a $*$-subalgebra, then the map
$T^{\rm op}\mapsto T^{\rm d}$, where $T^{\rm d}\overline{\xi}=\overline{T^*\xi}$, $\xi\in L^2(M_X)$, is  a 
$*$-isomorphism. In what follows we will often identify $T^{\rm op}$ with $T^{\rm d}$. 
For a linear operator $T : L^2(M_X)\to L^2(M_X)$, 
we define  $\bar T:L^2(M_X)^{\rm d}\to L^2(M_X)^{\rm d}$, by letting
$\bar T\bar\xi = \overline{T\xi}$, $\xi\in L^2(M_X)$.

\begin{lemma}\label{idempotent}
Let $G=(M_X,\tr,A)$ be an algebraic quantum graph. Then $e=(1\otimes A)(m^*(1))$ is a projection when considered as element in $M_X^{\rm op}\otimes M_X$.
\end{lemma}
\begin{proof}
By Lemma \ref{l_PsiS'}, we only have to show that $e$ is an idempotent. Using (\ref{eq_idfore}), we have 
\begin{eqnarray*}
e^2
& = &
n^2\left(\sum_{i,j=1}^n\epsilon_{i,j}\otimes A(\epsilon_{j,i})\right)\left(\sum_{k,l=1}^n\epsilon_{k,l}\otimes A(\epsilon_{l,k})\right)\\
& = & 
n^2\sum_{i,j,k,l=1}^n \epsilon_{i,j}\cdot_{\rm op}\epsilon_{k,l}\otimes A(\epsilon_{j,i})A(\epsilon_{l,k})\\
\end{eqnarray*}
\begin{eqnarray*}
& = &
n^2\sum_{i,j,k=1}^n \epsilon_{k,j}\otimes A(\epsilon_{j,i})A(\epsilon_{i,k})
= n \sum_{j,k=1}^n \epsilon_{k,j}\otimes m(A\otimes A)m^*(\epsilon_{j,k})\\
& = &
n\sum_{j,k=1}^n\epsilon_{k,j}\otimes A(\epsilon_{j,k}) = e.
\end{eqnarray*}
\end{proof}

\begin{remark}\label{proj}\rm
We remark that reversing the arguments of Lemmas \ref{l_PsiS'} and \ref{idempotent}, we can easily see that any projection $e\in M_X^{\rm op}\otimes M_X$, such that $e=\frak{f}(e)$, gives rise to self-adjoint  
operator $A : L^2(M_X)\to L^2(M_X)$ satisfying the conditions 
(1) and (2) of quantum adjacency matrix and linked to $e$ through the identity (\ref{eq_idfore}).
\end{remark}

Let $J : L^2(M_X)\to L^2(M_X)$ be the conjugate-linear map, given by $J(\Lambda(a)) = \Lambda(a^*)$, 
and the map $\kappa : \cl B(L^2(M_X))\to\cl B(L^2(M_X))$ be given by $\kappa(x) = Jx^*J$. 
We have that 
$\kappa$ is an anti-$*$-homomorphism such that $\kappa^2=\id$;
writing  $\pi: L^2(M_X)\to L^2(M_X)$ for the $*$-homomorphism given by  
$\pi(x)\Lambda(a) = \Lambda(xa)$, we have that $\kappa(\pi(M_X)) = \pi(M_X)'$.

\begin{proposition}
Let $\tilde{\cl U}$ be a quantum pseudo-graph such that $\Psi^{-1}((\partial^{-1}\otimes 1)(\tilde{\cl U}))$ is an $M_X$-bimodule. Then there exists an algebraic quantum graph $G=(M_X,\tr, A)$ such that $\tilde{\cl U}=\tilde{\cl U}_G$.
\end{proposition}
\begin{proof}
Let $\cl U=(\partial^{-1}\otimes 1)(\tilde{\cl U})$ and $S'=\Psi^{-1}(\cl U)$. By assumption, $S'$ is an $M_X$-bimodule and hence $\kappa(S')$ is a $\pi(M_X)'$-bimodule. Under the canonical bijection between $\cl B(L^2(M_X))$  and 
$L^2(M_X)^{\rm d} \otimes L^2(M_X)$, the $\pi(M_X)'$-bimodule $\kappa(S')$ corresponds to the
$(\pi(M_X)')^{\rm op}\otimes \pi(M_X)'$-invariant subspace $\cl U'$.  Thus it gives rise to the  
projection $e\in M_X^{\rm op}\otimes M_X$ onto $\cl U'$.   
By Lemma \ref{pseudograph}, $S'$ is self-adjoint and hence so is $\kappa(S')$, which implies, again by Lemma \ref{pseudograph}, that 
$e = \frak{f}(e)$ and $J_0(\cl U')=\cl U'$. 

Let $A : L^2(M_X)\to L^2(M_X)$ be the  linear map corresponding to $e$ as in Remark \ref{proj}. 
We have that $\kappa(S')$ is the $\pi(M_X)'$-bimodule generated by $A$. 
It follows that $S'$ is the $\pi(M_X)$-bimodule generated by $A$. 
In fact, since $\kappa(\pi(M_X)')=\pi(M_X)$, it suffices to verify that $JA^*J = A$. 
Write $A=\sum_{i=1}^m \lambda_i \Theta_{\Lambda(x_i),\Lambda(x_i)}$, $\lambda_i\in\mathbb R$, $x_i\in M_X$, 
$i = 1,\dots,m$. 
Then $e=\Psi(A)=\sum_{i=1}^m \lambda_i x_i^*\otimes x_i$. 
On the other hand,
$$JA^*J = JAJ = \sum_{i=1}^m \lambda_i\Theta_{\Lambda(x_i^*),\Lambda(x_i^*)}.$$
Thus 
$\Psi(JA^*J) = \sum_{i=1}^m \lambda_i x_i\otimes x_i^* = \frak{f}(e)$.
As $e=\frak{f}(e)$, we get $\Psi(JA^*J) = \Psi(A)$, implying that $JA^*J = A$.

Finally, reversing arguments in Proposition \ref{p_flipnew} we see that skewness of $\tilde{\cl  U}$ 
implies that $m(A\otimes 1)m^*=0$, showing that $A$ is a quantum adjacency matrix.
Letting $G = (M_X,A,{\rm tr})$, we have that $\tilde{\cl  U} = \tilde{\cl  U}_G$.  
\end{proof}

We now fix a quantum pseudo-graph $\tilde{\cl U}_r$ in $L^2(M_X)\otimes L^2(M_X)$, for which 
the corresponding space $S_r'$ is an $M_X$-bimodule, and let 
$\cl U_r := (\partial^{-1}\otimes 1)(\tilde{\cl U}_r)$, $r = 1,2$. 
We assume that $\tilde{\cl U}_1$ and $\tilde{\cl U}_2$ are qc-pseudo-isomorphic, and let 
$\mathcal N$ be a von Neumann algebra with trace $\tau$, and 
$U = (u_{k,i})_{k,i=1}^{n^2}$ be a bi-unitary, with $u_{k,i}\in \cl N$, $k,i=1,\dots,n^2$,  
such that $U$ gives rise, via (\ref{eq_twistGa}), to a QNS correlation implementing 
a qc-pseudo-isomorphism between $\tilde{\cl U}_{G_1}$ and $\tilde{\cl U}_{G_2}$.
The proof of Theorem \ref{th_qcalgqc} implies that 
\begin{equation}\label{pu}
(P_{\tilde{\mathcal U}_2}^{\perp}\otimes 1)U_{2,3}\bar U_{1,3}(P_{\tilde{\mathcal U}_1}\otimes 1) = 0 \text{ and } (P_{\tilde{\mathcal U}_1}^{\perp}\otimes 1) U_{2,3}^*\bar U_{1,3}^*(P_{\tilde{\mathcal U}_2}\otimes 1) = 0;
\end{equation}
reversing the arguments in its proof, %of Theorem \ref{th_qcalgqc}, 
we obtain the equivalent conditions
\begin{equation}\label{su}
U(S_1'\otimes 1)U^*\subseteq S_2'\otimes\cl N \text{ and } U^*(S_2'\otimes 1)U\subseteq S_1'\otimes\cl N.
\end{equation}

Note that the map $\alpha_U : \cl B(L^2(M_X))\to  \cl B(L^2(M_X))\otimes\cl N$, 
given by 
$\alpha_U(x) = U(x\otimes 1)U^*$, is trace preserving, that is, satisfies the identities 
$(\tr\otimes \id)(\alpha_U(x))=\tr(x)I$, $x\in \cl B(L^2(M_X))$. Indeed, 
for $i,j\in X\times X$, we have 
\begin{eqnarray*}
(\tr\otimes \id)(\alpha_U(\epsilon_{i,j})))
& = & 
(\tr\otimes \id)((u_{k,i}u_{l,j}^*)_{k,l})
=
\frac{1}{n}\sum_{k=1}^n u_{k,i}u_{k,j}^*\\
& = & 
\frac{1}{n}\delta_{i,j}I
=
\tr(\epsilon_{i,j})I.
\end{eqnarray*}
Assume that there exists a $*$-homomorphism $\rho: M_X\to M_X\otimes \cl N$, such that 
$U(\Lambda(b)\otimes\xi)=(\pi\otimes \text{id})(\rho(b))(\Lambda(1)\otimes\xi)$. 
According to \cite[Section 9.1]{daws} (see also (\ref{p_rho})), 
$$\alpha_U(\pi(a)) = (\pi\otimes \text{id}) (\rho(a)) \subseteq \pi(M_X)\otimes\cl N, \ \ \ a\in M_X;$$
we call $U$ the unitary implementation of $\rho$. Writing $\rho(f_i)=\sum f_j\otimes v_{j,i}$, we have 
$$U(\Lambda(f_i)\otimes\xi) = \sum_{j=1}^{n^2} \Lambda(f_j)\otimes v_{j,i}\xi, \ \ \ i = 1,2,\dots,n^2,$$
and hence $v_{i,j}=u_{i,j}$ for all $i,j$. 
The elements $u_{i,j}$ satisfy all of the relations of the generators of $\cl O(G_1, G_2)$, except for, possibly, 
relation \eqref{eq_rhoint} (equivalently, \eqref{eq_rhoint2}). The following theorem establishes this last relation.

\begin{theorem} \label{converse}
Let  $G_r = (M_X,\tr, A_r)$ be an algebraic quantum graph, $r = 1,2$.
Let $\cl N$ be a tracial von Neumann algebra and 
$U$ be a bi-unitary with entries in $\cl N$
giving rise, via (\ref{eq_twistGa}), to a QNS correlation $\Gamma$ that implements a 
qc-pseudo-isomorphism between $\tilde{\cl U}_{G_1}$ and $\tilde{\cl U}_{G_2}$.
Assume that $U$ is the unitary implementation of a trace-preserving 
$*$-homomorphism $\rho : M_X\to M_X\otimes\cl N$.
Then 
$U(A_1\otimes I) = (A_2\otimes I)U$ and hence $G_1\simeq_{\rm qc} G_2$.
\end{theorem}

The proof of Theorem \ref{converse} uses arguments from 
\cite{daws} and some auxiliary statements
which we now establish. 
Set $H = L^2(M_X)$ (equipped with the inner product associated with ${\rm tr}$).
We identify $L^2(M_X^{\rm op})$ with $L^2(M_X)^{\rm d}$
via the unitary map $\Lambda^{\rm op}(x)\mapsto\overline{\Lambda(x^*)}$. 

We write $\tilde\Lambda: \cl B(H)\to L^2(\cl B(H))$ for the GNS-map corresponding to 
(non-normalised trace) $\Tr$. 
We have 
\begin{eqnarray*}
&&\left\langle\tilde\Lambda(\Theta_{\Lambda(x),\Lambda(y)}),\tilde\Lambda(\Theta_{\Lambda(x'),\Lambda(y')})
\right\rangle
= \Tr\left(\Theta_{\Lambda(y),\Lambda(x)}\Theta_{\Lambda(x'),\Lambda(y')}\right)\\
&&
= \left\langle \Lambda(y),\Lambda(y')\rangle_H\langle \Lambda(x'),\Lambda(x)\right\rangle_H
= \left\langle \Lambda(y),\Lambda(y')\right\rangle_H 
\left\langle \overline{\Lambda(x)},\overline{\Lambda(x')}\right\rangle_{H^{\rm d}}.
\end{eqnarray*}
Hence the linear map $\omega: L^2(\cl B(H))\to H^{\rm d}\otimes H$, defined by 
$$\omega\left(\tilde\Lambda(\Theta_{\Lambda(x),\Lambda(y)})\right) = \overline{\Lambda(x)}\otimes\Lambda(y),$$
is a unitary operator. 

We now fix algebraic quantum graphs,  $G_r = (M_X,\tr, A_r)$, $r = 1,2$,
a von Neumann algebra $\cl N$ and a bi-unitary $U$ as in the statement of Theorem \ref{converse}. 
Assume that $\cl N$ acts on a Hilbert space $K$.
Let $e_r\in M_X^{\rm op}\otimes M_X$ be the projection associated with the 
adjacency matrix $A_r : M_X\to M_X$ of $G_r$ via (\ref{eq_idfore}), $r = 1,2$ 
(see the paragraph after the proof of Theorem \ref{th_qcalgqc}), and
let $p_r$ be the orthogonal projections from the Hilbert space 
$L^2(\cl B(H))$ (equipped with the inner product corresponding to $\Tr$) onto its subspace $\tilde\Lambda(S_r')$, 
$r = 1,2$. 
The following lemma specialises \cite[Lemma 9.17]{daws}; 
we include a direct proof for the convenience of the reader.

\begin{lemma}\label{l_daws3}
The following hold:
\begin{itemize}
    \item[(i)] $(\omega\circ\tilde\Lambda)(A_r) = (\Lambda^{\rm op}\otimes\Lambda)(e_r)$;
    \item[(ii)] $\omega p_r\omega^* = (\bar J\otimes J)e_r(\bar J\otimes J)$;
    \item[(iii)] $p_r\omega^*(\overline{\Lambda(1)}\otimes\Lambda(1)) = \tilde\Lambda(A_r)$. 
\end{itemize}
\end{lemma}

\begin{proof}
(i) 
Let $T=\Theta_{\Lambda(x),\Lambda(y)}$, $x$, $y\in M_X$. Then
$$\omega(\tilde\Lambda(T))=\overline{\Lambda(x)}\otimes\Lambda(y)=(\Lambda^{\rm op}\otimes\Lambda)(x^*\otimes y).$$
As $A_r = n \sum_{i,j=1}^{n} \Theta_{\Lambda(\epsilon_{i,j}), \Lambda(A_r(\epsilon_{i,j}))}$, we have 
\begin{eqnarray*}
(\omega\circ \tilde\Lambda) (A_r)
& = &
n\sum_{i,j=1}^{n} (\Lambda^{\rm op}\otimes\Lambda)(\epsilon_{i,j}^*\otimes A_r(\epsilon_{i,j}))\\
& = &
n\sum_{i,j=1}^{n} (\Lambda^{\rm op}\otimes\Lambda)(\epsilon_{j,i}\otimes A_r(\epsilon_{i,j}))
=(\Lambda^{\rm op}\otimes\Lambda)(e_r).
\end{eqnarray*}

(ii) 
Using (\ref{eq_LL}), for $a, b\in M_X$ we have 
\begin{eqnarray*}
\omega(\tilde\Lambda(aTb))&=&\omega(\tilde\Lambda(\Theta_{\Lambda (b^*x),\Lambda(ay)})=\overline{\Lambda(b^*x)}\otimes\Lambda(ay)
\\&=&\Lambda^{\rm op}(x^*b)\otimes\Lambda(ay)=(b\otimes a)(\Lambda^{\rm op}(x^*)\otimes\Lambda(y)),
\end{eqnarray*}
where the latter action is that of $M_X^{\rm op}\otimes M_X$ on $L^2(M_X^{\rm op})\otimes L^2(M_X)$. 
Thus 
\begin{equation}\label{eq_omsr'}
\omega(\tilde\Lambda(S_r')) = (M_X^{\rm op}\otimes M_X)(\Lambda^{\rm op}\otimes\Lambda)(e_r). 
\end{equation}

As $e_r\in M_X^{\rm op}\otimes M_X$, identifying it with its image under 
the map $\pi^{\rm op}\otimes\pi$ (which acts on 
$L^2(M_X)^{\rm d}\otimes L^2(M_X)$), 
for $a\otimes b,x\otimes y\in M_X^{\rm op}\otimes M_X$, we obtain that 
\begin{eqnarray}\label{eq_04}
&&
\hspace{0.4cm} (\overline{J}\otimes J)(x\otimes y)(\overline{J}\otimes J)((\Lambda^{\rm op}\otimes\Lambda)(a\otimes b))\\
&& 
= 
(\overline{J}\otimes J)(x\otimes y)(\overline{J\Lambda(a^*)}\otimes\Lambda(b^*))
=
(\overline{J}\otimes J)(x\otimes y)(\Lambda^{\rm op}(a^*)\otimes\Lambda(b^*)) \nonumber\\
&&
= (\overline{J}\otimes J)(\Lambda^{\rm op}(a^*x)\otimes\Lambda(yb^*))
= \overline{J\Lambda(x^*a)}\otimes\Lambda(by^*) \nonumber\\
&& 
=\Lambda^{\rm op}(x^*a)\otimes\Lambda(by^*)
=(a\otimes b)(\Lambda^{\rm op}\otimes\Lambda)(x^*\otimes y^*)\nonumber
\end{eqnarray}
which, together with the fact that $e_r$ is self-adjoint (see Lemma \ref{idempotent}),
implies that, for any $u\in M_X^{\rm op}\otimes M_X$, we have
\begin{equation}\label{barjerj}
(\overline{J}\otimes J)e_r(\overline{J}\otimes J)((\Lambda^{\rm op}\otimes\Lambda)(u))=u(\Lambda^{\rm op}\otimes\Lambda)(e_r).
\end{equation}
In particular, using (\ref{eq_omsr'}), 
$${\rm ran}\left((\overline{J}\otimes J)e_r(\overline{J}\otimes J)\right) = (M_X^{\rm op}\otimes M_X) (\Lambda^{\rm  op}\otimes\Lambda)(e_r) = (\omega\circ\tilde\Lambda)(S_r').$$ 
Statement (ii) now follows. 

(iii) Using (i), (ii), \eqref{barjerj} and the calculation (\ref{eq_04}) for $a = b = 1$, 
we have 
\begin{eqnarray*}
(\omega p_r\omega^*)(\overline{\Lambda(1)}\otimes\Lambda(1))
& = &
(\overline{J}\otimes J)e_r(\overline{J}\otimes J)((\overline{\Lambda(1)}\otimes\Lambda(1))\\
& = & 
(\Lambda^{\rm op}\otimes\Lambda)(e_r)
= (\omega\circ\tilde\Lambda)(A_r).
\end{eqnarray*}
\end{proof}

Let $\tilde U: L^2(\cl B(H))\otimes K \to  L^2(\cl B(H))\otimes K$ 
be the operator, given by 
$$\tilde U(\tilde\Lambda(b)\otimes\xi) = \alpha_U(b)(\tilde\Lambda(1)\otimes\xi), \ \ \ \xi \in K.$$  
For a Hilbert space $L$, let $j : L \to L^{\rm d}$ the anti-linear isomorphism, given by 
$j(g) = \overline{g}$, and 
$R : \cl B(L)\to \cl B(L^{\rm d})$ be the map, given by 
$R(x) = jx^*j$, $x\in \cl B(L)$.   
Note that, if $(g_i)_i$ is an orthonormal basis for $L$, 
$\epsilon_{i,j}\in \cl B(L)$ are the matrix units corresponding to $(g_i)_i$, 
and $\{\bar{\epsilon}_{j,i}\}$ is the matrix unit system for $\cl B (L^{\rm d})$
with respect to the orthonormal basis $(\bar g_i)_i$, then 
$$R(\epsilon_{i,j}) = j(g_ig_j^*)^* j = j(g_jg_i^*) j
= j(g_j) j(g_i)^* = \bar{\epsilon}_{j,i}.$$
In the following, we let $V = (R \otimes 1)(U^*)$.  Thus, if $U = (u_{i,j})_{i,j=1}^{n^2}$
with respect to the orthonormal basis $\left\{\Lambda(f_i)\right\}_{i=1}^{n^2}$ of $L^2(M_X)$, then $V$ 
is the operator on $L^2(M_{X})^{\rm d}\otimes K$ whose matrix with respect 
to the orthonormal basis $\left\{\overline{\Lambda(f_i)}\right\}_{i=1}^{n^2}$ is 
$(v_{i,j})_{i,j =1}^{n^2} := (u_{i,j}^*)_{i,j = 1}^{n^2}$.

\begin{lemma}\label{l_UV13}
We have that 
$(\omega\otimes 1)\tilde U(\omega^*\otimes 1)
=U_{2,3}V_{1,3}$.
\end{lemma}

\begin{proof}
For $1\leq s,t\leq n^2$ and $\xi\in K$ we have 
\begin{eqnarray*}
&&
(\omega\otimes 1)\tilde U(\omega^*\otimes 1)
(\overline{\Lambda(f_t)}\otimes\Lambda(f_s)\otimes\xi)\\
& = & 
(\omega\otimes 1)\tilde U(\tilde\Lambda(\Theta_{\Lambda(f_t),\Lambda(f_s)})\otimes\xi)\\
& = &
(\omega\otimes 1)\alpha_U(\Theta_{\Lambda(f_t),\Lambda(f_s)})(\tilde\Lambda(1)\otimes\xi)\\
& = &
(\omega\otimes 1)U(\Theta_{\Lambda(f_t),\Lambda(f_s)}\otimes 1)U^*(\tilde\Lambda(1)\otimes\xi)\\
& = &
(\omega\hspace{-0.03cm}\otimes \hspace{-0.03cm}1)
\hspace{-0.15cm}\left(\hspace{-0.05cm}\sum_{i,j=1}^{n^2} \hspace{-0.05cm}\epsilon_{i,j}\hspace{-0.07cm}\otimes \hspace{-0.07cm}u_{i,j}\hspace{-0.1cm}\right)
\hspace{-0.15cm}(\hspace{-0.02cm}\Theta_{\Lambda(f_t),\Lambda(f_s)} \hspace{-0.07cm}\otimes \hspace{-0.07cm}1\hspace{-0.02cm})\hspace{-0.15cm}
\left(\hspace{-0.05cm}\sum_{k,l=1}^{n^2} \hspace{-0.05cm}\epsilon_{l,k} \hspace{-0.07cm}\otimes \hspace{-0.07cm}u_{k,l}^*\hspace{-0.1cm}\right)
(\tilde\Lambda(1)\hspace{-0.07cm}\otimes\hspace{-0.07cm}\xi)\\
& = &
(\omega\otimes 1)\left(\sum_{i,j,k,l = 1}^{n^2}
\epsilon_{i,j}\Theta_{\Lambda(f_t),\Lambda(f_s)}\epsilon_{k,l}\otimes u_{i,j}u_{l,k}^*\right)(\tilde\Lambda(1)\otimes\xi)\\
& = &
\sum_{i,l = 1}^{n^2}\overline{\Lambda(f_l)}\otimes\Lambda(f_i)\otimes u_{i,s}u_{l,t}^*\xi
= 
U_{2,3}V_{1,3}(\overline{\Lambda(f_t)}\otimes\Lambda(f_s)\otimes\xi)
\end{eqnarray*}
The statement follows by linearity.
\end{proof}

\begin{proof}[Proof of Theorem  \ref{converse}] 
We recall that the von Neumann algebra $\cl N$ acts on the Hilbert space $K$, and that  
$p_r$ is the orthogonal projections from the Hilbert space $L^2(\cl B(H))$
(equipped with the inner product coming from $\Tr$) onto $\tilde\Lambda(S_r')$. 
By (\ref{su}), 
\begin{equation}\label{eq_withti}
\tilde U(\tilde\Lambda(b)\otimes\xi) = \alpha_U(b)(\tilde\Lambda(1)\otimes\xi)\in \tilde\Lambda(S_2')\otimes \cl N\xi, 
\ \ \ b\in S_1',
\end{equation} 
and hence 
$$\tilde U(p_1\otimes 1)=(p_2\otimes 1)\tilde U(p_1\otimes 1).$$ Similarly, $\tilde U^*(p_2\otimes 1)=(p_1\otimes 1)\tilde U^*(p_2\otimes 1)$, from which we get $$(p_2\otimes 1)\tilde U=\tilde U(p_1\otimes 1).$$ 
Using Lemmas \ref{l_daws3} and \ref{l_UV13}, for $\xi\in \cl B(K)$ we therefore have 
\begin{eqnarray*}
&&\alpha_U(A_1)(\tilde\Lambda(1)\otimes\xi)=\tilde U(\tilde\Lambda(A_1)\otimes\xi)\\&&=\tilde U(p_1\omega^*(\overline{\Lambda(1)}\otimes\Lambda(1))\otimes\xi)=(p_2\otimes 1)\tilde U(\omega^*(\overline{\Lambda(1)}\otimes\Lambda(1))\otimes\xi))\\
&&
= 
(p_2\omega^*\otimes 1)U_{2,3}V_{1,3}(\overline{\Lambda(1)}\otimes\Lambda(1)\otimes\xi).
\end{eqnarray*}

From the definition of $U$, we have 
\begin{equation} \label{fix1}
U(\Lambda(1)\otimes\xi)=\rho(1)(\Lambda(1)\otimes\xi) = \Lambda(1)\otimes\xi.
\end{equation}
Observe that 
\begin{align} \label{V-U-relation}
V = (F^{-1} \otimes 1)U(F \otimes 1)   
\end{align}
where $F : H^{\rm d} \to H$ is the unitary given by $F\overline{{\Lambda(x)}} = \Lambda(x^*)$.  
Indeed, to establish \eqref{V-U-relation}, we note that 
$F^{-1} = F^*$ and, for $\xi \in K$ and $i = 1,\dots,n^2$, we compute:
\begin{align*}
    (F^{-1} \otimes 1)U(F \otimes 1)(\overline{\Lambda(f_i)}\otimes\xi)&=(F^{-1} \otimes 1)U(\Lambda(f_i^*)\otimes\xi) \\
    &=(F^{-1} \otimes 1)(\pi\otimes \text{id})(\rho(f_i^*))(\Lambda(1)\otimes\xi) \\ 
    &= (F^{-1} \otimes 1)(\pi\otimes \text{id})\Big(\sum_{j =1}^{n^2} f_j^* \otimes u_{j,i}^*\Big)(\Lambda(1)\otimes\xi) \\
    &=(F^{-1} \otimes 1)\Big(\sum_{j =1}^{n^2} \Lambda(f_j^*) \otimes u_{j,i}^*\xi\Big)\\
    &=\sum_{j =1}^{n^2} \overline{\Lambda(f_j)} \otimes u_{j,i}^*\xi 
    = V(\overline{\Lambda(f_i)} \otimes \xi).
\end{align*}
By \eqref{V-U-relation}, and the 
identities $F(\overline{\Lambda(1)}) = \Lambda(1)$ and $F^{-1} (\Lambda(1)) =  \overline{\Lambda(1)}$, 
we have 
\begin{equation} \label{fix2}
V(\overline{\Lambda(1)}\otimes\xi)= \overline{\Lambda(1)}\otimes\xi.
\end{equation}
Using (\ref{eq_withti}), \eqref{fix1}, \eqref{fix2}, and Lemmas \ref{l_daws3} and  \ref{l_UV13}, we finally obtain
\begin{eqnarray*}
\alpha_U(A_1)(\tilde\Lambda(1)\otimes\xi)
& = & 
\tilde U(\tilde\Lambda(A_1)\otimes\xi)\\
& = & 
(p_2\omega^*\otimes 1)U_{2,3}V_{1,3}(\overline{\Lambda(1)}\otimes\Lambda(1)\otimes\xi) \\
& = &
(p_2\omega^* \otimes 1)(\overline{\Lambda(1)}\otimes\Lambda(1)\otimes\xi)\\
& = &
\tilde\Lambda(A_2)\otimes\xi = (A_2\otimes I)(\tilde\Lambda(1)\otimes\xi),
\end{eqnarray*}
where we consider $A_2$ in the left regular representation of $\cl B(L^2(M_X))$, that is, 
as an operator on $L^2(\cl B(L^2(M_X)))$. 
Thus, 
$$(\alpha_U(A_1) - A_2\otimes I)(\tilde\Lambda(1)\otimes\xi) = 0, \ \ \ \xi\in K.$$
This implies that 
$$(\id\otimes L_{\xi\eta^*})((\alpha_U(A_1) - A_2\otimes I))\tilde\Lambda(1) = 0, \ \ \ \xi,\eta\in K;$$
thus, 
$$(\id\otimes L_{\xi\eta^*})((\alpha_U(A_1) - A_2\otimes I)) = 0, \ \ \ \xi,\eta\in K.$$
Hence $\alpha_U(A_1) = A_2 \otimes I$ which, in turn, 
means that $U(A_1\otimes I) = (A_2\otimes I)U$. The proof is complete. 
\end{proof}

\end{document}